\documentclass{amsart}
\usepackage{amsmath,amssymb,amscd}
\usepackage{hyperref}
\usepackage{amsthm}

\usepackage{mathrsfs}
\usepackage{accents}
\usepackage{url}
\usepackage{xcolor}

\newcommand{\CC}{\mathbb{C}}
\newcommand{\QQ}{\mathbb{Q}}

\newcommand{\kld}{$k$-long decomposition}
\newcommand{\klds}{$k$-long decompositions}
\newcommand{\kspl}{$k$-skew periodic sequence of linear polynomials\textsf{-linear}}
\newcommand{\kspls}{$k$-skew periodic sequences of linear polynomials\textsf{-linear}}

\newcommand{\ksca}{k\textsf{-scalar}}
\newcommand{\kleq}{\mathscr{D}_k}
\newcommand{\Affk}{\operatorname{ASym}_k}
\newcommand{\Symk}{\operatorname{Sym}_k}
\newcommand{\STk}{\operatorname{ST}_k}
\newcommand{\STpk}{\operatorname{ST}^+_k}
\newcommand{\id}{\operatorname{id}}
\newcommand{\trace}{\operatorname{trace}}

\newcommand{\NN}{\mathbb{N}}

\theoremstyle{theorem}
\newtheorem{thm}{Theorem}[section]
\newtheorem{lem}[thm]{Lemma}
\newtheorem{prop}[thm]{Proposition}
\newtheorem{remark}[thm]{Remark}
\newtheorem{cor}[thm]{Corollary}

\newtheorem{fact}[thm]{Fact}
\newtheorem{claim*}{Claim}[thm]
\newtheorem{claim}{Claim}[thm]

\theoremstyle{definition}
\newtheorem{Def}[thm]{Definition}
\newtheorem{nota}[thm]{Notation}
\newtheorem{conv}[thm]{Convention}

\theoremstyle{remark}
\newtheorem{Rk}[thm]{Remark}

\newenvironment{pfc}{\noindent {\emph{Proof of Claim:}}}{\hspace{\fill} $\maltese$ \vspace{.1in}}

\newcommand{\ZZ}{\mathbb{Z}}

\newcommand{\reduct}{\overrightarrow}
\newcommand{\congru}{\overline}

\newcommand{\indeg}{\operatorname{in-deg}}
\newcommand{\outdeg}{\operatorname{out-deg}}

\newcommand{\ins}{\operatorname{in}}
\newcommand{\outs}{\operatorname{out}}
\newcommand{\pud}{\operatorname{pud}}

\begin{document}

\title[Skew-invariance and Mahler functions]{Skew-invariant curves and the algebraic independence of
Mahler functions}

\author{Alice Medvedev}
\email{medvedev.math.ccny@gmail.com}
\address{Signal Fox \\
175 Varick St., 2nd Floor \\
New York, NY 10014 \\
USA}

\author{Khoa D.~Nguyen}
\email{dangkhoa@ucalgary.ca}
\address{University of Calgary. MS 542 \\
2500 University Drive NW \\
Calgary, AB T2N 4T4 \\
Canada}

\author{Thomas Scanlon }
\email{scanlon@math.berkeley.edu}
\address{University of California, Berkeley \\
Department of Mathematics \\
Evans Hall \\
Berkeley, CA 94720-3840 \\
USA}

\begin{abstract}
For $p \in \mathbb{Q}_+ \smallsetminus \{ 1 \}$
a positive rational number different from one, we
say that the Puisseux series
$f \in \mathbb{C}((t))^\text{alg}$ is
$p$-Mahler of non-exceptional polynomial type if
there is a polynomial $P \in \CC(t)^\text{alg}[X]$
of degree at least two which is not conjugate to either a
monomial or to plus or minus a Chebyshev polynomial for
which the equation $f(t^p) = P(f(t))$ holds.
We show that if $p$ and $q$ are multiplicatively
independent and $f$ and $g$ are $p$-Mahler
and $q$-Mahler, respectively, of non-exceptional
polynomial type, then $f$ and $g$ are algebraically
independent over $\mathbb{C}(t)$.  This theorem is proven 
as a consequence of a more general theorem that if $f$ is 
$p$-Mahler of non-exceptional polynomial type, and 
$g_1, \ldots, g_n$ each satisfy some difference equation 
with respect to the substitution $t \mapsto t^q$, then 
$f$ is algebraically independent from 
$g_1, \ldots, g_n$.  These theorems are themselves consequences
of a refined classification of skew-invariant curves for 
split polynomial dynamical systems on $\mathbb{A}^2$.    	
\end{abstract}

\maketitle

\section{Introduction}

Mahler introduced the class of functions which now bear his name in~\cite{Mahler-first} with the
aim of systematically deducing transcendence, and more generally algebraic independence, of special
values of certain transcendental functions satisfying functional equations.  Specifically,
for a positive rational number $q$ not equal to $1$  we say that the formal Puissuex series
$f \in \CC((t))^\text{alg}$ is $q$-Mahler of polynomial type if there is a polynomial
$Q(X) \in \CC(t)^\text{alg}[X]$ of degree at least two for which $f(t^q) = Q(f(t))$.\footnote{Mahler
usually considers the case that $q$ is the reciprocal of a positive integer.  The greater
generality adds no complexity to our arguments and permits us to make suitable reductions throughout
the course of our proof.  Functions satisfying other functional equations appear in the literature under the
name of Mahler functions.  Notably, often, as in~\cite{Nishioka-book}, $Q$ is allowed to be a
rational function in which case we might say that $f$ is $q$-Mahler of rational
type. In other works (see, for instance~\cite{AB,ScSi}) $q$-Mahler functions are those $f$ for
which $f(t^{q^j})$ (for $j \in \mathbb{N}$) are linearly dependent over
$\mathbb{C}(t)$.  In this case, we might say that $f$ is $q$-Mahler
of linear type.  }

We address the problem of determining the possible algebraic relations between
$p$-Mahler and $q$-Mahler functions for multiplicatively independent $p$ and $q$.
Versions of this problem have been studied.  For example, Zannier shows in~\cite{Za99}
that a transcendental function cannot be simultaneously $p$-Mahler and $q$-Mahler of
polynomial type when $p$ and $q$ are multiplicatively independent integers and the
polynomials have coefficients in $\mathbb{C}[t]$.  Adamczewski and Bell~\cite{AB} prove
analogous results when ``polynomial type'' is replaced by ``linear type''.
Our main theorem is that
if $p$ and $q$ are multiplicatively independent
positive rational numbers and $f$ and $g$ are
$p$-Mahler and $q$-Mahler, respectively, of
\emph{non-exceptional} polynomial type, then
$f$ and $g$ are algebraically independent over
$\CC(t)$.

To state this theorem precisely we should recall
the notion of exceptional polynomials.

There are three basic kinds of exceptional
polynomials: linear polynomials, monomials, and
(scalings of) Chebyshev polynomials.  \footnote{
For each positive integer $N$ there is a
unique monic polynomial $C_N$ which
satisfies the functional equation
$C_N(X + \frac{1}{X}) = X^N + \frac{1}{X^N}$. 
For us, a Chebyshev polynomial is a polynomial of
the form $C_N$ for some $N \geq 2$.} 
A polynomial $P$ is exceptional if it is linear or
$P$ is conjugate~\footnote{Usually, for two polynomials $P$ and $Q$ we say that
$P$ and $Q$ are conjugate if there is a linear
polynomial $L$ with $P \circ L = L \circ Q$.
For the questions treated in this paper we require the notion of \emph{skew-conjugacy} relative to a field automorphism $\sigma$. The polynomials $P$ and $Q$ are skew-conjugate if there is a linear polynomial $L$ for which $P \circ L = L^\sigma \circ Q$ where if $L(x) = a x + b$, 
then $L^\sigma = \sigma(a) x + \sigma(b)$.  
If $\sigma$ acts trivially, then the usual notion
of conjugacy agrees with skew-conjugacy.   See 
Definition~\ref{def:lin-eq} for details.   } to
$X^N$ or to $\pm C_N$ where $N =
\operatorname{deg}(P)$.  
 
Otherwise,
$P$ is \emph{non-exceptional}.~\footnote{In~\cite{MS}
the non-exceptional polynomials are called ``disintegrated''.}  For
$p \in \QQ_+ \smallsetminus \{ 1 \}$ we say that
$f \in \CC((t))^\text{alg}$ is $p$-Mahler of
non-exceptional polynomial type if there is a
non-exceptional polynomial $P \in \CC(t)^\text{alg}[X]$
for which the equation $f(t^p) = P(f(t))$ holds.
Our main theorem takes the following form.

\begin{thm}
\label{thm:ind-mahler}
Let $p$ and $q$ be two multiplicatively independent
positive rational numbers, $f \in \CC((t))^\text{alg}$
and $g \in \CC((t))^\text{alg}$ be two formal
Puisseux series over the complex numbers, and suppose
that $f$ and $g$ are $p$-Mahler and $q$-Mahler,
respectively, of non-exceptional polynomial type.
Then $f$ and $g$ are
algebraically independent over $\CC(t)$.
\end{thm}

The hypotheses in Theorem~\ref{thm:ind-mahler} are much
stronger than we actually require and the conclusion
is much weaker than what prove.   Indeed, it suffices
to assume merely that $f$ is $p$-Mahler
of non-exceptional polynomial type and that
$g$ satisfies some difference equation with respect
to the substitution $t \mapsto t^q$.  In fact,
our proof Theorem~\ref{thm:ind-mahler} is algebraic
and applies in a much wider context.

The substitutions $t \mapsto t^p$ and $t \mapsto t^q$
induce commuting field automorphisms $\sigma$ and
$\tau$ of the fields of algebraic
functions $\CC(t)^\text{alg}$ and of
formal Laurent series.  Moreover,
the multiplicative independence
of $p$ and $q$ imply that $\sigma$
and $\tau$ are independent in the sense
that if $h$ is a non-constant
formal Puisseux series
and $\left( m, n \right) \neq
\left( 0, 0 \right)$, then
$\sigma^m \tau^n(h) \neq h$.     To say that $f$ is a
$p$-Mahler of non-exceptional polynomial type is
simply to say that there is a non-exceptional
polynomial $P \in \CC(t)^\text{alg}[X]$ with
$\sigma(f) = P(f)$.   Abstracting this situation,
we formulate our more general theorem on
algebraic independence.

\begin{thm}
\label{thm:abs-Mah}
Let $C \subseteq K \subseteq L$ be a tower
of algebraically closed fields of characteristic
zero.   Suppose that $\sigma$ and
$\tau$ are commuting automorphisms of $L$ which
preserve $K$ and have $C$ as their common fixed
field.  Suppose moreover that $\sigma$ and
$\tau$ are independent in the sense that if
$a \in L \smallsetminus C$ and $\left( m, n \right)
\in \mathbb{Z}^2 \smallsetminus \{ \left( 0, 0 \right)
\}$, then $\sigma^m \tau^n (a) \neq a$.

We make two additional hypotheses about difference
equations.  First, for any automorphism $\rho$ in
the group generated by $\sigma$ and $\tau$ and any
$c \in C$, there are no nonconstant solutions to
the equation $\rho(x) = c x$ in $L$.  Secondly,
for $\rho$ and $\mu$ independent automorphisms
in the group generated by $\sigma$ and $\tau$,
$M \in \mathbb{Z}_+$ a positive integer, and
$a \in L^\times$ an element of $L$ satisfying
$\rho(\frac{\mu(a)}{a}) = (\frac{\mu(a)}{a})^M$,
then $a \in K$.

Suppose that $f \in L$ satisfies an equation of
the form $\sigma(f) = P(f)$ for some
non-exceptional polynomial $P \in K[X]$ and that
$g_1, \ldots, g_n \in L$ each satisfy non-trivial
$\tau$-difference equations over $K$.  That is,
for some $r \in \mathbb{N}$
there are non-zero polynomials
$Q_i(X_0, X_1, \ldots, X_r) \in K[X_0, \ldots, X_r]$
so that $Q_i(g_i,\tau(g_i), \ldots, \tau^r(g_i)) = 0$
for $1 \leq i \leq n$.   Then $f$ is algebraically
independent from $g_1, \ldots, g_n$ over $K$. 	
\end{thm}

Recently, Adamczewski, Dreyfus, Hardouin, and 
Wibmer~\cite{ADHW} proved a similar algebraic independence
statement under the hypothesis that $f$ satisfies a 
linear $\sigma$-difference equation.

Our proof of Theorem~\ref{thm:abs-Mah} makes use
of the full strength of the characterization of
skew-invariant curves from~\cite{MS}.  
Most uses of~\cite{MS} to date invoke merely the
characterization of invariant curves for
dynamical systems of the form $\left( x, y \right)
\mapsto \left( P(x), P(y) \right)$ where
$P$.  For example, the second author of the present paper obtained 
in~\cite{Ng15} 
a result on the algebraic independence of the analytic conjugacies associated to
Mahler functions by using just the characterization of invariant varieties.
For invariant varieties, the relevant description may also be
obtained using the theory of orbifolds
as shown by Pakovich~\cite{Pa17} and admits
a generalization to the case where $P$ is a
rational function~\cite{Pa19}.  For the application
to our problem on the algebraic independence of
Mahler functions, we need a precise description
of the possible skew-invariant curves for maps of
the form $\left( x, y \right)
\mapsto \left( P(x), P^\tau(y) \right)$ where
$P$ is a non-exceptional polynomial and
$\tau:K \to K$ is any field automorphism.   The
main theorem of~\cite{MS} describes these
combinatorially and in the
present paper we push that characterization further,
showing in the language of~\cite{MS},
that these skew-invariant curves are
given by skew-twists.

Let us give the requisite definitions and state
the precise form of our theorem on skew-invariant
curves.

A difference field $(K,\sigma)$ is a field $K$
given together with a distinguished field
automorphism $\sigma:K \to K$.   For any
object $X$ defined over $K$, we write $X^\sigma$
for the transform of $X$ over $K$.  For example,
for a polynomial $P \in K[X]$, $P^\sigma$ is the
polynomial obtained by applying $\sigma$ to the
coefficients of $P$.   Given an algebraic variety
$X$ over $K$ and a regular map $f:X \to X^\sigma$,
we say that the subvariety $Y \subseteq X$ is
\emph{skew-invariant} if the restriction of $f$ to
$Y$ maps $Y$ to $Y^\sigma$.

Given two one-variable polynomials $P, Q \in K[X]$,
we say that $Y \subseteq \mathbb{A}^2_K$ is a
skew-twist curve for $(P,Q)$ if there are polynomials $\alpha$
and $\beta$ and an integer $n$ so that 
$P=\beta\circ \alpha$, $Q=\alpha^{\sigma^{n+1}}\circ \beta^{\sigma^n}$, and
\begin{itemize}
\item either $n\geq 0$ and $Y$ is defined by
$y = \alpha^{\sigma^n} \circ P^{\sigma^{n-1}}
\circ \cdots \circ P^\sigma \circ P (x)$

\item or $n\leq -1$ and $Y$ is defined by 
$x = \beta^{\sigma^{-1}}\circ Q^{\sigma^{-n-2}}\circ Q^{\sigma^{-n-3}}\circ\cdots\circ Q(y)$; note that when $n=-1$, this simply mean $x=\beta^{\sigma^{-1}}(y)$.
\end{itemize}

With these definitions in place we may state our
theorem on skew-invariant curves.

\begin{thm}
\label{thm:skew-inv}
Let $(K,\sigma)$ be a difference field of characteristic
zero and $\tau:K \to K$ be a field automorphism commuting with $\sigma$.
Let $P \in K[X]$ be a non-exceptional polynomial.
Then every skew-invariant curve for $(P,P^\tau)$
is either a horizontal or vertical line or a skew-twist curve.	 
\end{thm}

Section~\ref{sec:skew-inv} comprises the bulk of 
this paper and 
is devoted to proving Theorem~\ref{thm:skew-inv}.  
In~\cite{MS} we worked with a formalism of decompositions
and Ritt monoid actions on these decompositions in 
order to encode and manipulate skew-invariant 
curves.  Starting with Subsection~\ref{sec:monoid} we 
modify that formalism by replacing the decompositions
with what we call long decompositions.  The monoid 
actions are then replaced with partial group actions
and once again a $(P,Q)$-skew-invariant curve will be 
encoded by an equation of the form 
$w \star \vec{P} = \vec{Q}$ where $\vec{P}$ and 
$\vec{Q}$ are long decompositions of $P$ and $Q$, 
respectively. When $Q = P^\tau$, then 
decomposition $\vec{Q}$ may be taken to be 
$\vec{P}^\tau$, and then as far as the ``shape'' 
of the long decompositions are concerned, it would 
seem that the action by $w$ does not have any real
effect.  Theorem~\ref{thm:skew-inv} is a mathematical
expression of this observation.  

In practice, we must work to produce enough invariants
of these long decompositions to limit the possibilities
for such a $w$.  One such invariant might be 
what we call a wall, that is, a point in the 
decomposition at which swaps are impossible.  We 
prove early on with Corollary~\ref{cor:warm-up}
that Theorem~\ref{thm:skew-inv} holds when 
a long decomposition of $P$ admits a wall.  

In Subsection~\ref{sec:inout} we study the case in 
which Chebyshev functions do not appear in a 
decomposition of $P$.  In this case, we will 
find a semi-invariant way to break a long 
decomposition of $P$ into blocks and then to describe
the our partial group actions on the decomposition
in terms of certain numerical functions.    
Finally, in Subsection~\ref{sec:clusterings} we expand
the theory of clusterings from~\cite{MS} to 
produce sufficiently robust combinatorial 
invariants in the remaining cases.  
We collect all of the steps required for the proof of 
Theorem~\ref{thm:skew-inv} in 
Subsection~\ref{sec:complete-skew-inv}.
This paper concludes with the proof of 
Theorems~\ref{thm:ind-mahler} and~\ref{thm:abs-Mah}
in Section~\ref{non-exceptional-sect}.

\subsection*{Acknowledgements}
During the writing of this paper A.M.
was partially supported by the NSF grant DMS-1500976 
and the Simons Foundation award 317672, K.N. was partially 
supported by NSERC Discovery Grant RGPIN-2018-03770 and
CRC tier-2 research stipend 950-231716, and T.S. was 
partially supported by NSF grants FRG DMS-1760413 
and DMS-1800492 and a Simons Fellowship in Mathematics. 

\section{Skew-invariant curves}
\label{sec:skew-inv}

There are two natural procedures for producing skew-invariant
curves for $\left( P, Q \right)$ where $P$ and $Q$
are univariate polynomials.  First, we have
the skew-twists.  That is, if we factor
$P = \beta \circ \alpha$ and $Q = \alpha^{\sigma^{n+1}} \circ \beta^{\sigma^n}$
for some natural number $n$ and polynomials $\alpha$ and
$\beta$, then the curve defined by $y = \alpha^{\sigma^n} \circ P^{\sigma^{n-1}}
\circ \cdots \circ P(x)$, which we call a
\emph{skew-twist curve} is $\left( P, Q \right)$-skew-invariant.
Note that here we allow for $\alpha$ or $\beta$ to be linear and $n = 0$,
so that this includes the case that $Q$ is skew-linearly conjugate to $P$.
Another natural way to obtain a skew-invariant curve would be have an
identity of the form $X^n \circ P = Q \circ X^n$ for some positive
integer $n$ in which case the curve defined $y = x^n$, which we call
a \emph{monomial curve}, is
$\left( P, Q \right)$-skew-invariant.  Of course, we may
swap the roles of $P$ and $Q$ to obtain skew-invariant curves for
$\left( Q, P \right)$.  For example, if we have the identity
$P \circ X^n = X^n \circ Q$, then the curve defined by
$y^n = x$ is $\left( P, Q \right)$-skew-invariant.  Likewise,
we may compose such skew-invariant curves to obtain new ones.
That is, if $C$ is $\left( P, Q \right)$-skew-invariant and
$D$ is $\left( Q, R \right)$-skew-invariant, then the
(possibly reducible) $D \circ C$ is $\left( P, R \right)$-skew-invariant.
The process of iterating these operations and possibly taking components of
the resulting reducible curves gives rise to all of the
skew-invariant curves.  However, this fact on its own
does not give a sufficiently clear picture of the class of
skew-invariant curves.  The description is complicated, for
instance, by the fact that the composite of two skew-twists need not be
a skew-twist itself since a given polynomial may admit
several different presentations as a composite of other polynomials.
Moreover, our two classes of basic skew-invariant curves are not
disjoint.   If $P = R \circ X^n$ and $Q = X^n \circ R$ where
$R^\sigma = R$, then the $\left( P, Q \right)$-skew-invariant
curve defined by $y = x^n$ is a skew-twist curve and is also a
monomial curve.

The skew-invariant curves for pairs
$\left( P, Q \right)$ of non-exceptional
polynomials are described in~\cite[Proposition 6.19]{MS} as components of curves encoded
by equations of the form
$w \star \vec{f} = \vec{g}$ where $\vec{f}$
is a decomposition of $f$, $\vec{g}$  is a
decomposition of $g$, and $w$ belongs to the
augmented skew-twist monoid.
With~\cite[Theorem 6.26]{MS}, various results on canonical forms
for such actions are applied to convert the
equation $w \star \vec{f} = \vec{g}$ into a
concise geometric form.  In this section, which is dedicated to the 
proof of Theorem~\ref{thm:skew-inv}, 
we modify the formalism
of decompositions and monoid actions to
work with what we call long decompositions and will
then specialize to describe the skew-invariant
curves appearing in Theorem~\ref{thm:skew-inv}.

\subsection{Long decompositions and partial group actions}
\label{sec:monoid}

In this subsection we develop the formalism of long decompositions 
and partial group actions.  With this work we may prove Theorem~\ref{thm:skew-inv} in 
two cases: when $P$ is indecomposable (Proposition~\ref{prop:warmup-indecomposable}) and then 
when $P$ admits a decomposition with an unswappable factor (Corollary~\ref{cor:warm-up}). 

\subsubsection{Long decompositions}
While we recall some of the definitions and
results from~\cite{MS}, the reader is advised to consult
that text for further details. Let us begin by recalling and introducing our basic
definitions.

\begin{Def}
\label{def:indecomposable}
A polynomial $f \in K[X]$ is \emph{indecomposable} if
$\operatorname{deg}(f) \geq 2$ and whenever
$f = A \circ B$ 	with $A, B \in K[X]$, then either
$A$ or $B$ is linear.   A \emph{decomposition}
of the polynomial $f$ is a finite
sequence $\left( f_k, \ldots, f_1 \right)$
of indecomposable polynomials for which $f =
f_k \circ \cdots \circ f_1$.
A ($k$-)\emph{long decomposition} of
$f$ is a $k$-skew-periodic $\mathbb{Z}$-indexed sequence
$\left( f_i \right)_{i \in \mathbb{Z}}$ of
indecomposable polynomials (meaning that
for every $i \in \mathbb{Z}$ we have
$f_{i+k} = f_i^\sigma$) for which
$\left( f_k, \ldots, f_1 \right)$ is
a decomposition of $f$.
\end{Def}

\begin{nota}
Throughout the rest of this section we will
fix the number $k$.
Generally, when working with a \kld, we shall write
something like $f = \left( f_i \right)_{i \in
\mathbb{Z}}$ and may use the notation $\vec{f}$ when
we wish to distinguish between sequences and
individual polynomials.	
\end{nota}

\begin{Rk}
Every nonlinear polynomial admits a long decomposition.
Part of Ritt's~\cite{Ritt} theorem on polynomial decompositions is
that the number $k$ depends only on the polynomial.
Of course, unless the polynomial is indecomposable,
the long decomposition itself is not an invariant, but the
rest of Ritt's theorem expresses the extent to which
it is and how one decomposition may be obtained from
another.   	
\end{Rk}

\begin{Def}
\label{def:lin-eq}
A \kspl 
is a $\mathbb{Z}$-indexed sequence $L =
\left( L_i \right)_{i \in \mathbb{Z}}$ of linear
polynomials for which $L_{i + k} = L_i^\sigma$ holds
for all $i \in \mathbb{Z}$.

The \kld $f$ is \emph{linearly equivalent} to the
\kld $g$ via the \kspl $L$ provided that
$g_i = L_{i+1}^{-1} \circ f_i \circ L_i$
for all $i \in \mathbb{Z}$.	
We write $\kleq$ for the set of equivalence
classes of \klds.
\end{Def}

\begin{Rk}
On one hand, Definition~\ref{def:lin-eq} describes
an operation which takes as input $f$ and $L$ and
returns $g$.  On the other hand, the relation
``$f$ is linearly equivalent to $g$ via some
\kspl'' is an equivalence relation on the set of
\klds.  This notion matches the \emph{skew}-linear
equivalence of~\cite{MS} rather than linear equivalence.
For $f$ a \kld we write $[f]$ for its equivalence
class relative to linear equivalence.
\end{Rk}

Ritt's theorems on polynomial
decompositions~\cite{Ritt} express the
extent to which a polynomial admits a
unique presentation as a composition of
indecomposable polynomials.
Loosely speaking, if the polynomial
$P$ may be expressed in two different
ways as $P = P_k \circ \cdots \circ P_1$ and
$P = Q_\ell \circ \cdots \circ Q_1$ where each
$P_i$ and $Q_j$ is indecomposable, then
$\ell = k$ and the decomposition
$\left( Q_k, \ldots, Q_1 \right)$ may be
obtained from the decomposition
$\left( P_k, \ldots, P_1 \right)$ by transforming
these sequences through a finite number of what
in~\cite{MS} we call \emph{Ritt swaps}.
We encode this process of passing from
one decomposition to another via a
sequence of Ritt swaps through a certain
partial monoid action and then take this
further to encode all skew-invariant
curves through more general partial
monoid actions.

Let us recall some of the formalism of
Ritt identities and Ritt swaps.
For more details, consult~\cite[page 102]{MS}.

\begin{Def}
\label{def:Ritt-identity}
The following identities are called \emph{basic
Ritt identities.}
\begin{itemize}
\item $x^p \circ x^q = x^q \circ x^p$ for distinct
prime numbers $p$ and $q$,
\item $C_p \circ C_q = C_q \circ C_p$ for distinct
\emph{odd} primes $p$ and $q$ where $C_p$ is the
$p^\text{th}$ Chebyshev polynomial, and
\item $x^k \cdot u(x)^{p} \circ x^p = x^p
\circ x^k \cdot u(x^{p})$  (and the
identity with the left and right sides reversed)
where  $u(0) \neq 0$, $p$ is prime,
and  $x^k \cdot u(x^{p})$ is
indecomposable.
\end{itemize}

\end{Def}

\begin{Def}
\label{def:inoutdeg-1st}
Fix a polynomial $u$ with no initial nor terminal, compositional nor multiplicative monomial factors.
The in-degree and out-degree of a Ritt polynomial $g(x) = x^k u(x^\ell)^n$ are
 $$\operatorname{in-deg}(g) := \ell \mbox{ and } \operatorname{out-deg}(g) := n.$$
 \end{Def}

\begin{Rk}
Of course, for any polynomial $h$, $h \circ h = h \circ
h$.  However, we do not consider such tautological
identities to be basic Ritt identities.
It is also the case that for any prime $p$,
$C_2 \circ C_p = C_p \circ C_2$, but we
do not treat this relation as a basic Ritt
identity.   	
\end{Rk}

\begin{Def}
\label{def:Ritt-poly} 
The polynomials appearing in basic Ritt identities are
called \emph{Ritt polynomials}.  An indecomposable polynomial $P$ is \emph{swappable} if there are
linear polynomials $A$ and $B$ for which
$A \circ P \circ B$ is a Ritt polynomial.  A
swappable polynomial
 $P$ is said to be of type $\mathsf{C}$
if there are linear polynomials $A$ and $B$ for which
$A \circ P \circ B$ is a Chebyshev polynomial of
odd degree.  Otherwise, the swappable polynomial
$P$ is said to be \emph{$\mathsf{C}$-free}.	
\end{Def}

The distinction between type $\mathsf{C}$ and
$\mathsf{C}$-free swappable polynomials and the
choice of defining $C_2$ to be $\mathsf{C}$-free
will become clear in our work on clusterings
in Subsection~\ref{sec:clusterings}.

\begin{Rk}
In~\cite[Definition 2.41]{MS} Ritt polynomials are
required to be monic.  We do not impose that condition
here.  However, when we speak of \emph{monomials}
we always mean expressions of the form $x^n$ for
some $n \geq 2$.  	Another difference between the definition of 
Ritt polynomial here and in~\cite{MS} is that for this paper we require 
Ritt polynomials to be indecomposable, but this is not so in~\cite{MS}.
\end{Rk}

\begin{Def}
\label{def:Ritt-swap}
For $i \in \ZZ$, we consider the symbol $t_i$.
For \klds $f$ and $g$. We say that \emph{$g$ is obtained from $f$ by a Ritt swap at $i$} and write $t_i \star [f] = [g]$ if there are
linear polynomials $L$, $M$, and $N$ such that
$$g_i = S \circ N^{-1}, g_{i+1} = L \circ R, g_i = f_i \,\,\, \text{ for } j \not\equiv i, i+1 \pmod{k}$$
and $(L^{-1} \circ f_{i+1} \circ M) \circ (M^{-1} \circ f_i \circ N) = R \circ S$ is a basic Ritt identity.
\end{Def}

\begin{Def}\label{def:strict-Ritt-swap}
For $\tilde{f}$ and $\tilde{g}$ two \klds,
we say that $\tilde{g}$ is obtained from $\tilde{f}$
via a Ritt swap in the strict sense
at $i$ if $\tilde{f}_{i+1} \circ \tilde{f}_i =
\tilde{g}_{i+1} \circ \tilde{g}_i$ is a basic
Ritt identity and $\tilde{f}_j = \tilde{g}_j$
for $j \not \equiv i
\text{ or } i + 1 \pmod{k}$.  In this case, we
write $t_i \bullet \tilde{f} = \tilde{g}$.
\end{Def}

\begin{Rk}\label{rk:witness-Ritt-swap}
For $i \neq k$, our definition of $t_i \star [f] = [g]$ is exactly \cite[Definition 2.43]{MS}.

If $t_i \bullet \tilde{f} = \tilde{g}$, then taking $L=M=N=\id$ shows that $t_i \star [\tilde{f}] = [\tilde{g}]$.
For $k \neq 2$, the converse also holds: if $t_i \star [f] = [g]$, then there are long decompositions $\tilde{f}$ and $\tilde{g}$
such that $t_i \bullet \tilde{f} = \tilde{g}$. In this case, we say that $(L,\tilde{f},\tilde{g},M)$ is a witness of $t_i \star [f] = [g]$
 when $f$ and $g$ are linearly equivalent to $\tilde{f}$ and $\tilde{g}$ via $L$ and $M$.
\end{Rk}

\begin{Rk}
It follows from skew-periodicity that if the
\klds $f$ and $g$ are Ritt swap
related at $i$, then they are also
Ritt swap related at $j$ for any $j \equiv i
\pmod{k}$.	 Thus, it is natural to index the generators
$t_i$ by $\mathbb{Z}/ k \mathbb{Z}$.  So
for example, when we write an expression like
$t_{k+1}$ we mean $t_1$.   To use the order on the 
indices, we often choose to work with $\{ t_i ~:~ a \leq i < 
a + k \}$ for some $a$, which is not necessarily $0$.  \end{Rk}

\begin{Rk}
If the Ritt swap in Definition~\ref{def:Ritt-swap}
involves an identity of one of the
first two kinds, then it really is a swap in that
$f_{i+1} = g_i$ and $f_i = g_{i+1}$.
However, for an identity of the third kind,
a Ritt swap at $i$ will transform one of
$f_i$ or $f_{i+1}$ into a different polynomial.	
\end{Rk}

We shall consider two other kinds of
operations on $\kleq$.  The symbol
$\phi$ will be used to encode
skew-twist curves.

\begin{Def}
\label{def:phi-action}
For $f = (f_i)_{i \in \mathbb{Z}}$ a
\kld, we define $\phi \bullet f :=
(f_{i+1})_{i \in \mathbb{Z}}$ and
$\phi^{-1} \bullet f := (f_{i-1})_{i \in \mathbb{Z}}$.
This gives an action the symbol $\phi$ on
$\kleq$ defined by $\phi \star [ f ] = [ \phi \bullet f]$ and
of $\phi^{-1}$ via $\phi^{-1} \star [f] = [\phi^{-1} \bullet f]$.
\end{Def}

\begin{Rk}
Unlike the symbols $t_i$, for any $x \in
\kleq$, $\phi \star x$ is defined.  Moreover,
$\phi^{-1} \star = (\phi \star)^{-1}$.	
\end{Rk}

\begin{Rk}
What we call $\phi^{-1}$ corresponds to the
symbol $\beta$ of~\cite{MS}.	
\end{Rk}

Finally, we define partial operations
specifically connected to the third kind
of basic Ritt identity.  These are used to encode
monomial curves.

\begin{Def}
\label{def:epsilon}
For each prime number $p$ we have
symbols $\epsilon_p$ and $\epsilon_p^{-1}$.
If $\tilde{f}$ and $\tilde{g}$ are \klds such
that for all $i \in \mathbb{Z}$,
$\tilde{f}_i = x^{r_i} u_i (x^p)$ where either $u_i(0) \neq 0$ and $\deg u_i > 0$  or $p \neq r_i$ and
$\tilde{g}_i = x^{r_i} u_i(x)^p$, then we write $\epsilon_p \bullet \tilde{f} = \tilde{g}$
and $\epsilon_p^{-1} \bullet g = f$.  This relation
induces a partial action of $\epsilon_p$ and of $\epsilon_p^{-1}$ on $\kleq$
given by $\epsilon_p \star [\tilde{f}] := [\epsilon_p \bullet \tilde{f}]$ and
$\epsilon_p^{-1} \bullet [\tilde{g}] := [\epsilon_p^{-1} \bullet\tilde{g}]$.
If $f$ and $g$ are \klds with $[f] = [\tilde{f}]$ and $[g] = [\tilde{g}]$,
and $L$ and $M$ are \kspls so that $\tilde{f}$ is linearly equivalent to $f$ via $L$ and
$\tilde{g}$ is linearly equivalent to $g$ via $M$, then we say that $(L,\tilde{f},\tilde{g},M)$ is a
witness to $\epsilon_p \star [f] = [g]$.   In this context, $(M^{-1},\tilde{g},\tilde{f},L^{-1})$ is
a witness to $\epsilon_p^{-1} \star [g] = [f]$.
\end{Def}

\begin{Rk}
What we call here $\epsilon_p^{-1}$ corresponds
to the $\delta_p$ of~\cite{MS}.	
\end{Rk}

\begin{Rk}
The condition that $\tilde{f}_i$ is not a monomial of
degree $p$ is included to disambiguate the
monomial curves encoded by identities of the
form $\epsilon_p \star [f] = [g]$ and the skew-twist
curves encoded by relations of the form
$\phi \star [f] = [g]$.  	
\end{Rk}

With this notation in place, we may prove an easy case of Theorem~\ref{thm:skew-inv}.

\begin{prop}
\label{prop:warmup-indecomposable}
Theorem~\ref{thm:skew-inv} holds when $P$ is indecomposable.	
\end{prop}

Our proof of Proposition~\ref{prop:warmup-indecomposable} passes through several lemmas.

\begin{lem}\label{lem:data QABA_0B_0}
Let $P$ be an indecomposable non-exceptional polynomial and let $\mathcal{C}$ be a skew-invariant curve for 
$(P,P^{\tau})$. Then there exist non-constant polynomials $Q,A,B,A_0,B_0$ in $K[X]$ and integers $a,b,c,d$
with $a,b\geq 0$ and $c,d\geq -1$ satisfying the following properties:
\begin{itemize}
	\item [(i)] $\mathcal{C}$ is parametrized by $(A,B)$, i.e. the map $\mathbb{A}^1 \to \mathbb{A}^2$ given by 
	$t \mapsto (A(t),B(t))$ is dominant onto  $\mathcal{C}$. 
		\item [(ii)] $A$ and $B$ have no shared nonlinear initial compositional factors.
	\item [(iii)] $P\circ A=A^{\sigma}\circ Q$ and $P^{\tau}\circ B=B^{\sigma}\circ Q$.
	\item [(iv)] $A = P^{\sigma^{-1}} \circ  \cdots \circ P^{\sigma^{-a}} \circ A_0 \circ Q^{\sigma^c} \circ \cdots \circ Q$
and 
$B = P^{\tau \sigma^{-1}} \circ \cdots \circ P^{\tau \sigma^{-b}} \circ B_0 \circ Q^{\sigma^d} \circ \cdots \circ Q$.
	\item [(v)] $Q^{\sigma^{c+1}}$ is not an initial 
compositional factor of $A_0$, $P^{\sigma^{-a-1}}$
is not a terminal compositional factor of $A_0$, 
$Q^{\sigma^{d+1}}$ is not an initial 
compositional factor of $B_0$, and
 $P^{\tau \sigma^{-b-1}}$
is not a terminal compositional factor of $B_0$.

	\item [(vi)] $P^{\sigma^{-a}} \circ A_0 = A_0^\sigma \circ Q^{\sigma^{c+1}}$ and $P^{\tau \sigma^{-b}} \circ B_0 = B_0^\sigma	 \circ Q^{\sigma^{d+1}}$.
	
	\item [(vii)]   $A=A_0^{\sigma^a}\circ Q^{\sigma^{a+c}}\circ\cdots Q=P^{\sigma^{-1}}\circ\cdots\circ P^{\sigma^{-(a+c+1)}}\circ A_0^{\sigma^{-(c+1)}}$ and $B=B_0^{\sigma^b}\circ Q^{\sigma^{b+d}}\circ\cdots Q=P^{\tau\sigma^{-1}}\circ\cdots\circ P^{\tau\sigma^{-(b+d+1)}}\circ B_0^{\sigma^{-(d+1)}}$. Consequently, $A_0^{\sigma^{-(c+1)}}$ and $B_0^{\sigma^{-(d+1)}}$ do not have a nontrivial common initial compositional factor and either $(a,c)$ or $(b,d)$ is $(0,-1)$.
\end{itemize} 
\end{lem}

\begin{proof}
Let $d:=\deg(P)$. If the lemma holds for the given data $(P,\mathcal{C})$ then it holds for $(L^{\sigma}\circ P\circ L^{-1},(L,L)(C))$ where $L$ is any linear polynomial. Hence we may assume that $P$ is 
centered meaning that the coefficient of $X^{d-1}$ is $0$.

From~\cite[Corollary 2.35]{MS}, we know that $\mathcal{C}$ is parametrized by $(A,B)$ and there exists a polynomial $Q$ such that (iii) 
holds 
since $\mathcal{C}$ is skew-invariant under $(P,P^{\tau})$. By choosing a pair $(A,B)$ with minimal
$\deg(A)+\deg(B)$, we may assume that they do not have a common nonlinear initial compositional factors.

We now express $A$ and $B$ as in (iv) with maximal $a,b,c,d$ then (v) holds. Now we  establish Property (vi). From (iii) and (iv), we have:
\begin{equation}\label{eq:(iii) and (iv)}
P\circ P^{\sigma^{-1}}\circ\cdots\circ P^{\sigma^{-a}}\circ A_0=P\circ P^{\sigma^{-1}}\circ\cdots\circ P^{\sigma^{-(a-1)}}\circ A_0\circ Q^{\sigma^{c+1}}.
\end{equation}
When $a=0$, this identity means $P\circ A_0=A_0\circ Q^{\sigma^{c+1}}$. From now on, assume $a>0$.

First, assume the case when $P$ does not have the form $P_1\circ X^m$ where $m>1$. Since $P$ is centered, the only linear polynomial $L$ such that $P\circ L=P$ is the identity polynomial $L(X)=X$. Consequently, whenever
$P\circ R=P\circ S$ for non-constant polynomials $R,S$ then we have $R=S$. Applying this observation repeatedly to
\eqref{eq:(iii) and (iv)}, we have:
$$P^{\sigma^{-a}}\circ A_0=A_0\circ Q^{\sigma^{c+1}}.$$

Now suppose that $P=P_1(X^m)$ where $m>1$. Since $P$ is indecomposable, we must have that $m=d$ is a prime and $P_1$ is linear. The polynomial $P\circ P^{\sigma^{-1}}\circ\cdots\circ P^{\sigma^{-(a-1)}}$
is centered and we write
\begin{equation}\label{eq:P2 circ X^ell}
P\circ P^{\sigma^{-1}}\circ\cdots\circ P^{\sigma^{-(a-1)}}=P_2\circ X^{\ell}
\end{equation}
where $\ell$ is largest possible. We claim that $\ell=d$. Suppose $\ell>d$ and we arrive at a contradiction. Since $d$ is a prime, $\ell$ is a power of $d$ and hence $d^2\mid \ell$. Identity \eqref{eq:P2 circ X^ell} implies:
$$P\circ P^{\sigma^{-1}}\circ\cdots\circ P^{\sigma^{-(a-2)}}\circ P_1^{\sigma^{-a-1}}=P_2\circ X^{\ell/d^2}\circ X^{d}.$$
This implies:
\begin{equation}\label{eq:P3}
P^{\sigma^{-(a-2)}}\circ P_1^{\sigma^{-a-1}}=P_3\circ X^d
\end{equation}
for some linear polynomial $P_3$. Since $P$ is centered, \eqref{eq:P3} implies that $P_1$ is a scaling and hence $P(X)=\lambda X^d$ for some constant $\lambda$, which is a contradiction. Therefore, we must have $\ell=d$.
Identity \eqref{eq:(iii) and (iv)} implies:
$$P^{\sigma^{-a}}\circ A_0=\zeta \cdot (A_0\circ Q^{\sigma^{c+1}})$$
where $\zeta$ is a $d$-th root of unity. After replacing $A_0$ by $\zeta A_0$, we have that $A$ is unchanged (since $a>0$) and $P^{\sigma^{-a}}\circ A_0=A_0\circ Q^{\sigma^{c+1}}$. By similar arguments, we can also arrange for the identity
$P^{\tau\sigma^{-b}}\circ B_0=B_0^{\sigma}\circ Q^{\sigma^{d+1}}$. Property (vii) follows from (iv) and repeated applications of (vi). Finally we have either $(a,c)$ or $(b,d)$ is $(0,-1)$ since otherwise $Q$ would be a common initial factor of $A$ and $B$.
\end{proof}

\begin{lem}\label{lem:linear A0,B0}
Let $P$ and $\mathcal{C}$ be as in Lemma~\ref{lem:data QABA_0B_0}. Let $Q,A,B,\ldots$ satisfy the conclusion of Lemma~\ref{lem:data QABA_0B_0}. If $A_0$ and $B_0$ are linear then $\mathcal{C}$ is a 
skew-invariant curve for $(P,P^{\tau})$.
\end{lem}
\begin{proof}
From Lemma~\ref{lem:data QABA_0B_0}, we have:
$$A=A_0^{\sigma^a}\circ Q^{\sigma^{a+c}}\circ\cdots Q,\  B=B_0^{\sigma^b}\circ Q^{\sigma^{b+d}}\circ\cdots Q,$$
and without loss of generality, assume $a=0$ and $c=-1$. Then $A=A_0$ is linear.

After replacing $(B,Q,B_0)$ by $(B\circ A^{-1},A^{\sigma}\circ Q\circ A^{-1},B_0\circ (A^{-1})^{\sigma^{d+1}})$, we may assume that $A=A_0=X$. Then property (vi) gives that $P=Q$ and $P^{\tau\sigma^{-b}}\circ B_0=B_0^{\sigma}\circ Q^{\sigma^{d+1}}=B_0^{\sigma}\circ P^{\sigma^{d+1}}$. Therefore:
$$P^{\tau}=B_0^{\sigma^{b+1}}\circ P^{\sigma^{b+d+1}}\circ (B_0^{-1})^{\sigma^b}=(B_0^{\sigma^{-d}}\circ P)^{\sigma^{b+d+1}}\circ \left((B_0^{-1})^{\sigma^{-d}}\right)^{\sigma^{b+d}}.$$
Put $\alpha=B_0^{\sigma^{-d}}\circ P$, $\beta=(B_0^{-1})^{\sigma^{-d}}$, and $n=b+d$, then we have
$P=\beta\circ\alpha$, $P^{\tau}=\alpha^{\sigma^{n+1}}\circ\beta^{\sigma^n}$. and the curve
$\mathcal{C}$ is given by the equation:
$$y=B(x)=B_0^{\sigma^b}\circ Q^{\sigma^{b+d}}\circ\cdots Q(x)=\alpha^{\sigma^n}\circ P^{\sigma^{n-1}}\circ\cdots\circ P(x)$$
fitting the definition of a skew-twist curve.
\end{proof}

\begin{proof}[Proof of Proposition~\ref{prop:warmup-indecomposable} when $P$ is not swappable of type $\mathsf{C}$]
Let $P$ be an indecomposable non-exceptional polynomial and let $\mathcal{C}$ be a skew-invariant curve for $(P,P^{\tau})$. Let $Q,A,B,\ldots$ be the data as in the conclusion of Lemma~\ref{lem:data QABA_0B_0}.

Case 1: $P$ is unswappable. From $P^{\sigma^{-a}} \circ A_0 = A_0^\sigma \circ Q^{\sigma^{c+1}}$ and the fact that $A_0^\sigma$ does not have 
$P^{\sigma^{-a}}$ as a terminal compositional factor, we conclude that $A_0$ is linear. Similarly, $B_0$ is linear and $\mathcal{C}$ is a skew-twist curve thanks to Lemma~\ref{lem:linear A0,B0}.

Case 2: $P$ is swappable and not of type $\mathsf{C}$. If both $A_0$ and $B_0$ are linear then we are done thanks to Lemma~\ref{lem:linear A0,B0}. It remains to consider the following 2 cases.

Case 2.1: $A_0$ is linear and $N:=\deg(B_0)\geq 2$. From properties (v) and (vi) in 
Lemma~\ref{lem:data QABA_0B_0} and \cite[Corollary~5.24]{MS}, there exist linear polynomials $L_1$ and $L_2$ such that:
$$P^{\tau\sigma^{-b}}=L_2^{\sigma}\circ X^kU(X)^N\circ L_2^{-1},$$
$$B_0=L_2\circ X^N\circ L_1^{-1},$$
$$Q^{\sigma^{d+1}}=L_1^{\sigma}\circ X^kU(X^N)\circ L_1^{-1},$$
where $U(X)$ is a non-constant polynomial such that $U(0)\neq 0$ and $k\in\mathbb{N}$ with $\gcd(k,N)=1$.  
 
Put $\eta=\tau^{-1}\sigma^{b-a}$ and $\mu=\sigma^{c-d}$ then the above identities yield:
$$P^{\sigma^{-a}}=L_2^{\eta\sigma}\circ X^kU^{\eta}(X)^N\circ (L_2^{\eta})^{-1}\ \text{and}$$
$$Q^{\sigma^{c+1}}=L_1^{\mu\sigma}\circ X^kU^{\mu}(X^N)\circ (L_1^{\mu})^{-1}$$
Combining this with the identity $P^{\sigma^{-a}} \circ A_0 = A_0^\sigma \circ Q^{\sigma^{c+1}}$ and the 
assumption that $\deg(A_0)=1$, we have:
\begin{equation}\label{eq:case 2.1 first eq}
L^{\sigma}\circ X^kU^{\eta}(X)^N\circ L^{-1}=X^kU^{\mu}(X^N)
\end{equation}
where $L=A_0^{-1}\circ (L_1^{\mu})^{-1}\circ L_2^{\eta}$. Since $P$ is not type $\mathsf{C}$, identity
\eqref{eq:case 2.1 first eq} yields that $L$ is a scaling thanks to \cite[Theorem~3.15]{MS}. This implies:
\begin{equation}\label{eq:case 2.1 second eq}
U(X^N)=cU^{\eta\mu^{-1}}(X)^N
\end{equation}
for some $c\in K^*$. It follows immediately that $U(X)$ must be a monomial: otherwise suppose $X^r$ and $X^s$ with $r>s$ are the highest degree terms with a non-zero coefficients in $U$, then the two such terms in $U(X^N)$ are $X^{Nr}$ and $X^{Ns}$ while the ones for $U^{\eta\mu^{-1}}(X)^N$ are
$X^{Nr}$ and $X^{r(N-1)+s}$, contradicting \eqref{eq:case 2.1 second eq}. But when $U$ is a monomial, we have a contradiction to the assumption that $U(0)\neq 0$. Therefore the case $\deg(A_0)=1$ and $\deg(B_0)\geq 2$ cannot happen. It remains to treat the following case.

Case 2.2: $M:=\deg(A_0)\geq 2$ and $N:=\deg(B_0)\geq 2$. We proceed as in the above case. By \cite[Corollary~5.24]{MS}, there exist linear polynomials $\ell_1,\ell_2,L_1,L_2$ such that:
\begin{align}\label{eq:case 2.2 3 pairs}
\begin{split}
P^{\sigma^{-a}}=\ell_2^{\sigma}\circ X^{k'}V(X)^M\circ \ell_2^{-1}&,\  P^{\tau\sigma^{-b}}=L_2^{\sigma}\circ X^kU(X)^N\circ L_2^{-1},\\
A_0=\ell_2\circ X^M\circ \ell_1^{-1}&,\ B_0=L_2\circ X^N\circ L_1^{-1},\\
Q^{\sigma^{c+1}}=\ell_1^{\sigma}\circ X^{k'}V(X^M)\circ \ell_1^{-1}&,\ Q^{\sigma^{d+1}}=L_1^{\sigma}\circ X^kU(X^N)\circ L_1^{-1},
\end{split}
\end{align}
where $U(X),V(X)$ are non-constant polynomials such that $U(0)V(0)\neq 0$ and $k,k'\in\mathbb{N}$ with $\gcd(k',M)=\gcd(k,N)=1$. As before, put $\eta=\tau^{-1}\sigma^{b-a}$ and $\mu=\sigma^{c-d}$, we then have:
$$P^{\sigma^{-a}}=\ell_2^{\sigma}\circ X^{k'}V(X)^M\circ \ell_2^{-1},\  P^{\sigma^{-a}}=L_2^{\eta\sigma}\circ X^kU^{\eta}(X)^N\circ (L_2^{\eta})^{-1},$$
$$Q^{\sigma^{c+1}}=\ell_1^{\sigma}\circ X^{k'}V(X^M)\circ \ell_1^{-1},\ Q^{\sigma^{c+1}}=L_1^{\mu\sigma}\circ X^kU^{\mu}(X^N)\circ (L_1^{\mu})^{-1}.$$

This implies:
\begin{align}\label{eq:case 2.2 with 2 identities}
\begin{split}
X^{k'}V(X)^M=(\ell_2^{-1})^{\sigma}\circ L_2^{\eta\sigma}\circ X^kU^{\eta}(X)^N\circ (L_2^{\eta})^{-1}\circ \ell_2,\\
X^{k'}V(X^M)=(\ell_1^{-1})^{\sigma}\circ L_1^{\mu\sigma}\circ X^kU^{\mu}(X^N)\circ (L_1^{\mu})^{-1}\circ\ell_1.
\end{split}
\end{align}
The first identity together with the assumption that $P$ is not type $\mathsf{C}$ imply
that $\ell_2^{-1}\circ L_2^{\eta}$ is a scaling thanks to \cite[Theorem~3.15]{MS}. The second identity
together with the fact that $X^{k'}V(X^M)$ and $X^kU^{\eta}(X^N)$ are centered imply that 
$\ell_1^{-1}\circ L_1^{\mu}$ is a scaling. Thanks to these extra properties, \eqref{eq:case 2.2 with 2 identities} yields the following:
\begin{itemize}
\item $k=k'$.
\item $\indeg(V)=\indeg(U)$.
\item $M\cdot\indeg(V)=\indeg(X^{k'}V(X^M))=\indeg(X^kU^{\mu}(X^N))=N\cdot\indeg(U)$. Therefore $M=N$.
\end{itemize}

The middle pair of equations in \eqref{eq:case 2.2 3 pairs} gives:
$$B_0^{\mu}=L_2^{\mu}\circ X^N\circ (L_1^{\mu})^{-1}\circ \ell_1\circ \ell_1^{-1}=L_3\circ X^N\circ \ell_1^{-1}$$
for some linear $L_3$ since $(L_1^{\mu})^{-1}\circ \ell_1$ is a scaling. This implies that 
$X^N\circ \ell_1^{-1}$ is a common initial compositional factor of $A_0$ and $B_0^{\mu}$. After applying $\sigma^{-(c+1)}$, we have that $A_0^{\sigma^{-(c+1)}}$ and $B_0^{\sigma^{-(d+1)}}$ have a nontrivial common initial compositional factor, contradicting property (vii) in Lemma~\ref{lem:data QABA_0B_0}. Therefore the case when $\deg(A_0)>2$ and $\deg(B_0)>2$ cannot happen either.
\end{proof}

\begin{remark}\label{rem:before proof type C}
To finish the proof of Proposition~\ref{prop:warmup-indecomposable}, it remains to consider the case when $P$ is type $\mathsf{C}$. We express $P=R\circ C_q\circ N$ where $R$ and $N$ are linear and $q=\deg(P)$ is an odd prime. Since skew-invariant curves for one pair of polynomials
are transformed into skew-invariant curves for $\sigma$-linear conjugates of those polynomials, after replacing $P$ by an appropriate $\sigma$-linear conjugate, we may assume that $R=\id$. Hence $P=C_q\circ N$ with $N\neq \pm\id$. As in the above proof, we let $Q,A,B,\ldots$ be the data in the conclusion of Lemma~\ref{lem:data QABA_0B_0} then
\cite[Corollary~5.14]{MS} implies that there exist linear polynomials $L_1$ and $L_2$ such that:
$$P^{\tau\sigma^{-b}}=L_2^{\sigma}\circ X^kU(X)^{\delta}\circ L_2^{-1},$$
$$B_0=L_2\circ X^{\delta}\circ L_1^{-1},$$
$$Q^{\sigma^{d+1}}=L_1^{\sigma}\circ X^kU(X^{\delta})\circ L_1^{-1},$$
where $\delta=\deg(B_0)$, $U(X)$ is a non-constant polynomial such that $U(0)\neq 0$ and $k\in\mathbb{N}$ with $\gcd(k,\delta)=1$. Now since $P$ is $C_q\circ N$, we have that $\delta\leq 2$. Hence $\deg(B_0)\leq 2$. Similarly $\deg(A_0)\leq 2$.
\end{remark}

\begin{lem}\label{lem:L1CdL2=xU(x)^2}
Let $d\geq 3$ be an odd integer. Suppose $L_1(x)$, $L_2(x)$, and $U(x)$ are polynomials in $K[x]$
such that $\deg(L_1)=\deg(L_2)=1$, $U(0)\neq 0$, and 
$$L_1\circ C_d\circ L_2=xU(x)^2.$$
Then one of the following cases holds:
\begin{itemize}
\item [(i)] $L_1(x)= a(x+2)$ and $L_2(x)=bx-2$ for some $a,b\in K^*$.
\item [(ii)] $L_1(x)=a(x-2)$ and $L_2(x)=bx+2$ for some $a,b\in K^*$.
\end{itemize}
\end{lem}
\begin{proof}
Since $0$ is a critical value of $xU(x)^2$ while $\pm 2$ are the only critical values of $C_d(x)$, we have that $L_1(2)=0$ or $L_1(-2)=0$.

First, we consider the case $L_1(-2)=0$. From $L_1(C_d(L_2(0)))=0$ and $L_1(-2)=0$, we have $C_d(L_2(0))=-2$. Henc $L_2(0)$ is either $-2$ or one of the remaining $(d-1)/2$ critical points lying above $-2$. But $0$ is not a critical point of $xU(x)^2$, we must have $L_2(0)=-2$. This yields case (i). The case $L_1(2)=0$ is similar and we finish the proof.
\end{proof}

\begin{lem}\label{lem:2 cases for L, N, Q}
Let $d\geq 3$ be an odd integer and let $P_2(x)=x^2$. Suppose that $N$, $L$, and $Q$ are polynomials in $K[x]$ such that
$\deg(N)=\deg(L)=1$, $L$ is a translation, and:
\begin{equation} \label{eq: lem 2 cases for L, N, Q}
N\circ C_d\circ L\circ P_2=L^{\sigma}\circ P_2\circ Q.
\end{equation}
Then one of the following cases holds:
\begin{itemize}
	\item [(i)] $L(x)=x-2$, $N(x)=\lambda(x+2)-2$, and $Q(x)=\pm\sqrt{\lambda} C_q(x)$ for some
	$\lambda\in K^*$.
	
	\item [(ii)] $L(x)=x+2$, $N(x)=\lambda(x-2)+2$, and $Q(x)=\pm\sqrt{\lambda}C_q(x)$ for some $\lambda\in K^*$. 
\end{itemize}
\end{lem}
\begin{proof}
Write $L(x)=x+\ell$. By Ritt's theorem, there exist linear polynomials $\alpha,\beta,\gamma,\delta$ and polynomial
$U(x)$ in $K[x]$ such that:
\begin{align}
\alpha\circ N\circ C_d\circ\beta^{-1} &=xU(x)^2,\label{eq:lem 2 cases for L, N, Q eq1}\\
\beta\circ L\circ P_2\circ\gamma &=x^2,\label{eq:lem 2 cases for L, N, Q eq2}\\
\alpha\circ L^\sigma\circ P_2\circ\delta^{-1} &=x^2, \label{eq:lem 2 cases for L, N, Q eq3}\\
\delta\circ Q\circ P_2\circ\gamma &=xU(x^2) \label{eq:lem 2 cases for L, N, Q eq4}
\end{align}

By Lemma~\ref{lem:L1CdL2=xU(x)^2} for \eqref{eq:lem 2 cases for L, N, Q eq1}, we have two cases. The first case is that:
\begin{equation}\label{eq:alpha o N and beta^-1}
\alpha\circ N=a(x+2) \text{and}\ \beta^{-1}=bx-2\ \text{for some $a,b\in K^*$}.
\end{equation}

Since $P_2(x)=x^2$ and $L(x)=x+\ell$, from \eqref{eq:lem 2 cases for L, N, Q eq2}, we have:
\begin{equation}\label{eq:gamma and beta}
\gamma(x)=cx,\ \beta\circ L=\frac{1}{c^2}x,\ \beta=\frac{1}{c^2}(x-\ell)\ \text{for some $c\in K^*$}.
\end{equation}

From \eqref{eq:alpha o N and beta^-1} and \eqref{eq:gamma and beta}, we have:
$$\frac{x+2}{b}=\frac{1}{c^2}(x-\ell).$$
This gives $b=c^2$ and $\ell=-2$. We now have $L(x)=x-2$ and the given \eqref{eq: lem 2 cases for L, N, Q}
becomes:
$$N(C_d(x^2-2))=Q(x)^2-2.$$
Combining this with the identity $C_d(x^2\pm 2)=C_d(x)^2\pm 2$ and expressing $N(x)=\lambda x+d$, we have:
$$\lambda C_d(x)^2-2\lambda + d=Q(x)^2-2.$$
This implies $\lambda C_d(x)^2=Q(x)^2$ and $-2\lambda+d=-2$. This yields case (i).

The remaining case when applying Lemma~\ref{lem:L1CdL2=xU(x)^2} for \eqref{eq:lem 2 cases for L, N, Q eq1} is that:
\begin{equation}
\alpha\circ N=a(x-2) \text{and}\ \beta^{-1}=bx+2\ \text{for some $a,b\in K^*$}.
\end{equation}
This can be settled by similar arguments.
\end{proof}

\begin{proof}[Proof of Proposition 2.22]
We apply the reduction and observation in Remark~\ref{rem:before proof type C}: $P=C_q\circ N$ where $q$ is an odd prime and $N\neq \pm \id$ is linear, let $Q,A,B,\ldots$ be the data in the conclusion of Lemma~\ref{lem:data QABA_0B_0} and we have $\deg(A_0),\deg(B_0)\leq 2$. 
We prove by contradiction that the case $(\deg(A_0),\deg(B_0))\neq (1,1)$ cannot happen then 
Lemma~\ref{lem:linear A0,B0} finishes the proof.

\textbf{Case 1:} $\deg(A_0)=\deg(B_0)=2$. With $P_2(x)=x^2$, we can express $A_0=L\circ P_2\circ M$ and $B_0=L_1\circ P_2\circ M_1$ where $L,L_1,M,M_1$ are linear and $L$ and $L_1$ are translation. Then property (vi) in Lemma~\ref{lem:data QABA_0B_0} together with $P=N\circ C_q$ yield:
\begin{align}
N^{\sigma^{-a}}\circ C_q\circ L\circ P_2 &=L^{\sigma}\circ P_2\circ (M^\sigma\circ Q^{\sigma^{c+1}}\circ M^{-1}),\label{eq:Case 2 first eq in the pair}\\
N^{\tau\sigma^{-b}}\circ C_q\circ L_1\circ P_2 &=L_1^{\sigma}\circ P_2\circ (M_1^\sigma\circ Q^{\sigma^{d+1}}\circ M_1^{-1}).\label{eq:Case 2 second eq in the pair}
\end{align}
We apply Lemma~\ref{lem:2 cases for L, N, Q} to the above pair of equations.

\textbf{Case 1.1:} both data $(N^{\sigma^{-a}},L,M^\sigma\circ Q^{\sigma^{c+1}}\circ M^{-1})$ and 
$(N^{\tau\sigma^{-b}},L_1,M_1^\sigma\circ Q^{\sigma^{d+1}}\circ M_1^{-1})$ satisfy the conclusion in case (i) of Lemma~\ref{lem:2 cases for L, N, Q} or both data satisfy the conclusion in case (ii). This gives
$L=L_1=x-2$ or $L=L_1=x+2$ and there exist $\alpha,\alpha_1\in K^*$ such that:
$$M^\sigma\circ Q^{\sigma^{c+1}}\circ M^{-1}=\alpha C_q\ \text{and}\ M_1^\sigma\circ Q^{\sigma^{d+1}}\circ M_1^{-1}=\alpha_1 C_q.$$
This implies:
$$Q^{\sigma^{c+1}}=(M^{-1})^\sigma\circ(\cdot\alpha)\circ C_q\circ M\ \text{and}\ Q^{\sigma^{d+1}}=(M_1^{-1})^\sigma\circ(\cdot\alpha_1)\circ C_q\circ M_1.$$
Together with the fact that if $\ell_1\circ C_q\circ \ell_2=C_q$ for some linear polynomials $\ell_1$ and $\ell_2$
then $\ell_1=\ell_2=\pm x$, we have:
$$M^{\sigma^{-(c+1)}}=\pm M_1^{\sigma^{-(d+1)}}.$$
We now have:
$$A_0^{\sigma^{-(c+1)}}=L\circ P_2\circ M^{\sigma^{-(c+1)}}=L_1\circ P_2\circ M_1^{\sigma^{-(d+1)}}=B_0^{\sigma^{-(d+1)}}.$$
This contradicts property (vii) in Lemma~\ref{lem:data QABA_0B_0}.

\textbf{Case 1.2:} between the data $(N^{\sigma^{-a}},L,M^\sigma\circ Q^{\sigma^{c+1}}\circ M^{-1})$ and
$(N^{\tau\sigma^{-b}},L_1,M_1^\sigma\circ Q^{\sigma^{d+1}}\circ M_1^{-1})$, one satisfies (i) while the other one satisfies (ii) in Lemma~\ref{lem:2 cases for L, N, Q}. Without loss of generality, assume
the first data satisfies (i) and the second satisfies (ii). Hence there exist $\alpha,\beta\in K^*$ such that:
$$N^{\sigma^{-a}}=\alpha x+2\alpha-2\ \text{and}\ N^{\tau\sigma^{-b}}=\beta x-2\beta+2.$$
This gives $\beta=\alpha^{\tau\sigma^{a-b}}$ and $-2\beta+2=2\alpha^{\tau\sigma^{a-b}}-2=2\beta-2$. Hence $\beta=1$ and $N(x)=x$, contradicting the condition that $N\neq\pm \id$.

\textbf{Case 2:} without loss of generality, the remaining case is that $\deg(A_0)=2$ and $\deg(B_0)=1$.
From property (vi) in Lemma~\ref{lem:data QABA_0B_0}, we have:
$$Q^{\sigma^{d+1}}=(B_0^{-1})^\sigma\circ P^{\tau\sigma^{-b}}\circ B_0=(B_0^{-1})^\sigma\circ N^{\tau\sigma^{-b}}\circ C_q\circ B_0.$$
Applying $\sigma^{c-d}$ both sides gives:
\begin{equation}\label{eq: Case 3 first eq}
Q^{\sigma^{c+1}}=(B_0^{-1})^{\sigma^{c-d+1}}\circ N^{\tau\sigma^{c-d-b}}\circ C_q\circ B_0^{\sigma^{c-d}}.
\end{equation}

As in the earlier Case 1, we express $A_0=L\circ P_2\circ M$ where $M$ is linear and $L$ is a translation. Then we get
\begin{align}\label{eq: Case 3 but same eq as before}
N^{\sigma^{-a}}\circ C_q\circ L\circ P_2 &=L^{\sigma}\circ P_2\circ (M^\sigma\circ Q^{\sigma^{c+1}}\circ M^{-1}),
\end{align}
from property (vi) in Lemma~\ref{lem:data QABA_0B_0} as before. Applying Lemma~\ref{lem:2 cases for L, N, Q} to \eqref{eq: Case 3 but same eq as before}, there exists $\alpha\in K^*$ such that:
\begin{equation}\label{eq: Case 3 eq with alpha}
M^\sigma\circ Q^{\sigma^{c+1}}\circ M^{-1}=\alpha C_q.
\end{equation}
Together with \eqref{eq: Case 3 first eq}, we have:
\begin{equation}\label{eq: Case 3 after combining equations}
(M\circ B_0^{-1})^{\sigma}\circ N^{\tau\sigma^{c-d-b}}\circ C_d\circ B_0\circ M^{-1}=\alpha C_q.
\end{equation}
This implies $B_0\circ M^{-1}=\pm x$, then $(M\circ B_0^{-1})^{\sigma}=\pm x$, and hence
\begin{equation}\label{eq: Case 3 N has no constant term}
N^{\tau\sigma^{c-d-b}}=\pm\alpha x.
\end{equation}
Applying Lemma~\ref{lem:2 cases for L, N, Q} to \eqref{eq: Case 3 but same eq as before} again, there is $\beta\in K^*$ such that $N^{\sigma^{-a}}=\beta x+2\beta-2$ or $N^{\sigma^{-a}}=\beta x-2\beta+2$. But in either case, \eqref{eq: Case 3 N has no constant term} implies that $\beta=1$ and hence $N(x)=x$ contradicting the condition $N\neq \pm \id$.
\end{proof}

\subsubsection{Partial group actions}

We introduce partial actions and rewriting systems to view the result of sequences of these partial
operations on $\kleq$ as partial group actions.
From the intended actions of these generators of the groups to get a (usual) group action, we would have to show that the relations given by
a presentation of the relevant group are respected. However, this is not the case here.
To get these desired partial group actions, asymmetric term-rewriting systems replace the relations in a presentation.

All of the following definitions are standard, with the exception of the very simple and natural notion of actionability.
A standard reference for rewriting systems is~\cite{BoOt} and for partial group actions is~\cite{Exel}.  However, we use somewhat different
notation from these other authors.

\begin{Def}

For a set $X$, let $M_X$ be the monoid of partial injective functions on $X$.
That is, elements of $M_X$ are injections $m: Y \rightarrow X$ for some $Y \subseteq X$, and the operation is composition, with maximal possible domain.

For a set $T$, let $T^*$ be the free monoid generated by $T$, also known as the set of words on $T$.
For a function $\alpha: T \rightarrow M_X$, let $\alpha^*: T^* \rightarrow M_X$ be the induced monoid homomorphism.

For a monoid $G$ and a set $X$, a \emph{partial action} of $G$ on $X$ is a function $\beta: G \rightarrow M_X$ such that
for any $a,b \in G$,  $$\beta(b) \circ \beta(a) \subseteq \beta(ba).$$
Here, the subset relation is in the sense of the graphs of functions. In more functional notation, $m_1 \subseteq m_2$ means that $m_1$ is the restriction of $m_2$ to a possibly smaller set. When $\beta$ is clear, we write $g \star x$ instead of $\beta(g) (x)$ for $g \in G$ and $x \in X$.

\end{Def}

The purpose of a rewriting system $R$ is that whenever $(u,v) \in R$, we are allowed to replace a substring $u$ in some big string $aub$ by the string $v$ to get $avb$. This is the asymmetric analog of the relations in a group presentation.

\begin{Def}(See~\cite[Definition 2.1.1 and 2.1.2]{BoOt})
A \emph{string rewriting system} on $T$ is a subset $R$ of $T^* \times T^*$.
The \emph{reduction relation} $\reduct{R}$ induced by $R$ is the reflexive transitive closure of the relation
$$\{ (aub, avb) \in T^* \times T^*) : a, b \in T^* \text{ and } (u,v) \in R \}.$$

The equivalence relation generated by $\reduct{R}$ is denoted by $\congru{R}$. It is the monoid congruence generated by $R$. The quotient monoid and the quotient map are denoted by $G_R$ and $\pi_R: T^* \rightarrow G_R$, respectively.

The string rewriting system $R$ is \emph{confluent} if whenever $(v,w) \in \congru{R}$, there is another word $u$ with both $(v,u) \in \reduct{R}$ and $(w,u) \in \reduct{R}$.
 That is, whenever two words are $R$-equivalent, they are $R$-reducible to the same word.

A word $w \in T^*$ is \emph{$\congru{R}$-reduced} if it is not $\congru{R}$-equivalent to any shorter word.

\end{Def}

\begin{Def} \label{def:actionable}
Fix sets $X$ and $T$, a function $\alpha: T \rightarrow M_X$, and a confluent rewriting system $R$ on $T$.
The pair $(\alpha, R)$ is \emph{actionable} if $\alpha^*(u) \subseteq \alpha^*(v)$ for all $u, v \in T^*$ with $(u, v) \in R$.
\end{Def}

The next proposition demonstrates that the technical notion of ``actionable'' from Definitions~\ref{def:actionable} exactly
corresponds to the idea that an action defined on generators is really ``actionable'' in that it may be made into a partial action.

\begin{prop} \label{prop:actionable}
Fix sets $X$ and $T$, a function $\alpha: T \rightarrow M_X$, and a confluent rewriting system $R$ on $T$.
If the pair $(\alpha, R)$ is actionable, then
 $$\bar{\alpha}(g) := \cup \{ \alpha^*(w) : w \in T^* \text{ and } \pi_R(w) = g \}$$
defines a partial action $\star$ of $G_R$ on $X$.

In functional notation, this union means that for $g \in G$ and $x,y \in X$,
 $$g \star x = y \text{ if and only if } \alpha^*(w) (x) = y \text{ for some } w \in T^* \text{ with } \pi_{R}(w) = g \text{ .}$$
\end{prop}

\begin{proof}
It is clear that whenever $(R, \alpha)$ is actionable, so is $(\reduct{R}, \alpha)$.

We first show that the union $\bar{\alpha}(g)$ of functions is itself a function.
 Suppose that $g \star x = y$ and $g \star x = y'$; that is, there are words $w, w' \in T^*$
 with $\pi_R(w) = \pi_R(w') = g$ and $\alpha^*(w) (x) = y$ and $\alpha^*(w') (x) = y'$. Since $\pi_R(w) = \pi_R(w') = g$, we must have
 $(w,w') \in \congru{R}$. Since $R$ is confluent, there is a word $u \in T^*$ with both $(w,u)$ and $(w', u)$ in $\reduct{R}$.
Since $(\reduct{R}, \alpha)$ is actionable, $(w,u) \in \reduct{R}$ and $\alpha^*(w) (x) = y$ implies that $\alpha^*(u) (x) = y$.
The same reasoning with $w'$ and $y'$ shows that $\alpha^*(u) (x) = y'$.  So $y = y'$ as wanted.

We now show that $\bar{\alpha}$ satisfies the definition of partial action.
Fix $a,b \in G_R$ and $x,y,z \in X$, and suppose that $a \star x = y$ and $b \star y = z$.
Then there are words $u, v \in T^*$ with $\pi_{R}(u) = a$ and $\pi_{R}(v) = b$ and
$\alpha^*(u) (x) = y$ and $\alpha^*(v) (y) = z$. Let $w := vu$ be the concatenation of the two words.
Then $\pi_{R}(w) = \pi_{R}(vu) = ba$ and $\alpha^*(w) (x) = z$. So $(ba) \star x = z$.
\end{proof}

The generators we are interested in are $\phi$ and its inverse; the $\epsilon_p$ for prime $p$, and their inverses; and the $t_i$ for $i = 1, 2, \ldots k$.

\begin{Def}
\label{def:rewrite-groups}
Foe $k \geq 3$, here are the four groups we are interested in and their rewriting systems.
\begin{itemize}
\item $\Symk$ is generated by $T_{\Symk} := \{t_1, t_2, \ldots t_{k-1}\}$ and defined by the rewriting system $$R_{\Symk} := \{ (t_i t_i, \id) \} \cup \{ (t_i t_j, t_j t_i) : j \neq i \pm 1 \}
    \cup \{ (t_i t_j t_i, t_j t_i t_j) : j = i \pm 1 \} \text{ .}$$

\item $\Affk$ is generated by $T_{\Affk} := \{t_1, t_2, \ldots t_{k-1}, t_k\}$ and defined by the rewriting system
    $$R_{\Affk} := \{ (t_i t_i, \id) \} \cup \{ (t_i t_j, t_j t_i) : j \not \equiv i \pm 1 \pmod{k} \} $$
   $$ \cup \{ (t_i t_j t_i, t_j t_i t_j) : j \equiv i \pm 1 \pmod{k} \}$$
    where addition of indices is modulo $k$, so including $(t_kt_1t_k, t_1t_kt_1)$ and excluding $(t_1t_k, t_k t_1)$.

\item $\STk$ is generated by $T_{\STk} := T_{\Affk} \cup \{ \phi, \phi^{-1} \}$ and defined by the rewriting system
    $$R_{\STk} := R_{\Affk}
        \cup \{ (\phi\phi^{-1}, \id), (\phi^{-1}\phi, \id), (\id, \phi^{-1}\phi), (\id, \phi\phi^{-1}) \} \cup$$
        $$\cup \{ (\phi t_i, t_{i+1} \phi) : 1 \leq i \leq k \}
        \cup \{ (t_{i+1} \phi, \phi t_i) : 1 \leq i \leq k \},$$
    where addition of indices is again modulo $k$.

\item $\STpk$ is generated by $T_{\STpk} := T_{\STk} \cup \{ \epsilon_p, \epsilon_p^{-1} : \mbox{ prime } p \}$ and defined by
 the rewriting system $$R_{\STpk} := R_{\STk} \cup \{ (\epsilon_p \omega, \omega \epsilon_p),
 (\omega \epsilon_p, \epsilon_p \omega)  ~:~ p \text{ prime }, \omega \in T_{\STpk} \}$$
 $$\cup \{  (\epsilon_p \epsilon_p^{-1}, \id), (\epsilon_p^{-1} \epsilon_p, \id) ~:~ p \text{ prime }\} \text{ .}$$
That is, $\epsilon_p$ commutes with everything and $\epsilon_p^{-1}$ is its inverse.

\end{itemize}

When $k = 2$, we drop the braid relation $t_1 t_2 t_1 = t_2 t_1 t_2$.  So, $\operatorname{ASym}_2$ is the infinite dihedral
group.  We do not define these groups when $k=1$.
\end{Def}

\begin{Rk} \label{rk:meet-the-groups-rk}
 Interpreting $t_i$ as the transposition switching $i$ and $i+1$ in the set $\{1, 2, \ldots, k\}$ identifies $\Symk$ with the usual symmetric group. For $k \geq 3$, the next group $\Affk$ is well-known under many names 
 including ``affine symmetric group.''
 It has several copies of $\Symk$ inside it, generated by $T_j := \{ t_i: i \neq j, 1 \leq i \leq k \}$ for any fixed
 $j \leq k$. Even when $k = 2$, the group
 $\Affk$ admits a surjective homomorphism to $\Symk$ sending $t_i$ to the transposition switching $i$ and $i+1$ for $i \neq k$ and sending $t_k$ to the one switching $1$ and $k$. This homomorphism restricts to an isomorphism on each copy of $\Symk$ inside $\Affk$. The kernel of this homormophism is a copy of $\ZZ$. In particular, $\Affk$ is infinite. It is clear from the rewriting system that $\STk$ is the semidirect product of $\Affk \unlhd \STk$ and the infinite cyclic group generated by $\phi$.
 It can be useful to think of $\STk$ and $\Affk$ as subgroups of permutations of $\ZZ$, where $t_i$ is the permutation that switches $i+rk$ with $i+1 + rk$ for all $r \in \ZZ$ and $\phi$ is the shift by $1$. We write $\pi:\Affk \to \operatorname{Sym}(\ZZ)$ for this group homomorphism.
\end{Rk}

\begin{Rk}
\label{rk:semidirect}
The generating sets of Definition~\ref{def:rewrite-groups} form a chain,
and the groups are a chain of groups $\Symk < \Affk < \STk < \STpk$.
We have presented $\STk$ as a semidirect product of
$\ZZ \cong \langle \phi \rangle$ with $\Affk$.
The action of conjugation by $\phi$ on
$\Affk$ is given by the automorphism
$\eta:\Affk \to \Affk$ induced by
$t_i \mapsto t_{i+1}$.
Likewise, we have presented
$\STpk$ as a product of $\STk$ with
$\bigoplus_{i=0}^\infty \ZZ  \cong
\langle \epsilon_p ~:~ p \text{ a prime } \rangle$.
Thus, any element of $\STpk$ may be written in the form
$\prod_p \epsilon^{m_p}  \phi^n  w$ where $m_p \in \ZZ$, all
but finitely many of which are zero, $n \in \ZZ$, and $w \in \Affk$.
\end{Rk}

The first two groups, $\Symk$ and $\Affk$ are Coxeter groups (see~\cite{Brown}) with their usual Coxeter rewriting systems. The first part of each, relations of the form $(t_i t_i, \id)$, shorten the word they act on and are asymmetric; while the rest are symmetric and do not change the length of the word they act on.  Matsumoto's Theorem applies to $\Symk$ and $\Affk$, with $R_{\text{cox}}$ consisting of the second and third sets of rewriting rules,
that is, not including the pairs $(t_i t_i, \id)$, and $R_{\text{cox}+}$ being the whole rewriting system.

\begin{fact}[Matsumoto's Theorem~\cite{Matsumoto}]
 Any two $R_{\text{cox}+}$-equivalent, $R_{\text{cox}+}$-reduced words are also $R_{\text{cox}}$-equivalent; and every word is $R_{\text{cox}+}$-reducible to an $R_{\text{cox}+}$-reduced one.
Since $R_{\text{cox}}$ is symmetric, any two $R_{\text{cox}}$-equivalent words are $R_{\text{cox}}$-reducible to each other.
 In particular, it follows that $R_{\text{cox}+}$ is confluent.
\end{fact}

Recall that a word $w \in T^*$ is \emph{$\congru{R}$-reduced} if it is not $\congru{R}$-equivalent to any shorter word.
The next corollary follows easily from the hard result, Theorem~\ref{thm:braidhtm} below, analogous to~\cite[Theorem 2.52]{MS}  , that the braid relations $(t_i t_{i+1} t_i, t_{i+1} t_i t_{i+1})$ hold for Ritt swaps.

\begin{cor}  
\label{cor:group-action-corr}
Each of the four confluent rewriting systems of Definition~\ref{def:rewrite-groups} together with the actions of
its generators defined Definitions~\ref{def:Ritt-swap},~\ref{def:phi-action}, and~\ref{def:epsilon}  is actionable, and
the partial group action $\bar{\alpha}$ is also given by
 $$g \star x = y \text{ if and only if } \alpha^*(w) (x) = y \text{ for every reduced } w \in T^* \text{ with } \pi_{R}(w) = g.$$
\end{cor}

\begin{proof}
For $k \geq 3$, confluence of the first two rewriting systems $R_{\Symk}$ and $R_{\Affk}$ is a special case of Matsumoto's Theorem.  
Confluence in the case of $k = 2$ is triival.
 By Remark~\ref{rk:semidirect}, confluence of $R_{\STk}$ and of $R_{\STpk}$ follows.

For $k < 3$, braid relations are irrelevant. For $k \geq 3$, Theorem~\ref{thm:braidhtm} verifies the braid relations that
$$t_i \star (t_{i+1}  \star (t_i  \star [f] )) =  t_{i+1} \star  (t_i \star (t_{i+1} \star  [f] ))$$
for any $i = 1, 2, \ldots k$ and any long decomposition $f$. That is, they
are equal when defined, and one is defined if and only if the other is.
All other instances of the definition of ``actionable'' are immediate.

We can replace ``some $w \in T^*$ with $\pi_{R}(w) = g$'' in the functional version of
 Proposition~\ref{prop:actionable} with ``any reduced $w \in T^*$ with $\pi_{R}(w) = g$'' because the
 conclusion of Matsumoto's theorem is this strong version of confluence.
\end{proof}

\begin{Rk}
In~\cite{MS} we worked with the language of monoid actions
rather than partial group actions and the space on
which those monoids acted was the set of equivalence classes
up to skew-linear equivalence of decompositions of length
$k$. The actions of $\Symk$, $\STk$, and $\STpk$ given here correspond to the actions of the Ritt Monoid $\operatorname{RM}_{k}$, the skew-twist monoid $\operatorname{ST}_{k}$, and the augmented skew-twist monoid $\operatorname{ST}^+_{k}$ there.
Our $\STk$ is generated as a monoid
by $t_1$, \ldots, $t_{k-1}$, $\phi$, and $\phi^{-1}$
because $t_k = \phi t_1 \phi^{-1}$. The subgroup of $\STk$ generated by $t_1, t_2, \ldots t_{k-2}$ along with
$\psi := t_{k-1} \phi$ and $\gamma := \phi^{-1} t_{k-1}$ would be the analog of the border guard monoid $\operatorname{BG}_{k}$ from~\cite{MS}; but the new intermediate group $\Affk$ turns out to be much more convenient.
\end{Rk}

The two canonical forms for words in $\Symk$ from~\cite{MS} fit nicely with Matsumoto's theorem. We shall use them here, not only with the standard copy of $\Symk$ in $\Affk$, but also with the other copies of Remark~\ref{rk:meet-the-groups-rk}. The following are intuitive explanations of analogues of~\cite[Definitions 5.6, 5.8, and 5.14]{MS}.

\begin{Def}
\label{def:transit}
For $b+k > a > b$, the word
$t_{a-1} t_{a-2} \ldots t_b$ in $T_{\Symk}^*$ is called a \emph{right-to-left transit}.
In the action of $\Symk$ on long decompositions, this drags the factor $f_b$ left from its original position at $b$ to its new position at $a$, across all the intermediate factors.  A
left-to-right transit is defined symmetrically.
A \emph{transit} is one of the two.
\end{Def}

\begin{Rk}
\label{rk:canonical-form}
A sequence of Ritt swaps in the right-to-left
first canonical form is a sequence of
right-to-left transits whose action resembles an 
insert-sort: having arranged the factors
$f_k$ through $f_{i+1}$ in the correct order, this
sequence \emph{inserts} the next factor $f_i$ in the
required place among $f_k$ through $f_{i+1}$; and then
proceeds to deal with $f_{i-1}$, and so on, until all
factors are arranged as wanted.  We define left-to-right
first canonical form analogously.

Sometimes, there is a natural way to break a long decomposition into chunks. Rooms of Section~\ref{sec:walls} and clusters in the technical 
Section~\ref{sec:clusterings} are two examples. Words in second canonical form (with respect to such a chunkification) first shuffle factors within 
each chunk as much as necessary and only then move factors between chunks. This is closely related, but not identical, to a merge-sort.

Given integers $c_0 + k = c_r > c_{r-1} > \ldots > c_1 > c_0$, a word $T_{\Symk}^*$ is in \emph{second canonical form with respect to $\bar{c}$} if it is of the form $v w_r w_{r-1} \ldots w_1$ where \begin{itemize}
 \item each $w_j$ is a word in first canonical form, using only $t_i$ with $c_{i+1} > i > c_i$; and
 \item $v$ is a word in first canonical form that does not change the order of factors originating within the same chunk. \end{itemize}
\end{Rk}

We restate the key Propositions 5.11 and 5.15 from~\cite{MS} in the language of rewriting systems and observe that they are immediate consequences of Matsumoto's theorem

\begin{prop}
Every word in $T_{\Symk}^*$ is $R_{\Symk}$-reducible to a word in first canonical form.
For every $\bar{c}$, every word in $T_{\Symk}^*$ is $R_{\Symk}$-reducible to a word in second canonical form with respect to $\bar{c}$.
\end{prop}

\begin{proof}
This proof is really a template that can likely work for other canonical forms.

As should be clear, and is proven in detail in~\cite[Propositions 5.11 and 5.15]{MS},
every permutation in $\Symk$ is represented by a word in $T_{\Symk}^*$ in each of our canonical forms.
It is clear that words in these canonical forms are reduced: to write a permutation $\sigma$ as a product of adjacent transpositions requires at least one such transposition for each pair $a<b$ with $\sigma(b) < \sigma(a)$, and these canonical forms use exactly this many generators. So, for each word $w$ in $T_{\Symk}^*$, there is a reduced word $v$ in this canonical form, representing the same permutation as $w$. That is, $v$ and $w$ are $R_{\Symk}$-equivalent, and $v$ is $R_{\Symk}$-reduced; so by Matsumoto's theorem, $w$ is $R_{\Symk}$-reducible to $v$.
\end{proof}

\begin{Rk}
\label{rk:canonical-reduced}
Words in first or second canonical form are
themselves reduced since we have shown that every word is
not only equivalent to, but reducible to such a word. 	
\end{Rk}

The following refined version of the second canonical form, needed very soon, showcases how long decompositions
make it easy to talk about a period that is not $(f_k, f_{k-1}, \ldots f_1)$. Instead, we focus on  the behavior of $w \star$ on
the substring $\left( f_{k+i}, f_{k+i-1}, \ldots, f_{k+1}, f_k, \ldots f_{i+1} \right)$ separated into two chunks
$\left( f_{k+i}, f_{k+i-1}, \ldots, f_{k+1} \right)$ and $\left( f_k, \ldots f_{i+1} \right)$.

\begin{lem}
\label{lem:improved-2nd-canonical}
Fix some $i$ with $0 < i < k$.
Suppose $w \star f$ is defined for some long decomposition $f$ and some word $w$ in $\{ t_j ~:~ j \neq i \}$.
Then there are words $u_L^{  +}$, $u_R^{  +}$, $v$, $u_L^{  -}$, and $u_L^{  +}$ so that
\begin{itemize}
\item $u_L^{  +} u_R^{  +} v u_L^{  -} u_R^{  -} \star f
= w \star f$,
\item $u_L^{  +}$ and $u_L^{  -}$ are words in
$\{t_j ~:~ 1 \leq j < i \}$,
\item $u_R^{  +}$ and $u_R^{  -}$ are words in
$\{ t_j ~:~ i < j < k \}$,
\item for some $m \leq  \min \{i, k-i\}$, $v \star$
has the effect of swapping the block
$\left( f_k,\ldots, f_{k-m+1} \right)$ with
$\left( f_{k+m}, \ldots, f_{k+1}\right)$.   In terms of our
generators, $v$ takes the form $v = v_{k+1-m}
\cdots v_k$ where where
$v_j = t_{m+j-1} \cdots t_{j}$.     	
\end{itemize}
\end{lem}

\begin{proof}
Indeed, using the second canonical form
of~\cite[Proposition 5.15]{MS}, we may find words
$u_L^{  -}$, $u_R^{  -}$, and $\hat{w}$ so that
\begin{itemize}
\item $\hat{w} u_L^{  -} u_R^{  -} \star f =
w \star f$,
\item $u_L^{  -}$ is a word in $\{t_j ~:~ 1 \leq j < i \}$,
\item $u_R^{  -}$ is a word in $\{t_j ~:~ i < j < k \}$, and
\item $\hat{w}$ is composed of nonoverlapping
right-to-left transits across $k$:  there are  numbers
$k+i - 1 \geq \ell_k > \ell_{k-1} > \cdots > \ell_i \geq i$ so that we may
write $\hat{w} = \hat{w}_i \cdots \hat{w}_k$ with
$\hat{w}_j = t_{\ell_j} t_{\ell_j - 1} \cdots t_j$
if $\ell_j \geq j$ and $\hat{w}_j$ is empty
and $\ell_j = j - 1$ otherwise.
\end{itemize}

We now further decompose $\hat{w}$.
Let $n := \min \{ j ~:~  k >  j \geq i
~\&~ \ell_{j+1} \geq k \}$.

  If no such $n$ exists, then $\hat{w}$ is
  the empty word, so that we may take
  $u_R^{  +}$, $u_L^{  +}$, and $v$ to all
  be the empty word.

Otherwise, let $u_R^{  +} := \hat{w}_n \cdots \hat{w}_i$.
From the construction,
$u_R^{  +}$ is a word in $\{ t_j ~:~ i < j < k \}$.
If $n = k$, then $\hat{w} = u_R^{  +}$ and we finish
by taking $v$ and $u_L^{  +}$ to be the empty word.
So, we may assume that $n < k$.  Let $m := k - n$.
 Consider the
expanded product
\begin{eqnarray*}
\hat{w}_{n+1} \cdots \hat{w}_k & = &
(t_{\ell_{n+1}} \cdots t_{n+1})
\cdots (t_{\ell_{k-1}} \cdots
t_{k-1}) (t_{\ell_{k}} \cdots
t_k )\\
 & = & (t_{\ell_{n+1}} \cdots t_{k+1} t_k \cdots t_{n+1})
 	\cdots  \\
 	&& (t_{\ell_{k-1}} \cdots t_{k+m}t_{k-1+m}
 	  \cdots t_{k-1})(t_{\ell_k}
 		\cdots t_{k+m+1} t_{k+m} \cdots t_k)
\end{eqnarray*}

The blocks $t_{k-1+m} \cdots t_{k-1}$
and $t_{\ell_k} \cdots t_{k+m+1}$ commute.
Hence, the word $\hat{w}$ is equivalent to
$$(t_{\ell_{n+1}} \cdots
t_{k+1} t_k \cdots t_{n+1})
\cdots (t_{\ell_{k-2}}
\cdots t_{k-1+m} t_{k-2+m}
\cdots t_{k-2})[t_{\ell_{k-1}} \cdots $$
$$t_{k+m} ][t_{\ell_k} \cdots t_{k+m+1}]
[t_{k+m-1} \cdots t_{k-1}][t_{k+m}
\cdots t_k]$$

Continuing in this way, if we
set $$u_L^{  +} := [t_{\ell_{n+1}}
\cdots t_{k+1} ][t_{\ell_{n+2}}
\cdots t_{k+2}] \cdots [t_{\ell_k}
\cdots t_{k+m+1}] \text{ ,}$$
then we obtain the requisite decomposition.
\end{proof}

The following strong version of the braid relations
will be useful in analyzing wandering quadratics (see for example, the proof of Lemma~\ref{lem:wqua-one-transit}).

\begin{lem}
\label{lem:transit-twist-commute}
Fix $b >a > b-k$. Suppose that $s = t_{(b,a]}$ is a transit from $b$ to $a$, and $w$ is a word in $t_i$ for $b > i > a$, and $\widehat{w}$ is obtained from $w$ by replacing every $t_i$ by $t_{i-1}$. Then $sw = \widehat{w}s$.
\end{lem}
\begin{proof}
Apply the left-to-right first canonical form to maybe $(sw)^{-1}$ (see Remark~\ref{rk:canonical-form}).
\end{proof}

An identity of the form $w \star [\vec{f}] = [\vec{g}]$
with $w \in \STpk$ and long decompositions
$\vec{f}$ and $\vec{g}$ of $f$ and $g$,
respectively, encodes an
$(f,g)$-skew-invariant curve~\cite[Definition 6.3]{MS}.  Moreover,
every such $(f,g)$-skew-invariant curve is
encoded by such an identity~\cite[Proposition 6.19]{MS}.
Let us explain how this relates our new formalism
from Remark~\ref{rk:witness-Ritt-swap}.

\begin{Rk}
\label{rk:curve-encoded}
When $k \neq 2$, Remark~\ref{rk:witness-Ritt-swap}
permits a more elegant characterization of an encoded
curve.

If $f$ is a \kld
and $w = w_n \cdots w_1 \in \STpk$ where each $w_j$ is a generator not equal to $t_k$, then
a \emph{witnessing sequence} for
$w \star [f]$ consists of a sequence $\left( f^0, \ldots, f^n \right)$ of
\klds and a sequence $\left( L^1, \ldots, L^n \right)$ of \kspls so that
\begin{itemize}
\item $f = f^0$,
\item if $w_j \neq t_k$, then $w_j \bullet (\phi \bullet L^j)^{-1} f^{j-1} L^j = f^j$, and
\item if $w_j = \phi$ or $w_j = \phi^{-1}$, then every component of $L^j$ is the polynomial $x$.	
\end{itemize}

We say that $\mathcal{A}$ is a correspondence encoded by this witnessing sequence  if
$\mathcal{A}$ is a composition of curves
$\mathcal{B}_n \circ \cdots \circ \mathcal{B}_1$ where for each $j$
\begin{itemize}
\item if $w_j = t_i$, then $\mathcal{B}_j$ is the curve defined by $L_1^j(x_2) = x_1$,
\item if $w_j = \epsilon_p$, then $\mathcal{B}_j$ is the curve defined by $x_2 = (L_1^j)^{-1}(x_1)^p$,
\item if $w_j = \epsilon_p^{-1}$, then $\mathcal{B}_j$ is the curve defined by $x_1 = (L_1^j)^{-1}(x_2)^p$,
\item if $w_j = \phi$, then $\mathcal{B}_j$ is the curve defined by $x_2 = f_1^j(x_1)$, and
\item if $w_j = \phi^{-1}$, then $\mathcal{B}_j$ is the curve defined by $x_2 = f_1^{j+1}(x_1)$.
\end{itemize}	
\end{Rk}

\begin{Rk}(See~\cite[Lemma 2.61(3)]{MS})
The correspondence $\mathcal{A}$ encoded by an identity of the form
$w \star [f] = [g]$ is well-defined up to pre- and post-composition with
linear functions corresponding to the choices of representative of the
classes $[f]$ and of $[g]$. Even though
$\mathcal{A}$ is defined in terms of sequences
of generators, it is, in fact, well-defined for
elements of the group.  The verification of this fact
depends on knowing that Ritt swaps satisfy the
braid relations and will be proven with 
Theorem~\ref{thm:braidhtm}.
\end{Rk}

\begin{Rk} \label{rk:encoding}(See~\cite[Proposition 6.19]{MS})
For any pair of non-exceptional polynomials $f$ and $g$,
for any $(f,g)$-skew-invariant curve $\mathcal{A}$
and all long decompositions $\vec{f}$ and
$\vec{g}$ of $f$ and $g$, respectively, there is a
word $w \in \STpk$ so that $\mathcal{A}$ is an
irreducible component
of a curve encoded by the identity $w \star [\vec{f}] =
[\vec{g}]$.	 Recall from Remark~\ref{rk:semidirect}
that we may assume that $w$ has the form
$\prod_p \epsilon_p^{m_p}  \phi^n  u$ where $m_p \in \ZZ$
and all
but finitely many are zero, $n \in \ZZ$, and $u \in \Affk$.

If $\mathcal{A}$ is a skew-twist, then $w$ may be taken to be
$\phi^N u$ for some $N \in \ZZ$ and $u \in \Symk$.  Conversely,
any curve so encoded is a skew-twist.

Indeed, any curve encoded by $v \phi^N u \star \vec{f} = \vec{g}$ with
$v, u \in \Symk$ and $N \in \ZZ$ is  a skew-twist because $v^{-1} \star \vec{g}$ is another
long decomposition of $g$ and $u \star \vec{f}$ is another long
decomposition of $f$.
\end{Rk}

\subsubsection{Walls}
\label{sec:walls}

Let us recall that we wish to characterize
$(f,f^\tau)$-skew-invariant curves where
$f$ is a non-exceptional polynomial and $\tau$
is another field automorphism which commutes with
$\sigma$.   We know by Remark~\ref{rk:encoding}
that any such curve will be a component of a
curve encoded by an identity of the form
$w \star [f] = [f^\tau]$, where we have written
$f$ for a long decomposition of the
polynomial $f$.
 
 Before we take on some of the combinatorial complexity
of the theory of decompositions, let us pause
to work out a simple case of
Theorem~\ref{thm:skew-inv} where one of the
$t_i$'s cannot act.

\begin{Def}
\label{def:wall}	
For $f$ a \kld and an integer $i$
we say that $f$
has a wall at $i$ if
\begin{itemize}
\item for all words $w \in \Affk$, the
action $t_i w \star f$ is not defined
 and
\item for every prime number $p$ neither
$\epsilon_p \star f$ nor $\epsilon_p^{-1} \star f$  is
defined. 	
\end{itemize}
\end{Def}

The easiest way to have a wall in a decomposition is
to have an unswappable factor in the sense of
Definition~\ref{def:Ritt-poly}.

\begin{lem}
\label{lem:unswappable-to-wall}
If $f$ is \kld, $i \in \mathbb{Z}$, and $f_i$ is
unswappable, then $f$ has walls at $i - 1$ and at $i$.	
\end{lem}

\begin{proof}
Let us check that for no $w \in \STk$ is
$t_i w \star f$ defined.  We do this by showing
by induction on the length of a word in the
generators representing $w$ that, if defined,
$w \star f$ has an
unswappable factor at $i$.  This is trivial
for the empty word.  By induction, it
suffices to consider the case that $w = t_j$
for some $j$.  If $j \equiv i$ or $i - 1
\pmod{k}$, then $t_j \star f$ is not defined because
$f_i$ is unswappable.  Otherwise, the $i^\text{th}$
factor of $t_j \star f$ is $f_i$.   From this
we see that neither $t_i w \star f$ nor
$t_{i-1} w \star f$ is defined.

For the second part of the definition of
a wall, observe that
whenever
$\epsilon_p^{\pm} \star f$ is defined, then every
factor of $f$ must be swappable (see Definition~\ref{def:epsilon}), which is not
the case for $f_i$.
\end{proof}

\begin{Rk}
In Subsection~\ref{sec:clusterings}, we will work
with other classes of \kld's with walls.	
\end{Rk}

\begin{Def}
\label{def:wall-set}
A \emph{wall set} for a \kld $f$ is a translate
${\mathcal{W}}$ of $M \ZZ$ for an integer $M$ dividing
$k$ such that $f$ has a wall at $i$ for all
$i \in {\mathcal{W}}$.

A \emph{room} of $({\mathcal{W}}, f)$ consists of the $M$ factors of $f$ between consecutive elements of ${\mathcal{W}}$.
For example, if $k=16$ and $M=4$ and ${\mathcal{W}} = 3 + 4 \ZZ$, one room of $({\mathcal{W}}, f)$ is $(f_7, f_6, f_5, f_4)$.

Two rooms are \emph{different} if they are different
modulo $k$. There are $\frac{k}{M}$ different rooms.
Continuing the above example,
$(f_{23}, f_{22}, f_{21}, f_{20})$ is not a
different room from $(f_7,f_6,f_5,f_4)$,
but $(f_{27}, f_{26}, f_{25}, f_{24})$ is.
\end{Def}

\begin{prop}
\label{prop:separate-room-action}
If ${\mathcal{W}}$ is a wall set for a \kld $f$ with
$R$ rooms and $w \in \Affk$
with $w \star f$ defined, then
$w$ may be represented as $w_R \cdots w_2 w_1$ where
each $w_i$ acts on factors in a different
$i^\text{th}$ room of $({\mathcal{W}}, f)$.
Moreover, $\mathcal{W}$ is still a wall set for
$w \star f$.
\end{prop}

\begin{proof}
If ${\mathcal{W}}$ is a wall set for $f$ and
the Ritt swap $t_i \star f$ is defined,
the two factors swapped are in the same room.

If two Ritt swaps $t_i$ and $t_{j}$ act on factors in
different rooms, the indices $i$ and $j$ cannot be
adjacent modulo $k$, so the Ritt swaps commute.

To complete the proof, present $w$ as a product of
$t_i$'s, use commutation to
collect all of the $t_i$'s with $i$ in the
$j^\text{th}$ room in order to form $w_j$.

That $\mathcal{W}$ is a wall set for $w \star f$
follows from the definition of a wall.
\end{proof}

Recall from Remark~\ref{rk:semidirect}
that conjugation by $\phi$ induces
the automorphism $\eta$ of $\Affk$ given by
$t_i \mapsto t_{i+1}$.

If $i_0+M\ZZ$ is a wall set of $f$ and $w_i$ acts on the $i$th room of $f$,
 then $\eta^{jM}(w_i)$ acts on the $(i+j)^\text{th}$ room.

\begin{prop}
\label{prop:warm-up}
Suppose that $P$ is a polynomial
admitting the decomposition $P = P_k
\circ \cdots \circ P_1$ with associated
long decomposition $\vec{P}$ and that $\vec{P}$
has a wall at $i$.
Then any $(P,P^\tau)$-skew invariant curve is a
skew-twist.	
\end{prop}

\begin{proof}
From Remark~\ref{rk:encoding}
and the fact that no $\epsilon_p^{\pm 1} \star \vec{P}$ is defined,
every
$(P,P^\tau)$-skew-invariant curve is encoded by
an identity of the form $\phi^N w \star \vec{P} =
\vec{P}^\tau$ with $w \in \Affk$ and
$N \in \mathbb{Z}$.  In fact,
because we $\vec{P}$ has a wall at $i$,
$w$ may be expressed as a reduced word
in $\{ t_j ~:~ j \not \equiv i \pmod{k} \}$.

If $i = k$, then curve encoded by this identity is
already a skew-twist and we are done.  So, from now on,
we take $1 \leq i < k$.

We break into two cases depending on whether
$k$ divides $N$ or not.

\vspace{.2in}
\noindent
{\bf Case $N$ is not divisible by $k$:}
Write $N = N' k + L$ with $0 < L < k$.
Let $M := \gcd(k,L)$, let $R := \frac{k}{M}$, and set $\mu := \tau \sigma^{-N'}$. We
work instead with the identity $\phi^L w \star
\vec{P} = \vec{P}^\mu$.      From the
hypothesis that we have a wall at $i$,
it follows that $i + M \mathbb{Z}$ is a wall set for
$\vec{P}$.  Without loss of generality,
$0 < i < M$.  By Proposition~\ref{prop:separate-room-action} we may express $w$ as
a product $w_R \cdots w_1$
where each $w_s$ acts on the $s^\text{th}$ room.
Setting $\tilde{w}_R := \eta^L(w_R)$,
 
we have
$\phi^L w_R \cdots w_1 \star \vec{P} =
\tilde{w}_R \phi^L w_{R-1} \cdots w_1 \star \vec{P}
= \vec{P}^\mu$.   Taking $u$ to be the reverse word
of $\tilde{w}_R$, we have
$\phi^L w_{R-1} \cdots w_1 \star \vec{P} =
u \vec{P}^\mu$ where now
$u$, $w_1$, \ldots, $w_{R-1}$ are all words
omitting $t_k$.
Such an equation encodes a skew-twist curve, as
required.

\vspace{.2in}
\noindent {\bf Case $N$ is divisible by $k$:}
We will show in this case that
$w$ may be expressed without $t_k$. Write $N = k N'$.  Then we have
$\phi^{kN'} w \star \vec{P} = \vec{P}^\tau$
from which we conclude that $w \star \vec{P}
= \phi^{-kN'} \vec{P}^\tau = \vec{P}^{\tau \sigma^{-N'}}$. Write $\mu = \tau \sigma^{-N'}$, we would
then have $w \star \vec{P} = \vec{P}^\mu$.

Let us focus on the behavior of $w \star$ on
the substring $\left( P_{i-1}^\sigma, \ldots, P_1^\sigma, P_k, \ldots P_{i+1}, P_i \right)$.

By Lemma~\ref{lem:improved-2nd-canonical},
we may
find words $u_L^{  +}$, $u_R^{  +}$, $v$, $u_L^{  -}$,
and $u_L^{  +}$ so that
\begin{itemize}
\item $u_L^{  +} u_R^{  +} v u_L^{  -} u_R^{  -} \star \vec{P}
= w \star \vec{P}$,
\item $u_L^{  +}$ and $u_L^{  -}$ are words in
$\{t_j ~:~ 1 \leq j < i \}$,
\item $u_R^{  +}$ and $u_R^{  -}$ are words in
$\{ t_j ~:~ i < j < k \}$,
\item for some $m \leq  \min \{i, k-i\}$, $v \star$
has the effect of swapping the block $\left( P_k,
\ldots, P_{k-m+1} \right)$ with $\left( P^\sigma_{k+m},
\ldots, P_{k+1}^\sigma \right)$.   In terms of our
generators, $v$ takes the form $v = v_{k+1-m}
\cdots v_k$ where where
$v_j = t_{m+j-1} \cdots t_{j}$.     	
\end{itemize}

We will show that $v$ must be trivial so that
$w$ may be expressed without $t_k$.

Observe that because $t_k$ does not appear in
$u_L^{  +} u_L^{  -}$, the result of applying this
word to $\vec{P}$ is another long decomposition of
$P$, which we will temporarily call $\vec{Q}$.
Unless $v$ is trivial, for the action $v \star\vec{Q}$ to be defined,
it is necessary that (possibly replacing
$\vec{Q}$ by a linearly equivalent long decomposition)
either
\begin{itemize}
\item the polynomial $Q_{k+m} \circ
Q_{k+m-1} \circ \cdots \circ Q_{k-m+1}$ is an odd degree
Chebyshev polynomial and for $k-m < j < k < s \leq k+m$
$\deg(Q_j)$ and $\deg(Q_s)$ are distinct prime numbers
or
\item the polynomial $Q_{k+m} \circ \cdots
\circ  Q_{k+1}$ is
a monomial of degree prime to every monomial appearing
amongst the set $\{ Q_j ~:~ k-m < j \leq k \}$ or
\item the polynomial $Q_k \circ \cdots \circ Q_{k-m+1}$
is a monomial of degree prime to every monomial
appearing amongst the set $\{ Q_j ~:~ k < j \leq k+m\}$.
\end{itemize}

Each of these possibilities leads to a contradiction.
Let now $\vec{R} := v \star \vec{Q}$.

In the first case, the degree of the
polynomial $R_k \circ \cdots \circ R_i$ is
different from that of the polynomial
$Q_k \circ \cdots \circ Q_i$.  However,
$\deg (Q_k \circ \cdots \circ Q_i) =
\deg (P_k \circ \cdots \circ P_i)$
and because $\vec{P}^\mu = w \star \vec{P} =
u_L^{  +} u_R^{  +} \star R$ and $u_R^{  +}$ being
a word in $\{ t_j ~:~ i < j < k \}$ does not
affect the resulting polynomial obtained by
composing the terms from $i$ to $k$,
$\deg (R_k \circ \cdots \circ R_i) =
\deg (P_k^\mu \circ \cdots \circ P_i^\mu) =
\deg (P_k \circ \cdots \circ P_i)$.

The other two cases are handled similarly noting
that the number of instances of a given monomial
appearing in the block from $i$ to $k$ is also
invariant.  With this, we end the proof of the
proposition in the case that
$k$ divides $N$.  Notice that in this case,
the plain skew-twist is given by a skew-compositional
power of $P$.

\end{proof}

\begin{cor}
\label{cor:warm-up}
If $P$ is a polynomial having an unswappable
compositional factor, then Theorem~\ref{thm:skew-inv}
holds for $P$.	
\end{cor}
\begin{proof}
By Lemma~\ref{lem:unswappable-to-wall} any decomposition
of $P$ has a wall.  By Proposition~\ref{prop:warm-up}
any $(P,P^\tau)$-skew-invariant curve is a
skew-twist.	
\end{proof}

\subsection{in/out degrees}\label{sec:inout}

In this section, we focus on the use of the in-degree and 
out-degree functions as a way to study compositional 
identities.  A highlight of this section is the proof of Theorem~\ref{thm:skew-inv} for the case when the polynomial $P$ has a clean $\mathsf{C}$-free long decomposition $f$.  

\begin{Def}\label{def:clean-c-free}
A long decomposition $f$ is a \emph{clean $\mathsf{C}$-free long decomposition} if all factors $f_i$ are monomials or Ritt polynomials, at least one factor is not a monomial, and all of factors are
$\mathsf{C}$-free. (See Definition~\ref{def:Ritt-poly}.)
\end{Def}

Let us record a easy and useful fact about
linear equivalence between clean $\mathsf{C}$-free
long decompositions.

\begin{lem}\label{lem:clean-scaling}
Linear equivalence between two clean $\mathsf{C}$-free long decompositions is always witnessed by scalings.

\end{lem}
\begin{proof}
Suppose that $f$ is a clean $\mathsf{C}$-free
long decomposition and $L$ is a \kspl and $g$ is
a clean $\mathsf{C}$-free long decomposition linearly
related to $f$ via $L$. By definition of linear equivalence, we have $g_i = L_{i+1}^{-1} \circ f_i \circ L_i$ for all $i$. By~\cite[Theorem 3.15]{MS}, this forces $L_{i+1}$ to be a scaling.
\end{proof}

As in the other cases of Theorem 1.3, a $(P, P^\tau)$ skew-invariant curve $\mathcal{C}$ is encoded by 
$u \star f = f^\tau$ for some $u \in \STpk$.  
To show that $\mathcal{C}$ is a  skew-twist curve, we 
show in Lemma~\ref{lem:clean-no-epsilons}
that we may take $u \in \STk$ so that 
$\phi^N w \star f = f^\tau$ for some $w \in \Affk$. 
We will show that the curve $\mathcal{C}$ is also encoded by $u_1 \phi^N u_2 \star f = f^\tau$ for some $u_1, u_2 \in \Symk$.

The key idea of the proof in this case is to associate
certain combinatorial data, which we call
\emph{nomodata} to $f$ and then
to find another long decomposition $g$ of the same
polynomial for which this nomodata is $e$-
periodic for a certain special $e$.

\begin{Def}\label{def:magic-e}
Let $f$ be a clean $\mathsf{C}$-free long decomposition, and let $w \in \STk$ be such that $w \star f = f^\tau$ encodes a curve $\mathcal{C}$ parametrized by $(A(t), B(t))$ for some polynomials $A$ and $B$.

Let $n_A$, $n_B$, and $n_f$ be the number of non-monomial factors in $A$, $B$, and one
$k$-period of $f$, respectively.

The \emph{echo of $w \star f = f^\tau$} is $e := \gcd(n_f, (n_B - n_A))$.
\end{Def}

In Section~\ref{sec:scaffolding}, we set up and analyze the nomodata that is used throughout this section.
In Section~\ref{sec:coho}, we find some $v \in \Affk$ for which $v \star f$ has $e$-periodic nomodata.
Then in Section~\ref{sec:removing-tk}, we show that the same can be accomplished with some $u \in \Symk$.
Then $g := u \star f$ is the desired other long decomposition of the same polynomial, with $e$-periodic nomodata.
The skew-invariant curve $\mathcal{C}$ is now also encoded by $\phi^N w' \star g = g^\tau$ for some $w' \in \Affk$.
We observe (Lemma~\ref{lem:is-traceless}) that $w'$ cannot change the nomodata of $g$, and it will follow that all it can do is permute adjacent monomials. The proof ends just like the proof in the ``walls'' case, with ``puddles'' of monomials between non-monomials replacing the ``rooms'' between walls.

\subsubsection{Scaffolding}\label{sec:scaffolding}

In this subsection we identify and analyze the combinatorial data of a clean $\mathsf{C}$-free long decomposition $f$ in full generality, that is, without assuming that $f$ has anything to do with the Mahler problem.  

With the following definition, we recall the definitions
of in-degree and out-degree from 
Definition~\ref{def:inoutdeg-1st} and then decompose these
into their prime components.

\begin{Def}\label{def:inoutdeg}
Fix a polynomial $u$ with no initial nor terminal, compositional nor multiplicative monomial factors. Recall that the in-degree and out-degree of a Ritt polynomial $g(x) = x^k u(x^\ell)^n$ are
 $$\operatorname{in-deg}(g) := \ell \mbox{ and } \operatorname{out-deg}(g) := n.$$

We decompose the in- and out-degrees into their prime components by the formulae
$$\operatorname{in-deg}(g) = \prod_{p \in \mathcal{P}} p^{\indeg(g,p)}$$
and $$\operatorname{out-deg}(g) = \prod_{p \in \mathcal{P}} p^{\outdeg(g,p)}$$
where the products are taken over the set $\mathcal{P}$ all of the prime numbers.
\end{Def}

The condition on $u$ just means that $u$ has a non-zero constant coefficient, is not a power of another polynomial, and cannot be written as $u(x) = v(x^p)$ for any polynomial $v$ and any integer $p \geq 2$.

Each non-monomial Ritt polynomial has a unique representation of this form, so in-degree and out-degree are indeed properties of the polynomial itself.

These integers $k, \ell, n$ are the same for two scaling-related Ritt polynomials. 
Thus, by Lemma~\ref{lem:clean-scaling} the ``in-degree of $f_i$'' (respectively, ``out-degree of $f_i$'') is a property of the linear-equivalence class of a clean $\mathsf{C}$-free long decomposition $f$.

\begin{Def}\label{def:nomolist}
An increasing function $\alpha: \mathbb{Z} \to \mathbb{Z}$ is a \emph{non-monomial listing function}
for a clean $\mathsf{C}$-free long decomposition $f$ if
\begin{itemize}
 \item for each $j \in \mathbb{Z}$, $f_{\alpha(j)}$ is \emph{not} a monomial, and
 \item for each $i \in \mathbb{Z}$, if $f_i$ is not a monomial, then $i$ is in the range of $\alpha$.
\end{itemize}

A \emph{nomodata} for a clean $\mathsf{C}$-free long decomposition $f$ is a quadruple $(\alpha,\ins,\outs,\pud)$
where \begin{itemize}
 \item $\alpha$ is a non-monomial listing function for $f$;
 \item $\ins : \mathbb{Z} \times \mathcal{P} \rightarrow \mathbb{N}$ is defined by $\ins(j,p) := \indeg(f_{\alpha(j)},p)$;
 \item $\outs : \mathbb{Z} \times \mathcal{P} \rightarrow \mathbb{N}$ is defined by $\outs(j,p) := \outdeg(f_{\alpha(j)},p)$; and
 \item $\pud : \mathbb{Z} \times \mathcal{P} \rightarrow \mathbb{N}$ is defined by
$$\pud(j,p) := \# \{ i \in (\alpha(j-1), \alpha(j)) ~:~ f_i = x^p \}.$$
\end{itemize}
\end{Def}

\begin{Rk}\label{rk:puddles}
The last part $\pud$ of the nomodata of $f$ keeps track of the monomials $f_i$ with $\alpha(j-1) < i < \alpha(j)$:
$$\prod_{i = \alpha(j-1) +1}^{\alpha(j)-1} \deg(f_i) =
 \prod_{p \in \mathcal{P}} p^{\pud(j,p)}.$$
We informally say that these monomials are the $j$th puddle, viewing the long decomposition $f$ as a sequence of alternating non-monomials and puddles. These puddles may be empty.
\end{Rk}

\begin{Rk}
The components $\ins$, $\outs$, and $\pud$ of
a nomodata are two-variable functions of $j$, an
integer, and $p$ a prime.  Sometimes, we will drop
$p$ from the notation regarding them simply as
functions of the variable $j$. 	
\end{Rk}

\begin{Rk}\label{rk:nomodata-shift-unique}
If $(\alpha(j),\ins(j,p),\outs(j,p),\pud(j,p))$ is a nomodata for $f$, then
 $(\alpha(j+m),\ins(j+m, p),\outs(j+m,p),\pud(j+m,p))$
  is another nomodata for the same $f$ for any $m \in \ZZ$.  Conversely,
 any two non-monomial listing functions $\alpha$
 and $\beta$ for $f$ are
 related by a shift: $\beta(j) = \alpha(j+m)$ for some
 $m \in \ZZ$. The nomodata for $f$ is
 determined by the non-monomial listing function.

If there are exactly $m$ non-monomials in one $k$-period of $f$ (that is, amongst $f_1, \ldots, f_k$), then $\alpha(j+m) = \alpha(j) + k$ for all $j$; and the three degree functions are $m$-periodic.

\end{Rk}

\begin{Rk} \label{rk:phi-nomo}
If $(\alpha, \ins, \outs, \pud)$ is a nomodata for
 $h$, then $(\alpha-N, \ins, \outs, \pud)$ is a nomodata for $\phi^N \star h$.
\end{Rk}

\begin{Rk}
If $\tau$ is a field automorphism, then $f$ and $f^\tau$ obtained from $f$ by applying $\tau$ to all the coefficients of $f$ have the same nomodata.
\end{Rk}

\begin{Rk}
The index $j$ of a non-monomial factor $f_{\alpha(j)}$ is often the more reliable counter than the index $i$ of the factor $f_i$. For example, the periodicity of nomodata is better stated in terms of $j$.
\end{Rk}

The key feature of the nomodata is that it completely determines whether $t_i \star f$ is defined and the nomodata of $t_i \star f$ can be read off from the nomodata of $f$. The following trace function keeps track of this. 

\begin{Def} \label{def:trace-nomodata}
Suppose that $(\alpha, \ins, \outs, \pud)$ is a nomodata for a long decomposition $f$, and suppose that $t_i \star f$ is defined. The \emph{trace of $t_i$ on $(f,\alpha)$}, $\trace_{t_i,f,\alpha}$, is the following function $\lambda: \ZZ \times \mathcal{P} \rightarrow \ZZ$.
\begin{itemize}
 \item If $i$ and $i+1$ are not in the range of $\alpha$,\\ the trace $\lambda$ is the constant function zero.
 \item If $i = \alpha(j)$ and $f_{i+1}$ is the monomial $x^p$,\\
    then $\lambda(x,p) = 1$ for $x \equiv j (\mod m)$, and zero for all other inputs.
 \item If $i+1 = \alpha(j)$ and $f_{i}$ is the monomial $x^p$,\\
    then $\lambda(x,p) = -1$ for $x \equiv j (\mod m)$, and zero for all other inputs.
\end{itemize}
If $w$ is a word in $t_i$ with $w \star f$ defined, the \emph{trace of $w$ on $(f,\alpha)$} is the recursive sum:
 $$\trace_{u t_i, f, \alpha} := \trace_{u, t_i \star f, \beta} + \trace_{t_i, f, \alpha} \text{, }$$
 where $\beta = \alpha + \sum_{p \in \mathcal{P}}\trace_{t_i,f,\alpha}$.

In general, we write $w \star (f,\alpha) = (g,\beta)$
when $g = w \star f$ and $\beta = \alpha + \sum_{p \in \mathcal{P}}\trace_{w,f,\alpha}$.  This
notation is justified by Lemma~\ref{lem:starnomodata}
below.
\end{Def}

\begin{Rk}
It is easy to see that the trace of $w$ on $(f,\alpha)$
 is well-defined for $w \in \Affk$, independent of its
 representation as a product of generators.
\end{Rk}

\begin{lem} \label{lem:starnomodata}
Suppose that $(\alpha, \ins_{f, \alpha}, \outs_{f, \alpha}, \pud_{f, \alpha})$ is a nomodata for a \kld $f$;
 that $w \in \Affk$ with $w \star f$ defined;
 that $\lambda = \trace_{w,f,\alpha}$ is the trace of $w$ on $(f,\alpha)$.
Then a nomodata for $g := w \star f$ can be computed as follows.
 The function $$\beta(j) := \alpha(j) + \sum_p \lambda(j,p)$$ identifies the non-monomials of $g$.

$$\ins_{g, \beta}(j,p) = \ins_{f, \alpha}(j,p) - \lambda(j,p) \text{, }$$
$$\outs_{g, \beta}(j,p) = \outs_{f, \alpha}(j,p) + \lambda(j,p) \text{, and}$$
$$\pud_{g, \beta}(j,p) = \pud_{f, \alpha}(j,p) + \lambda(j,p) - \lambda(j-1,p) \text{ .}$$
\end{lem}

\begin{proof}
 This easy induction on the length of a word
representing $w$ is left to the reader.
  \end{proof}
 
\begin{Def} \label{def:admissi}
The \emph{admissibility conditions} for an $m$-periodic function $\lambda: \ZZ \times \mathcal{P} \rightarrow \ZZ$ to be a trace of a word acting on a \kld $f$ with $m$ non-monomials per $k$-period and with nomodata $(\alpha, in, out, pud)$ are\begin{enumerate}
\item $-\outs(j,p) \leq \lambda(j,p) \leq in(j, p)$ for all $j$ and $p$;
\item $\pud(j,p) + \lambda(j,p) - \lambda(x-1,p) \geq 0$ for all $j$ and $p$;
\item for each $p$, if $\pud(j,p) = 0$ for all $j$, then $\lambda(j,p) = 0$ for all $j$.
\end{enumerate}
\end{Def}

The first two admissibility conditions are necessary because the functions $\ins_{f,\alpha} - \lambda$,
$\outs_{f,\alpha} + \lambda$, and $\pud_{f,\alpha}(j,p) + \lambda(j,p) - \lambda(j-1,p)$ are
parts of the nomodata for $w \star (f, \alpha)$, so they must be non-negative. The last is also clearly necessary: if a monomial $x^p$ is not available anywhere in $f$,
 actions by $\Affk$ cannot change the $p$-parts of in- and out-degrees of non-monomials in $f$. The main result of this subsection is Theorem~\ref{thm:shinythm} showing that these admissibility conditions are also sufficient: given $f$ with nomodata $(\alpha, \ins, \outs, \pud)$ and a function $\lambda$ satisfying the admissibility conditions, we shall produce a word $w$ with $w \star f$ defined and with trace $\lambda$. We will build this word out of long-swaps, each of which moves its monomial through a puddle and then swaps it with a non-monomial.

\begin{Def}
\label{def:puddle-swap}
If $i$ and $i+1$ are not in the range of $\alpha$, then $t_i$ is a \emph{puddle-swap} for $(f, \alpha)$.
A \emph{long-swap} for $(f, \alpha)$ is a word of the form
 $$t_{\alpha(j)} t_{\alpha(j)+1} \ldots t_{i-1} t_{i}\text{ for }i \leq \alpha(j+1) \text{; or }
 t_{\alpha(j)-1} t_{\alpha(j)-2} \ldots t_i\text{ for }i > \alpha(j-1) \text{ .}$$
The first option above is a left-to-right transit (see Definition~\ref{def:transit}) that brings a monomial through the puddle to the left of $\alpha(j)$, followed by the Ritt swap of that monomial with the non-monomial $f_{\alpha(j)}$.

If $x^p$ is a monomial that swaps with the non-monomial when a long-swap $w$ acts, we call $w$ a \emph{$p$-long-swap}.
\end{Def}

\begin{Rk}  \label{rk:puddle-swap}
\begin{enumerate}
\item  The concepts in Definition~\ref{def:puddle-swap}
 only depend on the range of $\alpha$.  Thus,
they are the same for different nomodata for the same
$f$.
\item Puddle-swaps in different puddles commute.
\item Suppose that $w$ is a word in puddle-swaps for $(f,\alpha)$, and that $f$ has $r$ puddles per period.
 Then $w \approx w_r \ldots w_1$ where each $w_j$ is a word in puddle-swaps in the $j$th puddle; and these $w_j$ commute with each other.
 \item If $w$ is a word in puddle-swaps for $(f,\alpha)$, then $\trace_{w,f,\alpha}$ is identically zero.
Lemma~\ref{lem:easy-step-1} provides a converse.
\end{enumerate}
\end{Rk}

\begin{lem}
\label{lem:easy-step-1}
Suppose that $w$ is a reduced word for which
$w \star f$ is defined and $\trace_{w,f,\alpha}$ is identically zero.  Then
$w$ is a product of puddleswaps for
$(f,\alpha)$.
\end{lem}

\begin{proof}
Fix a non-monomial $f_{\alpha(j)}$. If suffices to show that $t_{\alpha(j)}$ and $t_{(\alpha(j)+1)}$ do not occur in $w$.

Let $w_x$ be the initial segment of $w$ of length $x$, so $w_x \star (f, \alpha)$ is an intermediate decomposition that occurs as $w \star f$ is carried out generator by generator. Let $I(x)$ be in-degree of the $j^\text{th}$ non-monomial of $w_x \star (f, \alpha)$. If $I(x)$ is constant in $x$, we are done. Otherwise, there is some prime $p$ for which the $p$-part of $I$ changes at some points. We work out the case that it increases at some point; the
other case is symmetric. Let $y$ and $z$ be such that the $p$-part of $I(x)$ attains its maximal value for $y < x <z$ but not at $y$ or $z$; and let $u := w_{[z,y]}$ be a middle subword of $w$.

Let us pause and regroup. We now work with $(g, \beta) := w_{y-1} \star (f, \alpha)$ and the word $u$, focusing on the non-monomial factor $g_{\beta(j)}$ and the monomial $x^p$. Since the first thing that happens as $u \star g$ is evaluated is the $p$-part of the in-degree of $g_{\beta(j)}$ increasing, it follows that $w_y$ (the first letter of $u$) is $t_{\alpha(j)-1}$ and $g_{\alpha(j)-1}$ is the monomial of degree $p$.

For $x := y, y+1, \ldots, z$, let $N(x)$ and $M(x)$ be the positions of the $j^\text{th}$ nonmonomial and this particular degree-$p$ monomial in $u_x \star f$, where $u_x = w_{[x,y]}$. Because the $p$-part of the the in-degree of the non-monomial does not change again, for $x = y+1, \ldots, z-1$ we always have $M(x) > N(x)$, and there are no degree-$p$ monomials between these. Because at the end, the $p$-part of the the in-degree of the non-monomial decreases, we must have a degree-$p$ monomial immediately to the left of this non-monomial at the end, so $M(z-1) = N(z-1)+1$.

Case 1: If this degree-$p$ monomial does not cross any other non-monomials as $u \star g$ is carried out, then it stays in the $(j+1)^\text{st}$ puddle of all intermediate decompositions. We first rewrite $u$, keeping the same length, so that it stays at the right edge of that puddle, except to let through other-degree monomials who cross our non-monomial. We then rewrite $u$, decreasing its length by $2$, so that this degree-$p$ monomial stays on the right, rather than on the left, of the non-monomial throughout. The decrease in length is a contradiction.

Case 2: If this degree-$p$ monomial does cross other non-monomials, it must in particular cross $g_{\alpha(j+1)}$. Consider the middle subword $v$ of $u$ which begins with this degree-$p$ monomial crossing $g_{\alpha(j+1)}$ right-to-left and ends with it crossing back. This $v$ is shorter than $u$, and satisfies the same hypotheses.

\end{proof}

\begin{lem}
\label{lem:get-longswap}
Suppose that
 $\lambda: \ZZ \times \mathcal{P} \rightarrow \ZZ$
looks like the trace of one Ritt swap in the sense
that there is exactly one pair $(j,p)$ with
$\lambda(j,p) = \pm 1$ and $\lambda(j',p') = 0$ for
all other pairs.
Suppose that $\lambda$ satisfies the admissibility
conditions for some $(f,\alpha)$.
Then there is a long-swap $v$ with trace $\lambda$ on $(f, \alpha)$.	
\end{lem}
\begin{proof}
We work out the case that $\lambda(j,p) = 1$.  The
other case is handled symmetrically.

By the second admissibility condition, there is a
monomial $x^p$ in the puddle to the left of $f_{\alpha(j)}$, say $x_p = f_i$ with $\alpha(j+1) > i > \alpha(j)$, and $i$ is minimal such.
The transit $w =    t_{\alpha(j)+1} \cdots t_{i-1}$
is a word in puddle-swaps for $(f,\alpha)$.  By
the first admissibility condition,
$\indeg(f_{\alpha(j)},p) \geq \lambda(j,p) = 1$.
Thus, $t_{\alpha(j)} w \star f$ is defined and
has trace $\lambda$ on $(f,\alpha)$.
\end{proof}

\begin{lem} \label{lem:glue-longswaps}
Suppose that $w_i \star f$ is defined for $i = 1,2$; that $w_i$ is a $p_i$-longswap for $f$; and that $p_1 \neq p_2$.
Then there is a word $v_2$ such that $v_2 w_1 \star f$ is defined, and $v_2$ is a $p_2$-longswap for $w_1 \star f$, and the trace of $v_2$ on $w_1 \star f$ is the same as the trace of $w_2$ on $f$. \end{lem}

\begin{proof}
If the factors moved by $w_1$ and $w_2$ are disjoint,
then $w_1$ and $w_2$ already commute.  Let us
consider a case where $w_1$ and
$w_2$ affect the same puddle but the nonmonomials on
opposite ends of that puddle.  Here we treat the
subcase where $w_1$ is right-to-left and $w_2$ is
left-to-right.  The other subcase is handled similarly.
So, $w_1 = t_{\alpha(j)-1} \cdots t_{a+1} t_a$ and
$w_2 = t_{\alpha(j-1)} \cdots t_{b-1} t_{b}$ for some
$j$, and $\alpha(j) > b+1 > a > \alpha(j-1)$.
Set $v_2 := t_{\alpha(j-1)} \cdots t_{b-1}$ ; so that $w_2 = v_2 t_b$.

The other cases where $w_1$ and $w_2$ affect the same non-monomial are equally tedious and straightforward.
\end{proof}

\begin{thm} \label{thm:shinythm}
Fix $(f, \alpha)$ with $m$ non-monomials per period, and suppose that $\lambda : \ZZ \times \mathcal{P} \rightarrow \ZZ$ is an $m$-periodic function satisfying the admissibility conditions (see Definition~\ref{def:admissi}).
Then there is some $w \in \Affk$  such that $w \star (f, \alpha)$ is defined, and the trace of $w$ on $(f, \alpha)$ is $\lambda$.
\end{thm}

\begin{proof}
The third condition in Definition~\ref{def:admissi} implies that $\lambda$ is supported on finitely many primes $p$: those with $f_i = x^p$ for some $i$. We first produce words $w_p$ for each such prime $p$, and then put them together, via Lemma~\ref{lem:glue-longswaps}.
 
Let $N := \sum_{j=1}^m |\lambda(j,p)|$. This is the total number of monomials $x^p$ that need to cross non-monomials, per period. The word $w_p$ that will be a sequence of $N$ $p$-longswaps. We induct on $N$ to build the word $w_p$, adding a longswap to decrement $N$. More precisely, we shall find a word $u$ which is a $p$-longswap for $f$, with trace $\mu$ on $(f, \alpha)$, such that $\lambda - \mu$ has a smaller $N$. The only hard part is finding a starting point where the monomial $x^p$ is available in the appropriate puddle.

\textit{Case 0.} If $\lambda(j,p) = 0$ for all $j$, then $N = 0$ and the induction is over.

\textit{Case 1.} If for some $j$ we have $\lambda(j,p) \leq 0$ and $\lambda(j-1, p) \geq 0$ and they are not both zero, then $\pud_{f, \alpha}(j,p) + \lambda(j,p) - \lambda(j-1, p) \geq 0$ gives
 $\pud_{f, \alpha}(j,p) \geq - \lambda(j,p) + \lambda(j-1, p) > 0$.
The hypothesis of this Case says that the $j^\text{th}$ puddle is not supposed to gain any monomials $x^p$ on either side, and is supposed to lose at least one monomial $x^p$ on at least one side; and we just showed that there is at least one monomial $x^p$ in this $j$th puddle. Lemma~\ref{lem:get-longswap} now gives the desired longswap $u$.

\textit{Case 2 [resp. 3].} If $\lambda(j,p) > 0$ for all $j$ [resp., $\lambda(j,p) < 0$ for all $j$ ], fix $j$ such that the $j^\text{th}$ puddle contains at least one monomial $x^p$; this is possible by the assumption of Step 2. Use Lemma~\ref{lem:get-longswap} to get a longswap $u$ that takes this monomial out of this puddle on the right, decreasing the positive $\lambda(j+1, p)$ [resp.,on the left, inreasing the negative $\lambda(j, p)$].

That covers all possible cases, by periodicity: $\lambda$ cannot force us to only remove monomials $x^p$ from every puddle.
\end{proof}

\subsubsection{Building the trace function} \label{sec:coho}

We now return to analyzing a clean $\mathsf{C}$-free long decomposition $f$ that fits into the Mahler problem.
Our goal for this section, accomplished in Theorem~\ref{thm:cohothm} at the end, is to find some $v \in \Affk$ for which $v \star f$ is defined and has $e$-periodic nomodata. With Theorem~\ref{thm:shinythm}, it suffices to find the trace $\lambda$ of such a word, accomplished in Proposition~\ref{prop:cohoprop} right before the main theorem.

We first need to connect the combinatorics from Section~\ref{sec:scaffolding} to the echo $e$.
Suppose that $w \star f = f^\tau$ for some $w \in \STk$, that $\alpha$ is a non-monomial listing function for $f$, and that $\lambda$ is the trace of $w$ on $(f,\alpha)$. By Lemma~\ref{lem:starnomodata}, $\beta(j) := \alpha(j) + \sum_p \lambda(j,p)$ is a nomodata for $f^\tau$. Since $\alpha$ is also a
nonmonomial listing function  for $f^\tau$, we must have $\beta(j) = \alpha(j+r)$ for some integer $r$. The next lemma identifies this $r$ explicitly.  A
 proof of Lemma~\ref{lem:starnomodata} could follow
 this template.

\begin{lem}
\label{lem:na-nb-ih}
Let $f$ be a clean
$\mathsf{C}$-free long decomposition, and let $w \in \STpk$ be such that $w \star f = g$ encodes a curve parametrized by $\{ (A(t), B(t)) \}$ for some polynomials $A$ and $B$.
Let $n_A$ and $n_B$ be the number of non-monomial factors in $A$ and $B$, respectively.
Let $\alpha$ and $\beta$ be increasing listings of the indices on non-monomial factors of $f$ and $g$, respectively. Suppose that $\alpha(0)$ and $\beta(0)$ are the least positive indices of non-monomials in the $f$ and $g$, respectively.
If the factor $f^\tau_b$ corresponds to $f_{\alpha(j)}$ via $w \star f  = g$,
 then $b = \beta(j+ (n_B - n_A) )$.
 \end{lem}

\begin{proof}
We argue by induction on the length of a word
representing $w$. As is our wont, we will identify
$w$ with this word.

 The base case where $w$ is empty is trivial. The induction step where $w$ is a single generator is verified below. The conclusion is obviously additive, which concludes the proof.
We organize the long case-out in order
of increasing complexity of the argument.

Let us start with the cases where $\beta = \alpha$
works. If $w = \epsilon_p$ or $w = \epsilon_p^{-1}$, one of $A$ and $B$ is identity and the other is a monomial, so $n_B - n_A =0$.
If $w = t_k$ and both $f_k$ and $f_{k+1}$ are monomials, both $A$ and $B$ are monomials, so $n_B - n_A = 0$.
In both cases, $\beta := \alpha$ works for $g$, and $g_{\alpha(j)}$ comes from $f_{\alpha(j)}$ via $w \star f = g$ for every $j$.

We next consider the cases where the non-monomials
do move, but do not cross between $k$ and $k+1$.
If $w = t_i$ for some $i \neq k$, both $A$ and $B$ are identity, so $n_B - n_A =0$. Now $\beta := \pi \circ \alpha$ is the desired non-monomial listing function for $g$, where $\pi$ is the permutation of $\mathbb{Z}$ associated with $t_i$, that is, swapping $i+kn$ with $i+1 +kn$ for all $n \in \mathbb{Z}$.
If $w = \phi$ and $f_1$ is a monomial, then $A$ is identity and $B$ a monomial, so $n_B - n_A =0$.
Now $\beta(j) := \alpha(j)-1$ is a non-monomial listing function for $g$, and $\alpha(0)-1 > 0$ since $\alpha(0) \neq 1$ since $f_1$ is a monomial. The argument for $\phi^{-1}$ and monomial
$f_k$ is symmetric.

Finally, we consider the cases when a non-monomial
crosses between $k$ and $k+1$. If $w = \phi$ and $f_1$ is not a monomial, then $A$ is identity and $B$ not a monomial, so $n_B - n_A = 1$.
Now $\alpha(0) = 1$, so while $\alpha(j)-1$ is a non-monomial listing function for $g$, it does not satisfy our extra requirement; but $\beta(j) := \alpha(j+1)-1$ does.  The argument for $\phi^{-1}$ is symmetric.
If $w = t_k$ and $f_1$ is not a monomial, then $A$ is a monomial and $B$ not a monomial, so $n_B - n_A = 1$.
Now $\alpha(0) = 1$, and $g_1$ (which comes from $f_0$) is a monomial, so $\beta(0)$ must be greater than $1$. In particular, $g_{\beta(0)}$ does not correspond to $f_{\alpha(0)}$, but to the next non-monomial to the left of it, namely $f_\alpha(1)$, as wanted. The case where $f_k$ rather than $f_1$ is not a monomial is symmetric.

\end{proof}

The special case of Lemma~\ref{lem:na-nb-ih}
where $g = f^\tau$ is a key step in
solving the Mahler problem.

\begin{cor} \label{cor:na-nb}
Let $f$ be a clean $\mathsf{C}$-free long decomposition, and let $w \in \STpk$ be such that $w \star f = f^\tau$ encodes a curve parametrized by $(A(t), B(t))$ for some polynomials $A$ and $B$.
Let $n_A$ and $n_B$ be the number of non-monomial factors in $A$ and $B$, respectively.
Let $\alpha$ be non-monomial listing function for $f$.

If the factor $f^\tau_b$ corresponds to $f_{\alpha(j)}$ via $w \star f  = f^\tau$,
 then $b = \alpha(j+ (n_B - n_A) )$.
\end{cor}

\begin{prop} \label{prop:khoa-mah-1}
Suppose that $(\alpha, \ins, \outs, \pud)$ is a nomodata for $f$, that
 $w \in \Affk$ has trace $\mu = \trace_{w,f,\alpha}$,
 and that
 $\phi^N w \star f = f^\tau$ encodes a curve parametrized by $(A(t), B(t))$ for some polynomials $A$ and $B$.
Let $n_A$ and $n_B$ be the number of non-monomial factors in $A$ and $B$, respectively.
Set $r := n_A - n_B$.

Then for all $j$,
$$\ins(j) - \mu(j) = \ins(j+r)$$
$$\outs(j) + \mu(j) = \outs(j+r)$$
$$\pud(j) + \mu(j) - \mu(j-1) = \pud(j+r) \text{ .}$$
\end{prop}

\begin{proof}
Lemma~\ref{lem:starnomodata} gives a nomodata for $(g, \beta) := w \star (f, \alpha)$ as
$$\beta(j) := \alpha(j) + \sum_p \mu(j,p) \text{, }$$
$$\ins_{g, \beta} = \ins_{f, \alpha} - \mu \text{, }$$
$$\outs_{g, \beta} = \outs_{f, \alpha} + \mu,$$
$$\pud_{g, \beta}(j) = \pud_{f, \alpha}(j) + \mu(j) - \mu(j-1) \text{ .}$$

 By Remark~\ref{rk:phi-nomo}, a nomodata for $\phi^N w \star f$  is given by $\beta - N$ with the same $\ins$, $\outs$, and $\pud$.
Since $\phi^N w \star f = f^\tau$, this is also a nomodata for $f^\tau$, and so it is another nomodata for $f$. By Remark~\ref{rk:nomodata-shift-unique}, we get some $s$ with
$$\alpha(j) + \sum_p \mu(j,p) - N = \alpha(j+s)$$
$$\ins(j) - \mu(j) = \ins(j+s)$$
$$\outs(j) + \mu(j) = \outs(j+s)$$
$$\pud(j) + \mu(j) - \mu(j-1) = \pud(j+s)$$

Corollary~\ref{cor:na-nb} shows that $s = r$.
\end{proof}

\begin{Rk}
\label{rk:sum-periodic}
With the setup in Proposition~\ref{prop:khoa-mah-1}, it follows that the sum $\ins(j) + \outs(j)$ is $r$-periodic:
$$\ins(j+r) + \outs(j+r) = \ins(j) - \mu(j) + \outs(j) + \mu(j) = \ins(j) + \outs(j) \text{ .}$$

For those primes $p$ whose monomials $x^p$ do not appear in $f$, we must have $\mu(j, p) = 0$ for all $j$, so $\ins(j, p)$ and $\outs(j, p)$ are also $r$-periodic.

Suppose that $f$ has $n_f$ nonmonomials per $k$-period.
Anything that is both $n_f$-periodic and $r$-periodic
is also $e$-periodic, where $e$ is the echo
$\gcd(n_f,r)$.
\end{Rk}

\begin{lem} \label{lem:lambda-quest}
Fix some $f$ with $m$ non-monomials per $k$-period.
Suppose that for some integer $r$, a nomodata $(\alpha, \ins, \outs, \pud)$ for $f$ satisfies
$\ins(j+r) + \outs(j+r) = \ins(j) + \outs(j)$ for all $j$.  Suppose that the trace of $u \star (f, \alpha) =: (g, \beta)$ is $\lambda := \trace_{u, f, \alpha}$ for some $u \in \Affk$.
 Then the following are equivalent:\begin{enumerate}
\item $\ins_{g, \beta}(j+r,p) = \ins_{g, \beta}(j,p)$ for all $j$ and prime $p$, that is, $\ins_{g, \beta}$ is $r$-periodic in $j$;
\item $\outs_{g, \beta}(j+r,p) = \outs_{g, \beta}(j,p)$ for all $j$ and prime $p$, that is, $\outs_{g, \beta}$ is $r$-periodic in $j$;
\item for each prime $p$ there are constants $C_i(p)$ for $i = 1, 2, \ldots d := \gcd(m,r)$ such that
 $\lambda(j,p) = \ins(j,p) + C_i(p)$ for all $j \equiv i \pmod{d}$.
\end{enumerate} \end{lem}

\begin{proof}
Ritt swaps do not change the sum of in- and out-degree, so these sums are still $r$-periodic. That is,
$$\ins_{g, \beta}(j+r) + \outs_{g, \beta}(j+r) = \ins_{g, \beta}(j) + \outs_{g, \beta}(j)$$
for all $j$. It follows that (1) and (2) are equivalent.

By Lemma~\ref{lem:starnomodata}, $\ins_{g, \beta} = \ins - \lambda$, so (1) is equivalent to
$\ins - \lambda$ being $r$-periodic. Since it is already $m$-periodic, it is $r$-periodic if and only if it is $d$-periodic. The constants $C_i(p)$ in (3) are the values of $\ins(j,p) - \lambda(j,p)$ for $j \in i + d \ZZ$.
\end{proof}

\begin{lem}\label{lem:cheap-periodic}
Let $f$ be a clean $\mathsf{C}$-free long decomposition with nomodata $(\alpha, \ins, \outs, \pud)$.
Let $e$ be the echo of $w \star f = f^\tau$ for some $w \in \STk$.	
Let $p$ be a prime for which the monomial $x^p$ does not appear in $f$.
Then the $p$-part of the nomodata for $f$ is $e$-periodic.
That is, $\ins(j,p)$, $\outs(j,p)$, and $\pud(j,p)$ are $e$-periodic in $j$.
\end{lem}
\begin{proof}
By Corollary~\ref{cor:na-nb}, a non-monomial $f_{\alpha(j)}$ corresponds to $f^\tau_{\alpha(j+(n_B -n_A))}$ via $w \star f = f^\tau$. Since $x^p$ does not occur in $f$, the $p$-part of the in-degree of this non-monomial does not change as $w \star f$ is carried out. So
$\ins(j,p) = \indeg(f_{\alpha(j)},p) = \indeg(f^\tau_{\alpha(j+(n_B -n_A))},p) =
\indeg(f_{\alpha(j+(n_B -n_A))},p) = \ins(j+(n_B -n_A), p)$.
That is, the function $\ins(j,p)$ is $(n_B -n_A)$-periodic in $j$. Of course, it is also $n_f$-periodic in $j$, so it is $e$-periodic.
The exact same argument works for the out-degree. Since $x^p$ does not appear in $f$, we get $\pud(j,p) = 0$ for all $j$, and constant functions are as periodic as it gets.\end{proof}

\begin{lem}
\label{lem:coho-step-1}
Let $f$ be a clean $\mathsf{C}$-free long decomposition with $m$ non-monomials per $k$-period and with nomodata $(\alpha, \ins, \outs, \pud)$. Let $e$ be the echo of $w \star f = f^\tau$ for some $w \in \STk$.	
Suppose that we have numbers $C_i(p)$ for $i \in \ZZ / e \ZZ$ and $p$ prime satisfying
$$-(\outs(j,p) + \ins(j,p)) \leq  C_i(p) \leq 0$$
and
$$C_i(p) + \ins(j+zr,p) \geq  C_{i-1}(p) + \ins(j-1+zr,p)$$
for some $r$ and for all primes $p$, all $i$, and all $j \equiv i \pmod{e}$ and
some $z$ which may depend on $j$.
Suppose further that for primes $p$ for which the
monomials $x^p$ do not appear in $f$, we have
$C_i(p) = -\ins(j,p)$ for all $i$ and all $j \equiv i \pmod{e}$.

Then $\lambda$ defined by $\lambda(j,p) = \ins(j,p) + C_{i}(p)$ (for $j \equiv i \pmod{e}$)
satisfies the admissibility conditions for $(f,\alpha)$ and all of the following functions are $e$-periodic:

$$\alpha(j) + \sum_p \lambda(j,p),\,\,\, \ins - \lambda,\,\,\, \outs +\lambda,\,\,\, \pud(j) + \lambda(j) - \lambda(j-1) \text{ .}$$

\end{lem}

\begin{proof}
The only interaction between different primes is the conclusion that $\alpha(j) + \sum_p \lambda(j,p)$ is $e$-periodic.  We return to this conclusion after
proving the rest of the lemma separately for each prime.

For primes $p$ for which the monomials $x^p$ do not appear in $f$, we get $\lambda(j,p) = 0$ for all $j$, immediately satisfying the admissibility conditions. The last three functions in the conclusion of this lemma are just the $p$-parts of the nomodata of $f$, shown to be $e$-periodic in Lemma~\ref{lem:cheap-periodic}. The third admissibility condition only refers to these primes, so it is now verified for all primes.

Proposition~\ref{prop:khoa-mah-1} is our main tool for verifying the second admissibility condition.
Let $u \in \Affk$ be such that $w = \phi^N u$ for some integer $N$ and let $\mu$ be the trace of $u$ on $(f, \alpha)$.
Solving one of the conclusions of Proposition~\ref{prop:khoa-mah-1} for $\mu$ in terms of
the function $\ins(j,p)$ and substituting into the last conclusion of Proposition~\ref{prop:khoa-mah-1} gives
\begin{equation}
\label{eq:Z} 	 [\ins(j, p) - \ins(j+r, p)] - [\ins(j-1, p) - \ins(j-1+r, p)] = \pud(j+r, p) - \pud(j, p)  \text{ .}
\end{equation}

Iterating Equation~\ref{eq:Z} $z$-times and adding the telescoping sums, we obtain

\begin{equation}
\label{eq:Z-z} 	 [\ins(j, p) - \ins(j+zr, p)] - [\ins(j-1, p) - \ins(j-1+zr, p)] = \pud(j+zr, p) - \pud(j, p)  \text{ .}
\end{equation}

Rearranging gives
\begin{equation}
\label{eq:A}
 \pud(j, p) + \ins(j, p) - \ins(j-1, p) + (\ins(j-1+zr, p) - \ins(j+zr, p))
 = \pud(j+zr,p) \geq 0\text{ .}
\end{equation}

When $j \equiv i \pmod{e}$
\begin{equation} \label{eq:hyp-z}
C_i(p) + \ins(j+zr,p) \geq  C_{i-1}(p) + \ins(j-1+zr,p)
\end{equation}
is one of our hypotheses for some $z$.

Equivalently,
\begin{equation}
\label{eq:B}	
C_i(p) - C_{i-1}(p) \geq \ins(j-1+zr,p) - \ins(j+zr,p)
\end{equation}

Combining Equation~\ref{eq:A} which we know for all $z$
and Equation~\ref{eq:B} which we know for some $z$ gives
\begin{equation} \label{eq:C}  \pud(j, p) + \ins(j, p) - \ins(j-1, p) + (C_i(p) - C_{i-1}(p)) \geq 0\text{ .} \end{equation}

Rearranging Equation~\ref{eq:C} and recalling that $\lambda(j,p) = \ins(j,p) + C_{i}(p)$ for $j \equiv i \pmod{e}$ produces

$$ \pud(j, p) + \lambda(j,p) -\lambda(j-1,p) \geq 0\text{ ,}$$
which is precisely the second admissibility condition for $\lambda$.

The first admissibility condition $-\outs(j,p) \leq \lambda(j,p) \leq \ins(j, p)$
follows immediately from the hypothesis that $-(\outs(j,p) + \ins(j,p)) \leq  C_i(p) \leq 0$.

Since $\lambda$ satisfies the admissibility conditions for $(f, \alpha)$, by Theorem~\ref{thm:shinythm} there exists some $v \in \Affk$ with trace $\lambda$ on $(f, \alpha)$. Let $(g, \beta) := v \star (f, \alpha)$.
The four functions in the conclusion of this lemma are then a nomodata for for $(g, \beta)$.

By Remark~\ref{rk:sum-periodic}, $(\ins + \outs)$ is $e$-periodic,
so Lemma~\ref{lem:lambda-quest} applies, and condition (3) there is satisfied. From the equivalent conditions (1) and (2) in that lemma, we get that $\ins - \lambda = \ins_{g, \beta}$ and $\outs +\lambda = \outs_{g, \beta}$ are $e$-periodic in $j$. With
Equation~\ref{eq:Z}, this implies that $\pud(j) + \lambda(j) - \lambda(j-1) = \pud_{g,\beta}$ is $e$-periodic.

We know that $\pud_{g,\beta}$ is $e$-periodic for all $p$,
so their sum $\sum_{p \in \mathcal{P}} \pud_{g,\beta}$ is $e$-periodic. This sum  is exactly the total number of monomials (of all degrees) in the $j$th puddle! That is, the function $(\beta(j) - \beta(j-1))$ is $e$-periodic; and it follows that $\beta$ itself is, too.
\end{proof}

\begin{lem}
\label{lem:magical-column-adjusting}
Fix positive integers $e$ and $k$ with $k$ being a multiple of $e$.
For every function $\varpi: \ZZ/k\ZZ \rightarrow \mathbb{N}$
there is a non-positive $e$-periodic function $\zeta: \ZZ/k\ZZ \rightarrow \ZZ$
so that their sum $\lambda := \varpi + \zeta$  takes values in $\mathbb{N}$
and satisfies that for every $j \in \ZZ / k\ZZ$ there is some $z \in \ZZ$ such that
\begin{equation}
\label{eq:desired-lambda}	
\lambda(j+ze) \geq  \lambda(j+ze-1)
\end{equation}

\end{lem}

\begin{proof}

The desired Inequality~\ref{eq:desired-lambda} unwraps to
\begin{equation}
\label{eq:bob}	
\varpi(j +ze) + \zeta(j +ze) \geq \varpi(j +ze-1) + \zeta(j +ze-1).
\end{equation}

Since $\zeta$ will be $e$-periodic, Inequality~\ref{eq:bob} is equivalent to

\begin{equation} \label{eq:AB}
\varpi(j +ze) + \zeta(j) \geq \varpi(j +ze-1) + \zeta(j-1)
\end{equation}

Take $s \in \ZZ/k\ZZ$ so that $\varpi$ attains its minimal value at $s$.
We construct $\zeta(s-x)$ by recursion on $x = 0, 1, \ldots e-1$; and then extend by $e$-periodicity.  Throughout this construction we verify that Inequality~\ref{eq:AB} holds
for $j \equiv s-x \pmod{e}$ and will maintain an additional property
that
\begin{equation}
\label{eq:MIH} \tag{$\spadesuit$}	
\lambda(s  - x + ze) \geq \lambda(s) \text{ for all } z
\end{equation}

We start by setting $\zeta(s) := 0$.  Condition~\ref{eq:MIH} holds in this case by minimality of
$\varpi(s)$. 
For the recursive step, we are given $\zeta(s-x)$ and we need to choose $\zeta(s-x-1)$.
We set $\zeta(s-x -1) :=0$ if
\begin{equation}
\label{eq:0-case}
\varpi(s-x +ze) + \zeta(s-x ) \geq \varpi(s-x +ze-1)
\end{equation}
is satisfied for some $z$ for in so doing Inequality~\ref{eq:AB} holds
for this choice of $\zeta(s-x-1)$.
Condition~\ref{eq:MIH} is satisfied in this case because for $j \equiv s-x-1 \pmod{e}$, we get
$$\lambda(j) = \varpi(j) + \zeta (s-x-1) = \varpi(j) + 0  = \varpi(j) \geq \varpi(s) = \varpi(s) + 0 = \lambda(s).$$

Otherwise,
$$\varpi(s-x +ze) - \varpi(s-x +ze-1) + \zeta(s-x) <0$$
for all $z$.
In this case, we let $\zeta(s-x -1)$ be the largest (i.e. the least negative) of all these
possible values:
$$\zeta(s-x -1) := \max_z ( \varpi(s-x +ze) - \varpi(s-x +ze-1) + \zeta(s-x) ).$$
In particular, for some $z$ we have
$$\zeta(s-x -1) = \varpi(s-x +ze) - \varpi(s-x +ze-1) + \zeta(s-x)$$
so
$$\varpi(s-x +ze-1) + \zeta(s-x -1) = \varpi(s-x +ze) + \zeta(s-x).$$
In particular, we conclude
$$\varpi(s-x +ze) + \zeta(s-x) \geq \varpi(s-x +ze-1) + \zeta(s-x-1)$$
which verifies Inequality~\ref{eq:AB} for this value of $z$.

To obtain Condition~\ref{eq:MIH}, unwrap the maximum function to see that
$$\zeta(s-x -1) \geq \varpi(s-x +ze) - \varpi(s-x +ze-1) + \zeta(s-x)$$
for all $z$. Rearranging and unwrapping the definition of $\lambda$ gives
$$\lambda(s-x+ze-1) = \varpi(s-x +ze-1) + \zeta(s-x -1) \geq \varpi(s-x +ze) + \zeta(s-x) = \lambda(s-x+ze)$$
for all $z$. By the truth of Condition~\ref{eq:MIH} for $s-x$,
$$\lambda(s-x+ze) \geq \lambda(s).$$

The recursive step and, therefore, the whole recursive construction, is now complete.
To verify the last remaining instance of the desired inequality, note that
$$\lambda(s+1) \geq  \lambda(s)$$
is an instance of Condition~\ref{eq:MIH}.

Since $\lambda(s) = \varpi(s) \geq 0$, Condition~\ref{eq:MIH}
also implies that $\lambda$ is nonnegative.

\end{proof}

\begin{prop}\label{prop:cohoprop}
Let $f$ be a clean $\mathsf{C}$-free long decomposition with $m$ non-monomials per $k$-period and with nomodata $(\alpha, in, out, pud)$. Let $e$ be the echo of $w \star f = f^\tau$ for some $w \in \STk$.
 Then there exists an $m$-periodic function $\lambda$ satisfying the Admissibility Conditions for $f$ such that all of the following functions are $e$-periodic:
 $$\alpha(j) + \sum_p \lambda(j,p),\,\,\, \ins - \lambda,\,\,\, \outs + \lambda,\,\,\, \pud(j) + \lambda(j) - \lambda(j-1).$$
\end{prop}

\begin{proof}
Applying Lemma~\ref{lem:magical-column-adjusting} to the function $\varpi(j) := \ins(l,p)$
and set $C_i(p) := \zeta(i)$.
These values satisfy the hypotheses of Lemma~\ref{lem:coho-step-1}, thus completing
the proof.

\end{proof}

\begin{Rk} The $\lambda$ we obtained in~\ref{prop:cohoprop}
using Lemma~\ref{lem:magical-column-adjusting} is necessarily non-negative, but there may
be (probably are) other functions satisfying the conclusion of that proposition.
The one we built is an upper bound on solutions of $-\outs(j) \leq \lambda(j) \leq \ins(j)$ for all $j$; and for all $j$ there is some $z$ such that $\lambda(j+zr) \geq \lambda(j-1+zr)$.
The same reasoning starting with $-\outs$ and adding positive constants on cosets to obtain the desired inequality would produce a non-positive solution, which would be a lower bound on other solutions.
The trace of $u \in \Symk$ obtained in Theorem~\ref{thm:expungetk} will be one
of these other solutions.
\end{Rk}

The main result of this section is now an immediate consequence of Proposition~\ref{prop:cohoprop} and Theorem~\ref{thm:shinythm}.

\begin{thm}\label{thm:cohothm}
Let $f$ be a clean $\mathsf{C}$-free long decomposition with $m$ non-monomials per $k$-period and with nomodata $(\alpha, in, out, pud)$. Let $e$ be the echo of $w \star f = f^\tau$ for some $w \in \STk$.
Then there exists some $v \in \Affk$ for which $v \star f$ is defined and has $e$-periodic nomodata.
\end{thm}

\begin{proof}
Applying Proposition~\ref{prop:cohoprop} we obtain a function $\lambda$ satisfying
the Admissibility conditions for $(f,\alpha)$.  By Theorem~\ref{thm:shinythm} there
is some $v \in \Affk$ for which $v \star f$ is defined and has trace
$\lambda$ on $(f,\alpha)$.
The $e$-periodic functions in the conclusion of Proposition~\ref{prop:cohoprop} are exactly a nomodata for $v \star f$.
\end{proof}

\subsubsection{From $\Affk$ to $\Symk$}
\label{sec:removing-tk}

The main result of this Subsection~\ref{sec:removing-tk} is Theorem~\ref{thm:expungetk} in
which Theorem~\ref{thm:cohothm} is upgraded replacing $\Affk$ by $\Symk$.

\begin{thm} \label{thm:expungetk}
Let $N$ be an integer dividing $k$, and let $f$ be a clean $\mathsf{C}$-free long decomposition, and suppose that $v \star f$ has $N$-periodic nomodata for some $v \in \Affk$. Then $u \star f$ has $N$-periodic nomodata for some $u \in \Symk$.
\end{thm}

\begin{Rk}
Theorem~\ref{thm:expungetk} holds trivially when $N = k$.  For the remainder of
this subsection we take $N < k$.  	
\end{Rk}

We prove Theorem~\ref{thm:expungetk} through a series of reductions.

\begin{prop}
\label{prop:reduction-tk-right}
Theorem~\ref{thm:expungetk} follows from the special case where $v = v_0 t_k$ for some
$v_0 \in \Symk$.
\end{prop}
\begin{proof}
We induct on the number $M$ of instances of $t_k$ in some presentation of $v$
as a word in $\{ t_1, \ldots, t_k \}$.
When $M = 0$ and there are no instances of $t_k$ in $v$, let $u = v$.
For the induction step, take a word $v = v_0 t_k \widehat{v}$ where $\widehat{v}$ has $M$ instances of $t_k$.
Let $g := \widehat{v} \star f$. Now $v_0 t_k \star g$ has periodic nomodata, so by the
special case, there is some word $\widehat{u}$ with no $t_k$'s
for which $\widehat{u} \star g$
has periodic nomodata. But then $\widehat{u} \widehat{v} \star f = \widehat{u} \star g$ has
periodic nomodata. Now $\widehat{u} \widehat{v}$ has only $M$ instances of $t_k$ in it, so by
the induction hypothesis, there is some  $u \in \left(
t_1, \ldots, t_{k-1} \right)$
 such that $u \star f$ has periodic nomodata. \end{proof}

Let us consider how the expression $v = v_0 t_k$ in Proposition~\ref{prop:reduction-tk-right}
acts on $f$.  First, $t_k$ swaps the boundary factors and then $v_0$ mixes everything
else using Ritt swaps.  The next lemma expresses the sense in which it would have
been possible to first perform this kind of mixing before performing a swap
at the boundary.  That is, one can arrange all of the factors in the
correct order first, allow $t_k$ to swap the boundard factors, and then bring the
newly changed boundary factors into the middle of the decomposition.

\begin{lem}
\label{lem:transit-tk-swap}
For any $v_0 \in \Symk$
there is some  $\widehat{v} \in \langle t_2, \ldots, t_{k-2} \rangle$ such that
$v_0 t_k = v_r v_\ell t_k \widehat{v}$ where
$v_\ell$ is a left-to-right transit $t_a t_{a+1} \ldots t_{k-1}$
and $v_r$ is a right-to-left transit $t_b t_{b-1} \ldots t_1$.
\end{lem}

\begin{proof}
First, put $v_0$ into first canonical (See Remark~\ref{rk:canonical-form}):
 $$v_0 = v_r v'$$
where $v_r$ is as above, and $v' \in \langle t_2, \ldots, t_{k-1} \rangle$.
Then put $v'$ into reverse first canonical form:
 $$v' = v_\ell \widehat{v}$$
where $v_{\ell}$ and $\widehat{v}$ are as required.
So,
 $$v_0 t_k = v_r v_\ell \widehat{v} t_k \text{ .}$$
Since non-adjacent Ritt swaps commute,
$$v_0 t_k = v_r v_\ell t_k \widehat{v} \text{, }$$
as wanted.
\end{proof}

\begin{prop}
\label{prop:reduction-transit-tk}
Theorem~\ref{thm:expungetk} follows from the special case where $v = v_r v_{\ell} t_k$ where
$v_\ell$ is a left-to-right transit $t_a t_{a+1} \ldots t_{k-1}$ bringing the leftmost factor into the middle
 and $v_r$ is a right-to-left transit $t_b t_{b-1} \ldots t_1$ bringing the rightmost factor into the middle.
\end{prop}

\begin{proof}
From the Proposition~\ref{prop:reduction-tk-right},
we know it suffices to prove the special case where $v = v_0 t_k$ and
$v_0 \in \Symk$.
Lemma~\ref{lem:transit-tk-swap} gives us $v_r$ and $v_\ell$ as in the
statement of this proposition and $\widehat{v} \in \Symk$
such that
$v_0 t_k = v_r v_\ell t_k \widehat{v}$ in $\Affk$.
As in the proof of Proposition~\ref{prop:reduction-tk-right}, let $g := \widehat{v} \star f$.
Now $v_r v_\ell t_k \star g$ has periodic nomodata, so for some $\widehat{u} \in
\Symk$ we have that $\widehat{u} \star g$ has periodic nomodata.
Thus, $\widehat{u} \widehat{v} \star f$ has periodic nomodata.  That is,  $u := \widehat{u} \widehat{v}$ works.
\end{proof}

The next proposition says that the altered boundary factors cannot go very far into the middle: they must stay within the first/last period of the periodic nomodata of  $v_r v_\ell t_k \star f$.

\begin{prop} \label{prop:coralf1fk}
Fix a long decomposition $f$ and a natural number $N < k$.
Suppose that $v_r v_{\ell} t_k \star f$ has $N$-periodic nomodata for some
$v_\ell = t_a t_{a+1} \ldots t_{k-1}$ and $v_r = t_b t_{b-1} \ldots t_1$.
Then $b+1 \leq N$ and $a \geq k-N$.
\end{prop}

\begin{proof}
We first show that $b+1 \leq N$. Leting $h := v_{\ell} t_k \star f$, we get that
$t_b t_{b-1} \ldots t_1 \star h$ has $N$-periodic nomodata.
In evaluating $t_b t_{b-1} \ldots t_1 \star h$, the rightmost factor $h_1$ of $h$ has to swap with $h_2$ and $h_3$ and so on until $h_{b+1}$. If $h_1$ is not a monomial, this means that all of $h_2, h_3, \ldots, h_{b+1}$ are monomials. If $h_1$ is a monomial $x^p$, this means none of $h_2, h_3, \ldots, h_{b+1}$ are monomials of the same degree $p$. In particular, whatever $h_1$ becomes after all these swaps cannot be the same shape as any of the things that $h_2, h_3, \ldots, h_{b+1}$ become after these swaps (recall Definition~\ref{def:def-becomes} for
the notion of ``becoming'') . However, if $b+1 > N$, this list $h_2, h_3, \ldots, h_{b+1}$ contains a full $N$-period of $t_b t_{b-1} \ldots t_1 \star h$, so $h_1$ has to have the same shape as one of them!

So we now know that $v_r$ does not touch the factors $h_k, \ldots, h_{N+2}, h_{N+1}$.
Thus, these are also the factors of $v_r \star h$ which has $N$-periodic nomodata.
Because $N$ is a proper divisor of $k$ it is at most $k/2$.
Thus, $v_r$  leaves fixed at least half of the factors and, in particular,
leaves fixed the whole left-most period. 

Exactly the same reasoning applies to $v_{\ell}$. Let $g := t_k \star f$, so that
$v_{\ell} \star g = h$. If $a < k-N$, then $g_k$ swaps across at least $N$ factors of $g$ which then become the leftmost $N$ factors $h_k, h_{k-1}, \ldots , h_{k-N+1}$ of $h$, which are then untouched by $v_r$ and become the leftmost period of $v_r \star h$ which has $N$-periodic nomodata. Whatever $g_k$ has become after all these swaps has to have something like the shape of one of these factors $h_k, h_{k-1}, \ldots , h_{k-N+1}$ that this very $g_k$ swapped across, giving the same contradiction.
\end{proof}

With the next proposition we extend the reduction of Proposition~\ref{prop:reduction-transit-tk} arguing that we may drop $v_r$.
 The idea here is that $v_r$ acts on the rightmost $N$-period, and we can instead undo its action on all other $N$-periods. Since these periods have the same nomodata, and nomodata completely determine whether Ritt swaps are defined, this action will necessarily be defined.

\begin{prop}
\label{prop:reduction-drop-vr}
Let $f$ be a long decomposition and $N$ a proper divisor of $k$.
Suppose that $g := v_r v_{\ell} t_k \star f$ has $N$-periodic nomodata for some
$v_\ell = t_a t_{a+1} \cdots t_{k-1}$ and $v_r = t_b t_{b-1} \ldots t_1$.
Then $w_{\ell} t_k u \star f$ has $N$-periodic nomodata for some left-to-right
transit $w_\ell = t_{a'} t_{a'+1} \cdots t_{k-1}$  and some
$u \in \Symk$.
\end{prop}

\begin{proof}
From Proposition~\ref{prop:coralf1fk}, we know that $b \leq N-1$, so that $v_r$ acts on the rightmost $N$-period. For $j = 1, 2, 3, \ldots, k/N$, let $u_j$ be the word that acts like $v_r^{-1}$ on the $j^\text{th}$ $N$-period.
Formally, expressing $v_r$ as a word in $t_i$'s,
$u_j$ is obtained from $v_r^{-1}$ by adding $(j-1)N$ to the index of each Ritt swap $t_i$ in the word.
So, $u_1 = v_r^{-1}$. 
Since $v_r^{-1} \star g$ is defined, and $g$ has $N$-periodic nomodata, $u_j \star g$ is also defined for each $j$. Since these $u_j$ act on disjoint parts of $g$, the composition $u_{k/N} \ldots u_2 u_1 \star g$ is also defined. Since
$u_{k/n} \ldots u_2 u_1$ does the same thing to each $N$-period, the new decomposition
 $$u_{k/N} \ldots u_2 u_1 \star g = u_{k/N} \ldots u_2 u_1 v_r v_{\ell} t_k \star f
 = u_{k/N} \ldots u_2 v_{\ell} t_k \star f$$
also has $N$-periodic nomodata.
For $j = 2, 3, \ldots, k/N -1$, the $u_j$ act on the ``middle'' $N$-periods -- not the first nor the last -- so they commute with $v_\ell$ and with $t_k$. Thus, we see that
$$u_{k/N} \ldots u_2 v_{\ell} t_k \star f = (u_{k/N} v_{\ell}) t_k  (u_{k/N-1} \ldots u_2) \star f$$
has $N$-periodic nomodata.
Finally, putting $u_{k/N} v_{\ell}$ into first canonical form $w_\ell u_0$ for some
$u_0 \in \langle t_{k-N}, \ldots t_{k-2} \rangle$, acting on all but the leftmost of the
factors of the leftmost $N$-period, and and left-to-right transit $w_\ell =
t_{a'} t_{a'+1} \cdots t_{k-1}$. Again, $u_0$ commutes with $t_k$. So we get
$$(u_{k/N} v_{\ell}) t_k  (u_{k/N-1} \ldots u_2) \star f
 = w_\ell u_0 t_k (u_{k/N-1} \ldots u_2) \star f
 = w_\ell t_k ( u_0 u_{k/N-1} \ldots u_2) \star f$$
to have $N$-periodic nomodata.  Setting $u := u_0 u_{k/N-1} \ldots u_2$ completes the proof.
\end{proof}

Next, we drop $v_\ell$ from Proposition~\ref{prop:reduction-transit-tk} in the same way.

\begin{prop}
\label{prop:reduction-drop-vl}
Fix a long decomposition $f$ and a proper divisor $N$ of $k$.
Suppose that $g := v_{\ell} t_k \star f$ has $N$-periodic nomodata for some
$v_\ell = t_a t_{a+1} \ldots t_{k-1}$.
Then $t_k u \star f$ has $N$-periodic nomodata for some word $u \in \Symk$.  \end{prop}

\begin{proof}
In Proposition~\ref{prop:coralf1fk} we proved that $a \geq k/N+1$.
We now improve this to $a > k/N+1$. Suppose towards contradiction that $a = k/N+1$.
Then the rightmost factor $g_{k/N+1}$ of the leftmost $N$-period in $g = v_{\ell} t_k \star f$
comes from $f_1$: first, $t_k$ moves $f_1$ to the $k^\text{th}$ spot,
and then $v_\ell$ moves it right to the end of the period.
Since $g$ has $N$-periodic nomodata, this $g_{k/N+1}$ has the same combinatorial shape as the rightmost factor $g_1$, which came from $f_k$.
So, something that came from $f_1$ and something that came from $f_k$ now have the same combinatorial shape, but this contradicts the fact that $f_1$ and $f_k$ Ritt-swapped when $t_k$ acted on $f$.

As in the proof of the Proposition~\ref{prop:reduction-drop-vr}, for $j = 1, 2, \ldots, k/N$,
let $u_j$ act on the $j^\text{th}$ $N$-period the same way that
$v_\ell^{-1}$ acts on the leftmost $N$-period. So $u_j$ is obtained from
$v_\ell^{-1}$ by adding $jN$ to (equivalently, subtracting $N (k/N - j)$ from) the index of each Ritt swap $t_i$ occurring in some presentation of $v_\ell^{-1}$.
As before, $u_1 u_2 \cdots u_{k/N} \star g$ is defined and has $N$-periodic nomodata.
Thus we conclude that
$$u_1 u_2 \cdots u_{k/N} \star g =
u_1 u_2 \cdots u_{k/N} v_{\ell} t_k \star f =
u_1 u_2 \cdots u_{k/N-1} t_k \star f$$
is $N$-periodic.
As before, $u_j$ with $j = 2, 3, \ldots, k/N-1$ commute with $t_k$.
We also know that $v_\ell$ does not involve $t_{k/N+1}$, so $u_1$ does not involve $t_1$
implying that $u_1$ also commutes with $t_k$. Thus,
$$u_1 u_2 \cdots u_{k/N-1} t_k \star f =
t_k (u_1 u_2 \cdots u_{k/N-1}) \star f$$
has $N$-periodic nomodata. Setting $u := u_1 u_2 \cdots u_{k/N-1}$ completes the proof.
\end{proof}

Finally, we get drop  $t_k$ in the same way.

\begin{prop}
\label{prop:drop-tk}
Fix a long decomposition $f$ and a proper divisor $N$ of $k$.
Suppose that $g := t_k \star f$ has $N$-periodic nomodata.
Then $u \star f$ has $N$-periodic nomodata for some $u \in \Symk$.
\end{prop}

\begin{proof}
Instead of $t_k$ swapping across the boundary of the $k$-period to achieve periodicity, swap across all the boundaries of $N$-periods that are not boundaries of the $k$-periods. That is, let
$$u := t_N t_{2N} \ldots t_{(k-N)} \text{ .}$$
\end{proof}

The proof of Theorem~\ref{thm:expungetk} is now complete.

\begin{cor} \label{cor:reduce-to-periodic}
Let $f$ be a clean $\mathsf{C}$-free long decomposition.
Let $e$ be the echo of $w \star f = f^\tau$ for some $w \in \STk$.
Then there exists some $u \in \Symk$ for which $u \star f$ is defined and has $e$-periodic nomodata.
\end{cor}

\begin{proof}
Theorem~\ref{thm:cohothm} provides some $v \in \Affk$ for which $v \star f$ is defined and has $e$-periodic nomodata.
Theorem~\ref{thm:expungetk} allows us to replace $v \in \Affk$ by $u \in \Symk$.
\end{proof}

Let us record the  reduction of the proof of Theorem~\ref{thm:skew-inv}
for clean $\mathsf{C}$-free long decompositions to the special case of $e$-periodic nomodata.

\begin{prop} \label{prop:reduce-to-periodic}
Let $f$ be a clean $\mathsf{C}$-free long decomposition of a polynomial $P$, let $\mathcal{C}$ be the $(P, P^\tau)$ skew-invariant curve encoded by $w \star f = f^\tau$ for some $w \in \STk$, and let $e$ be the echo of $w \star f = f^\tau$.

There exists another clean $\mathsf{C}$-free long decomposition $g$ of the polynomial $P$ and another  $w' \in \STk$ with $w' \star g = g^\tau$ encoding the same curve $\mathcal{C}$ with the additional property that the nomodata of $g$ is $e$-periodic.
\end{prop}

\begin{proof}
Corollary~\ref{cor:reduce-to-periodic} provides the new decomposition $g := u \star f$ of the same polynomial $P$ with $e$-periodic nomodata, for some $u \in \Symk$.

Replacing $f$ by $u^{-1} \star g$ in the Mahler equation $\phi^N w \star f = f^\tau$ encoding the skew-invariant curve $\mathcal{C}$ yields
$$\phi^N w u^{-1} \star g = u^{-1} \star g^\tau$$
 and then
$$u \phi^N w u^{-1} \star g = g^\tau$$
 and
$$ \phi^N w' \star g = g^\tau$$
 where $w' = \widehat{u} w u^{-1}$ where $\widehat{u}$ is defined by $u \phi^N = \phi^N \widehat{u}$.
Since $u \in \Symk$, the curve it encodes is the diagonal. Thus, all these equations encode the same curve $\mathcal{C}$, with the same echo $e$.
\end{proof}

In the remainder of this section we complete the proof
of Theorem~\ref{thm:skew-inv} for polynomials admitting
long decompositions with sufficiently periodic nomodata.

\begin{lem}\label{lem:is-traceless}
Let $f$ be a clean $\mathsf{C}$-free long decomposition of a polynomial $P$, let $\mathcal{C}$ be the $(P, P^\tau)$ skew-invariant curve encoded by $\phi^N w \star f = f^\tau$ for some $w \in \Affk$, and let $e$ be the echo of $w \star f = f^\tau$.  Fix a nonmonomial listing function $\alpha$ for $f$. If $f$ has $e$-periodic nomodata, then the trace of $w$ on $(f,\alpha)$ is trivial.
\end{lem}

\begin{proof}
 Since $f$ and $f^\tau$ admit identical
nomodata and $w \star f$ is defined, so is $w \star f^{\tau}$.
It follows from Corollary~\ref{cor:na-nb} that there is a multiple $r$ of $e$ such that
$w \star (f,\alpha) = (f^{\tau},\beta)$ where $\beta(j) = \alpha(j+ r)$
for all $j$.  Because $\alpha$ is $e$-periodic, $\beta = \alpha$. Because the in- and out-degrees
of the corresponding factors of $f$ and $f^{\tau}$ are the same, we see that
$0 = \trace_{w,f,\alpha}$.
\end{proof}

\begin{lem}
\label{lem:clean-no-epsilons}
If $w \in \STpk$ satisfies $w \star f = f^\tau$ for some clean $\mathsf{C}$-free
long decomposition $f$ of some non-exceptional polynomial, then $w \in \STk$. 	
\end{lem}
\begin{proof}
By Remark~\ref{rk:semidirect}, we may write $w = u \varepsilon$ where $u \in \STk$
and $\varepsilon = \prod_{p \in \mathcal{P}} \epsilon_p^{m_p}$ is a product of powers of $\epsilon_p$ for various $p$.
We may choose words $\widehat{u}$ and $\widehat{\varepsilon}$
representing $u$ and $\varepsilon$, respectively, for which the such $\widehat{u}\widehat{\varepsilon}$ is reduced.
Since $w \star f$ is defined, so is $\widehat{u}\widehat{\varepsilon} \star f$.
If $m_p >0$ (respecitvely, $m_p  < 0$) for some $p$, then $\epsilon_p \star f$ (resp., $\epsilon_p^{-1} \star f$) is defined.
Thus, in particular, the monomial $x^p$ does not occur in $f$. 

Fix a nonmonomial listing function $\alpha$ for $f$.  Recall that $\alpha$ is also a nomonomial listing function for $f^\tau$
and for $\varepsilon^{-1} \star f^\tau$.
Note that because $x^p$ does not appear in $f$, the $\trace_{u,f,\alpha}(j,p) = 0$ for all $j$.
Rewrite the equation $\varepsilon u \star f = f^\tau$ as $u \star f = \varepsilon^{-1} \star f^\tau$ and compare the $p$-part of the
nomodata for the two sides of the equation. On the left, there is some $r$ for which
$$\ins_{u \star (f,\alpha)}(j,p) = \ins_{f,\alpha}(j+r, p) + \trace_{u,f,\alpha}(j,p) =
\ins_{f,\alpha}(j+r, p)$$
(see Lemma~\ref{lem:starnomodata} and Remarks~\ref{rk:nomodata-shift-unique} and~\ref{rk:phi-nomo}). On the right, $\ins_{f^\tau,\alpha} = \ins_{f,\alpha}$, while $\ins_{\varepsilon^{-1} \star f^{\tau},\alpha} (j, p) = \ins_{f^{\tau},\alpha} (j, p) + m_p$.  If $m_p \neq 0$, this gives a contradiction since $\ins_{f,\alpha}(j,p)$ is bounded.	
\end{proof}

We conclude this section by proving Theorem~\ref{thm:skew-inv} for clean $\mathsf{C}$-free
long decompositions.

\begin{prop}
\label{prop:thm1.3-in-out-case}
Theorem~\ref{thm:skew-inv} holds in the case that the polynomial $P$ admits a
clean $\mathsf{C}$-free long decomposition $f$.
\end{prop}

\begin{proof}
Let $\mathcal{C}$ be a $(P,P^\tau)$-skew invariant curve encoded by $f^\tau = w \star f$ for some long decomposition
$f$ of $P$ and $w \in \STpk$.  Fix a parameterization $(A(t),B(t))$ of $\mathcal{C}$ by polynomials
with no common nonlinear initial compositional factors.
Recall from Definition~\ref{def:magic-e} that $n_f$ (respectively, $n_A$ and $n_B$) is the number of nonmonomials
appearing in one $k$-period of $f$ (respectively, in a decomposition of $A$ or $B$) and
the echo is $e = \gcd(n_f,(n_B - n_A))$.   Fix  $\alpha$ a nonmonomial listing function for $f$.

 By Lemma~\ref{lem:clean-no-epsilons}, $w \in \STk$.  Write  $w = \phi^N w'$ for some
$N \in \ZZ$ and $w' \in \Affk$.  By Proposition~\ref{prop:reduce-to-periodic}, we may assume that
$f$ has $e$-periodic nomodata.  By Lemma~\ref{lem:is-traceless},
$\trace_{w',f,\alpha}$ is identically zero.  Then Lemma~\ref{lem:easy-step-1} shows that $w'$ may be expressed as
$w' = u_{n_f} \cdots u_2 u_1$ , a commuting product of $u_j$'s, where each $u_j$ is a product of puddleswaps in the
$j^\text{th}$ puddle.    In particular, the part of the curve encoded by $w'$ consists of monomials.  Thus, $n_A = 0$
and $n_B$ is the number of non-monomials among $f_N, \ldots f_2, f_1$.

As noted in Remark~\ref{rk:encoding}, it suffices to show that $w = u \phi^N v$ for some $u, v \in \Symk$.

If $t_k$ is not a puddle-swap, it does not occur in $w'$ so that we may take $u = \operatorname{id}$ and $v = w'$.

From now on, at the cost of replacing $\alpha$, we may take $u_1$ to act on the puddle containing $f_k$ and $f_{k+1}$.

 If $n_f$ does not divide $(n_B - n_A)$, then $\phi^N$ moves the puddle containing $f_k$ and $f_{k+1}$ to another puddle.  Thus, $\phi^N w' = \widehat{u_1} \phi^N u_{n_f} \ldots u_2$ with both $u_{n_f} \cdots u_2$ and $\widehat{u_1}$ in $\Symk$.

 Otherwise, if $n_f$ does divide $(n_B - n_A)$, then the puddle containing $f_k$ and $f_{k+1}$ in $f$ corresponds to
     the ``same'' puddle containing $f_k^\tau$ and $f_{k+1}^\tau$ in $f^\tau$ via $\phi^N w' \star f = f^\tau$.
     In particular, $u_1$ must be empty, so again we may take $u = \operatorname{id}$ and $v = w' \in \Symk$.
    \end{proof}

\subsection{Clusterings}
\label{sec:clusterings}
In this section, we adapt the technology of \emph{clusterings} from~\cite{MS} to long decompositions to identify more sophisticated combinatorial invariants, beyond wall sets and nomodata.
 This allows us to identify several ways for walls to appear in decompositions all of whose compositional factors are swappable (see Propositions~\ref{prop:c-free} and~\ref{prop:no-gate-wall} and Lemma~\ref{lem:unowned-gate-wall}).
The same technology characterizes and then handles the last remaining case of decompositions with type $\mathsf{C}$ factors but without walls (see Proposition~\ref{prop:one-way-Mahler}).
 The key difference from~\cite{MS} is that the boundary between the $k^\text{th}$ and $(k+1)^\text{st}$ entries is now treated just like all the others.
 
\subsubsection{Preclusterings}
\label{sec:preclustering}
 The following definition is closely related to  \cite[Definitions~4.2 and~6.10]{MS}.

\begin{Def}
A \emph{preclustering} $A$ of a \kld $f$ is a
$k$-periodic subset of $\mathbb{Z}$ such that
for any integers $b > a$ in $A$ with no elements of $A$ between them,
 $$f_{[b,a)} := (f_b, f_{b-1}, \ldots , f_{a+1})$$
is a cluster in the sense of \cite[Definition 4.2]{MS}.
That is, either one of the two following conditions holds.
\begin{itemize}
\item There are linear polynomials $E$ and $B$ so that
$B \circ f_b \circ \cdots \circ f_{a+1} \circ E =
C_n$ is a Chebyshev polynomial of degree $n$ not a
power of two.In this case, $f_{[b,a)}$ is called a
\emph{$\mathsf{C}$-cluster}.
\item There are linear polynomials $E$ and $B$ so that
$B \circ f_b \circ \cdots \circ f_{a+1} \circ E
= h_b \circ \cdots \circ h_{a+1}$ where each
$h_i$ is a monomial or a Ritt polynomial which
is not type $\mathsf{C}$ for all
$i \in (a,b]$.  In this case,
$f_{[b,a)}$ is called a \emph{$\mathsf{C}$-free cluster}.
\end{itemize}

Such $f_{[b,a)}$ is a \emph{cluster of $(f, A)$}.\\
The empty set is a preclustering of $f$ if $f_{[b,a)}$ is a cluster for any integers $b\geq a+k$.
\end{Def}

\begin{Rk} \label{rk:ezrmk45}
Let us translate parts (1), (2), and (3)
of~\cite[Remark 4.5]{MS} to this context.
\begin{enumerate}
\item If some factor of $f$ is unswappable, $f$ has no preclusterings.
\item If all factors of a \kld $f$ are swappable, then $\mathbb{Z}$ is a preclustering of $f$.
\item Linearly-equivalent long decompositions
have the same preclusterings.
\end{enumerate} \end{Rk}

Let us expand the less trivial
~\cite[Remark 4.5(4)]{MS}.
It concerns the requirement that the degree of a
$\mathsf{C}$-cluster must not be a power of $2$.
Because of it, we cannot drop the hypothesis that
$a, b \in A$ from the definition of ``preclustering''.
This is also the reason for $b \geq a+k$ instead of $b >
a$ in the definition of empty preclustering.

With the following remark we record how a
piece of a cluster may fail to be a cluster itself.
 
\begin{Rk} \label{rk:precluster} \label{rk:preclustermk}
A piece of a cluster is usually a cluster: if $d \geq c > b \geq a$ and $f_{[d,a)}$ is a cluster, then $f_{[c,b)}$ is a cluster unless $f_{[d,a)}$ is a $\mathsf{C}$-cluster, and all factors of $f_{[c,b)}$ are quadratic, and $f_{[c,b)}$ has at least two factors. Here are some useful ways to avoid the exception.
\begin{enumerate}
\item Any piece of a $\mathsf{C}$-free cluster is a $\mathsf{C}$-free cluster.
\item Any piece of a $\mathsf{C}$-cluster is one of the following: \begin{itemize}
    \item contains an odd-degree factor and is a $\mathsf{C}$-cluster;
    \item contains exactly one factor, which is quadratic; and is a $\mathsf{C}$-free cluster;
    \item contains multiple quadratic factors and no odd-degree factors; and is not a cluster. \end{itemize}
\item Removing a \emph{quadratic} factor from a cluster is harmless: if $f_{[b,a)}$ is a cluster and $f_{a+1}$ (respectively, $f_b$) is quadratic, then $f_{[b, a+1)}$ (respectively, $f_{[b-1, a)}$) is a cluster of the same kind, $\mathsf{C}$ or $\mathsf{C}$-free.
\item A piece of a $\mathsf{C}$-cluster that contains a big-enough cluster is itself a $\mathsf{C}$-cluster.
That is, if
 $b'' \geq b \geq b' > a' \geq a \geq a''$, and
 $f_{[b'',a'')}$ is a $\mathsf{C}$-cluster, and
  $f_{[b',a')}$ is a cluster, but is not a single quadratic,
 then $f_{[b,a)}$ is also a $\mathsf{C}$-cluster.
\item In (4), the inner cluster $f_{[b',a')}$ being big enough is equivalent to it being a $\mathsf{C}$-cluster.
\end{enumerate}

Suppose that $A$ is a preclustering of $f$ and $B \supset A$ is also $k$-periodic.
This $B$ is also a preclustering of the same $f$ if and only if
for every $\mathsf{C}$-cluster of $(A,f)$, the degrees of its pieces in $(B,f)$ are either $2$ or not powers of $2$.
\end{Rk}

With the next remark we correlate our notion
of preclustering on long decompositions with the
corresponding version in~\cite{MS} for decompositions.

\begin{Rk} \label{rk:skew-cluster-rmk}
The first two observations are related to parts of
\cite[Lemma 6.11]{MS}
and the third is related to \cite[Lemma 6.14]{MS}.
\begin{enumerate}
\item If $A$ is a preclustering of $f$, and $0 \in A$, then $A \cap \{ 0, 1, \ldots k \}$ is a preclustering of $(f_k, \ldots f_1)$ in the sense of ~\cite[Definition 4.2]{MS}.
    The assumption $0 \in A$ is not cosmetic, because of Remark~\ref{rk:precluster}.

\item More generally, for any $a \in A$, the tuple $A \cap \{ a, a+1, \ldots a+k \}$ is a preclustering of the
 decomposition $(f_{a+k}, \ldots f_{a+1})$, in the sense of~\cite[Definition 4.2]{MS}.

\item If the empty set is a preclustering of $f$, then the polynomial $f_k \circ \cdots \circ f_1$ is linearly related to a (possibly decomposable) Ritt polynomial in the sense of~\cite[Definition 2.41]{MS}. In Proposition~\ref{prop:emptyclean} below, we almost drop ``linearly related'' from that statement.
\end{enumerate} \end{Rk}

The next definition is very close
to~\cite[Definition 4.3]{MS}
and to the cleanup in~\cite[Lemma 6.11(6)]{MS}.
When at least one of the clusters is a
$\mathsf{C}$~cluster, this literally translates the
old definition into the new notation because a
linear factor which is both a translation and a
scaling is identity.  When all clusters are
$\mathsf{C}$-free, we no longer have  a canonical
place to put scalings.

\begin{Def}
Fix a preclustering $A$ of a \kld $f$.

A \emph{precleanup} of $(A, f)$ is a pair $(h, L)$,
where $h$ is a \kld, $L$ is a \kspl, and 
the following hold. \begin{enumerate}
 \item The \kld $L \circ h$ is linearly equivalent to $f$.
 \item For $i \notin A$, the linear $L_i$ is identity.
 \item If $f_{[b,a)}$ is a $\mathsf{C}$-cluster of $(A, f)$, and $b \geq i > a$, then $h_i$ is a Chebyshev polynomial;
 \item If $f_{[b,a)}$ is a $\mathsf{C}$-free cluster of $(A, f)$, and $b \geq i > a$, then $h_i$ is a monic  Ritt polynomial;
\end{enumerate}
The precleanup is a \emph{cleanup} if, furthermore,

\begin{enumerate}
 \item[(5a)] if there is a $\mathsf{C}$-cluster, whenever $f_{[b,a)}$ is a $\mathsf{C}$-free cluster, $L_a$ is a translation, and
 \item[(5b)] whenever $f_{[b,a)}$ is a $\mathsf{C}$-free cluster and $h_{a+1} \circ L_a$ is also a Ritt polynomial, $L_a$ is a scaling.
\end{enumerate}

\end{Def}

All these notions -- preclusterings, precleanups, and cleanups -- are invariant under linear-equivalence. In particular, if $(h, L)$ is a (pre)cleanup of the preclustering $A$ for \emph{some} $f$, then is is also a (pre)cleanup of the same preclustering $A$ of the long decomposition $L \circ h$.
We now work towards existence and something like uniqueness of precleanups of a fixed preclustering.

\begin{lem} Every non-empty preclustering admits a precleanup. \label{lem:getpreclean} \end{lem}
\begin{proof}
Take the ordered sequence $a_r > a_{r-1} > \ldots > a_1 > a_0 = a_r - k$ of consecutive elements of $A$ in one period, including both endponits of the period. For each $j = 0, 1, \ldots , r-1$, unwrap the definition of
 ``$f_{[a_{j+1},a_j)}$ is a cluster''
to get $h_i$ for $a_0 < i \leq a_r$, and linear $B_j$ and $E_j$.
Let $L_{a_j} := E_j \circ B_{j-1}$ for $j >0$, and $L_{a_0} := E_0 \circ \sigma^{-1}(B_k)$.
For all other $i$ with $a_0 < i < a_r$, define $L_i$ to be identity.
These $k$-tuples $\vec{h}$ and $\vec{L}$ extend uniquely to a \kld $h$ and a \kspl $L$, and the pair $(h, L)$ is a precleanup of $f$.
\end{proof}

\subsubsection{$\mathsf{C}$-free decompositions}

As long as odd-degree Chebyshev polynomials do not enter the picture, the preclusterings suffice for the Mahler problem. We now work towards Proposition~\ref{prop:c-free} which shows that in this
situation we may either reduce to an analysis of walls
(and thus Proposition~\ref{prop:warm-up})
or to an analysis of in-out-degrees (and thus
Proposition~\ref{prop:thm1.3-in-out-case}).

\begin{Def}\label{def:c-free-ksplode}
A long decomposition $f$ is \emph{$\mathsf{C}$-free} if all factors $f_i$ are swappable and $\mathsf{C}$-free.
\end{Def}

As noted Remark \ref{rk:ezrmk45}, $\mathbb{Z}$ is a preclustering of any $\mathsf{C}$-free long decomposition.
The following is an easy special case of Proposition~\ref{prop:nonemptyclean} below.

\begin{lem} \label{lem:c-free-nonemptyclean}
The preclustering $\mathbb{Z}$ of a $\mathsf{C}$-free long decomposition admits a cleanup.
\end{lem}

\begin{proof}
Start with a precleanup $(h, L)$ of $(\mathbb{Z},f)$ obtained in Lemma~\ref{lem:getpreclean}.

To satisfy requirement (5b) from the definition of cleanup, just absorb the translation part of $L_a$ into the Ritt polynomial $h_{a+1}$ whenever possible.

Because a $\mathsf{C}$-free long decomposition has no $\mathsf{C}$-clusters, the remaining requirement (5a) is vacuously satisfied.
\end{proof}

\begin{lem}\label{lem:Claim2}
Suppose that $(g, M)$ and $(h, L)$ are both cleanups of the same preclustering $\mathbb{Z}$ of the same $\mathsf{C}$-free long decomposition. Then for each $i$, $M_i$ is a scaling if and only if $L_i$ is a scaling.
\end{lem}

\begin{proof}
Here, we pay the price of lining up the indices of linear factors with the indices of Ritt swaps, but not with the indices of the factors they naturally go with.
First, replace the long decomposition $L \circ h$ (respectively, $M \circ g$) by the one whose factors are
 $h_i \circ L_{i-1}$ (respectively, $g_i \circ M_{i-1}$). Each new one is linearly equivalent to the corresponding
 old one, via $L$ and $M$, respectively. Thus, if the old ones were linearly related, then so are the new ones. Let $N$ witness the linear equivalence of the new ones.

We first show that all $N_i$ are scalings. From the definition of linearly equivalence,
$$h_i \circ L_{i-1} = N_{i+1}^{-1}  \circ g_i \circ M_{i-1} \circ N_i \text{ .}$$
Since $g_i$ and $h_i$ are $\mathsf{C}$-free Ritt polynomials,
 $N_{i+1}^{-1}$ must be a scaling by~\cite[Theorem 3.15]{MS}.

If $M_{i-1}$ is a scaling, the right-hand side of the displayed equation is a (not necessarily monic) Ritt polynomial.
Condition (5b) in the definition of cleanup then forces $L_{i-1}$ to also be a scaling.
\end{proof}

The ``clean $\mathsf{C}$-free long decompositions'' of Section~\ref{sec:inout} are precisely those $\mathsf{C}$-free long decompositions that admit a cleanup where all linear factors are scalings. We now work towards showing that every other $\mathsf{C}$-free long decomposition has a wall.

A $\mathsf{C}$-free long decomposition $f$ has a \emph{prewall at $i$} if $L_i$ is not a scaling
 in some/every cleanup $(h, L)$ of $(\mathbb{Z},f)$.

This terminology is only used in the next three lemmas, where we show that a prewall is in fact a wall.

\begin{lem}\label{lem:Claim3}
If a $\mathsf{C}$-free long decomposition $f$ has a prewall at $i$, 
then $t_i \star f$ is not defined.
\end{lem} \begin{proof}
Suppose toward contradiction that $(h,L)$ is a
cleanup of the clustering $\mathbb{Z}$ of $f$,
 $t_i \star (L \circ h) $ is defined, and
 $L_i$ is not a scaling.  Then there must be linear polynomial $A$, $B$, and $E$ for which $( A \circ h_{i+1} \circ L_i \circ B^{-1}, B \circ h_i \circ E)$ is one side of a basic Ritt identity. Since both $h_i$ and $B \circ h_i \circ E$ are $\mathsf{C}$-free Ritt polynomials, $B$ is a scaling by~\cite[Theorem 3.15]{MS}. Since both $h_{i+1}$ and $A \circ h_{i+1} \circ L_i \circ B^{-1}$ are $\mathsf{C}$-free Ritt polynomials $A$ is a scaling by~\cite[Theorem 3.15]{MS}. So now $h_{i+1} \circ L_i$ is a \texttt{non-monic Ritt polynomial}, and by definition of ``cleanup'', $L_i$ must be a scaling. \end{proof}

\begin{lem}\label{lem:Claim4}
Suppose that $t_i \star f$ is defined for a $\mathsf{C}$-free long decomposition $f$.
Then $t_i \star f$ has a prewall at $j$ if and only if $f$ does.
\end{lem}

\begin{proof}
Lemma~\ref{lem:Claim3} shows that neither $f$ nor $t_i \star f$ has a prewall at $i$.
All cleanups in this proof are for the preclustering $\mathbb{Z}$.

We first show that $f$ admits a cleanup $(h, L)$ with $L_i = \id$. By Lemma~\ref{lem:Claim3}, $L_i$ is a scaling in any cleanup. Replacing $L_i(x) =: \lambda x$ by identity, $h_{i+1}$ by $\lambda \ast h_{i+1}$, and $L_{i+1}(x)$ by $L_{i+1} (\lambda^{-\deg(h_{i+1})} x)$ produces the desired cleanup.

Since $t_i \star f$ is defined, there are linear polynomials $A$, $B$, and $E$ for which
$( A \circ h_{i+1} \circ B^{-1}, B \circ h_i \circ E)$ is one side of a basic Ritt identity.
By~\cite[Lemma 3.37]{MS}, there are Ritt polynomials $g_{i+1}$ and $g_i$ for which
$(h_{i+1} \circ h_i =  g_{i+1} \circ g_{i})$ is a basic Ritt identity.

We finally show that setting $g_{\ell} := h_{\ell}$ for all other indices $\ell$ produces a cleanup $(g, L)$ of $t_i \star f$, which obviously has prewalls in exactly the same places as $(h, L)$. The only part of the definition of ``cleanup'' that does not obviously hold is part (5b) at $(i-1)$: could it be that $L_{i-1}$ is not a scaling, but $g_i \circ L_{i-1}$ is a Ritt polynomial? We noted in~\cite[Remark 3.33]{MS} that this is not possible.
\end{proof}

\begin{lem}\label{lem:Claim5}
If a $\mathsf{C}$-free long decomposition $f$ has a prewall at $i$,
then $\epsilon_p \star f$ and $\epsilon_p^{-1} \star f$ are not defined for any $p$.
\end{lem}

\begin{proof}
The hypothesis on $\widetilde{f}$ in the
definition of the action $\epsilon_p^{\pm 1} \star f$
(Definition~\ref{def:epsilon})
produces a cleanup of $f$ with only scaling linear
factors.  Hence, for $\epsilon_p^{\pm} \star f$ to
be defined, $f$ can have no prewalls.
\end{proof}

\begin{prop}\label{prop:c-free}
Every $\mathsf{C}$-free long decomposition $f$ either has a wall, or admits a cleanup $(h,L)$ where all $L_i$ are scalings.
\end{prop}

\begin{proof}
Since $f$ is $\mathsf{C}$-free, each $f_i$ is of the
form $A_i \circ h_i \circ B_i$ where $h_i$ is a monic
$\mathsf{C}$-free Ritt polynomial and $A_i$ and $B_i$
are linear.  Let $L_i := B_{i+1} \circ A_i$.  Then
$(h,L)$ is a precleanup of the preclustering
$\mathbb{Z}$ of $f$.  To satisfy requirement (5b) of
the definition of cleanup, just absorb the translation
part of $L_i$ into the Ritt polynomial $h_{i+1}$
whenever possible, obtaining a cleanup $(g,M)$ of
$f$.

If all $M_i$ are scalings, we are done. Otherwise, some $M_i$ is not a scaling, and Lemmas~\ref{lem:Claim3}, \ref{lem:Claim4}, and~\ref{lem:Claim5} show that $f$ has a wall at $i$. \end{proof}

This completes the proof of Theorem~\ref{thm:skew-inv} for polynomials admitting $\mathsf{C}$-free long decompositions: the two cases identified in Proposition~\ref{prop:c-free} are handled by Propositions~\ref{prop:warm-up} and~\ref{prop:thm1.3-in-out-case}.

\subsubsection{Clustering with type $\mathsf{C}$}
\label{sec:cluster-C}
In this subsection, we analyze the relations
between different cleanups of the same long
decomposition with a $\mathsf{C}$~cluster
since we already have a complete solution
to the Mahler problem for $\mathsf{C}$~free
long decompositions and the assumption that
we have a $\mathsf{C}$~cluster permits us to
to match our present definition of cleanup
and what appears in~\cite{MS}.

We first work with different
cleanups of the same preclustering and then
analyze the effect of removing unnecessary
boundaries.

The operation on cleanups in the following definition is the main way to obtain one cleanup from another.

\begin{Def}
\label{def:lambdaast}
Given a \kld $f$, its \emph{degree sequence}
$\deg(f)$ is the $k$-periodic $\mathbb{Z}$-indexed
sequence $(\deg(f_i))_{i \in \mathbb{Z}}$.

By a \emph{$\ksca$} we mean a $k$-skew-periodic
$\mathbb{Z}$-indexed sequence $\lambda = (\lambda_i)_{i
\in \mathbb{Z}}$.

For a \kld $f$ and a $\ksca$ $\lambda$ we define
$\lambda \ast f$ to be the \kld whose
$i^\text{th}$ component is
$$(\lambda \ast f)_i := \lambda_i^{-\deg f_i} f_i (\lambda_i x) \text{ .}$$

For a $k$-periodic $\mathbb{Z}$-indexed sequence
$d = (d_i)_{i \in \mathbb{Z}}$ of positive integers
and a $\ksca$ $\lambda$, we define the \emph{residue
of $\lambda$ at $d$} to be the $\ksca$ $\mu$ whose
$i^\text{th}$ component is given by
$$\mu_i := \lambda_{i+1}^{-1} \lambda_i^{d_i} \text{ .}$$ 	
\end{Def}

\begin{Rk}
\label{rk:ast-for-cleanup}
An easy computation verifies that for any
\kspl $L$, $\ksca$ $\lambda$, and
\kld $f$ with degree sequence $d = \deg(f)$,
$$(\lambda \ast (h \circ L)) = (\lambda \ast h)
\circ (\lambda \ast L)$$
and
$$(\lambda \ast (L \circ h)) = (\lambda^d \ast L)
\circ (\lambda \ast h) \text{ .}$$	
\end{Rk}

\begin{Rk}
\label{rk:ast-lin-eq}
If $\mu$ is the residue of the $\ksca$ $\lambda$
at the degree sequence $d = \deg(f)$ of the
\kld $f$, then $\mu \cdot (\lambda \ast f)$ is
linearly equivalent to $f$ via the \kspl $L$ of
scalings by $\lambda$ (\emph{i.e.}, $L_i(x) =
\lambda_i x$).
\end{Rk}

\begin{Rk}
We may use these operations to get new (pre)cleanups of the same $f$. Indeed, \begin{itemize}
\item for a translation $L_i$, the resulting $\lambda_i \ast L_i$ is also a translation;
\item for a monic Ritt polynomial $h_i$, the resulting $\lambda_i \ast h_i$ is also a monic Ritt polynomial; and
\item for a Chebyshev polynomial $h_i$, the resulting $\lambda_i \ast h_i$ is the same Chebyshev polynomial if and only if $\lambda_i = \pm 1$.
\end{itemize}

\end{Rk}

Our next Lemma~\ref{lem:cleanupRmk} on the relation
between $(h,L)$ and $\lambda \ast (h,L)$ produces a
cleanup from a precleanup.
Corollary~\ref{cor:scaling-non-un} then shows how
this operation could connect two cleanups.
Proposition~\ref{prop:cleanunique}
proves that any
cleanup can be obtained (almost) from any
other by this operation.

\begin{lem} \label{lem:cleanupRmk}
Suppose that $(h, L)$ is a precleanup of a preclustering $A$ of $f := L \circ h$.
We assume that $(f,A)$ has a $\mathsf{C}$~cluster.

Fix a $\ksca$ $\lambda$ and let $\mu$ be the residue of $\lambda$ at the degree sequence $d$ of $h$.\\
Let $\lambda^d$ be the $\ksca$ defined by $(\lambda^d)_i := \lambda_i^{d_i}$; let $g := \lambda \ast h$; and let $M := \mu \cdot (\lambda^d \ast L)$.

Then $(g, M)$ is also a precleanup of $f$ if and only if $\mu_i = 1$ for $i \not\in A$ and
$\lambda_i = \pm 1$ for $i$ inside $\mathsf{C}$~clusters (that is, when $f_{[b,a)}$ is a $\mathsf{C}$~cluster of $(f, A)$, and $b \geq i > a$).

\vspace{.1in}
When $(A,f)$ has a $\mathsf{C}$~cluster, then
condition (5a) in the definition of cleanup is satisfied by $(g, M)$ if and only if $L_a(x) = \mu_a^{-1}x + C$ whenever $f_{[b,a)}$ is a $\mathsf{C}$-free cluster of $(f, A)$.\\
In particular, $\mu_a = 1$ for all such $a$ if both $(h, L)$ and $(g, M)$ satisfy (5a).
\vspace{.1in}

Condition (5b) is satisfied by $(h,L)$
if and only if that condition is satisfied by $(g,M)$.

\end{lem}

\begin{proof}
Let us work through the definition of a precleanup of $(L \circ h)$.
\begin{enumerate}
\item  Unwrapping definitions, we get
 $$M \circ g = \mu \cdot (\lambda^d \ast L) \circ (\lambda \ast h) = \mu \cdot \lambda \ast (L \circ h),$$
which is linearly equivalent to $(L \circ h)$ by
Remark~\ref{rk:ast-lin-eq}.
\item For $i \notin A$, we know that $L_i$ is identity, and we need $M_i = \mu_i \cdot (\lambda_i^{d_i}) \ast L_i$ to be identity. Since $\mu_i \cdot (\alpha \ast \id) = \mu_i \cdot \id$ for any $\alpha$, we must have $\mu_i = 1$.
\item For $i$ inside $\mathsf{C}$-clusters, we know that $h_i$ is a Chebyshev polynomial, and we need $g_i := \lambda_i \ast h_i$ to be a Chebyshev polynomial. This forces $\lambda_i = \pm 1$.
\item For $i$ inside $\mathsf{C}$-free clusters, we get no extra requirements since $\alpha \ast P$ is a monic Ritt polynomial whenever $P$ is.
\end{enumerate}

We up the stakes asking for $(g, M)$ to be a cleanup, without requiring $(h, L)$ to be one.
\begin{enumerate}
\item[(5a)]  When $a$ is the right boundary of a $\mathsf{C}$-free cluster, we need $M_a = \mu_a \cdot (\lambda_a \ast L_a)$ to be a translation. If $L_a(x) = Bx + C$, then $(\lambda_a \ast L_a)(x) = Bx +C'$, so we need $\mu_a = B^{-1}$.
If $(h, L)$ also satisfies (5a) so that $L_a$ is a translation, this forces $\mu_a = 1$.
\item[(5b)] Still at the right boundary $a$ of a
$\mathsf{C}$-free cluster,
 suppose that $g_{a+1} \circ M_a$ is a Ritt polynomial. Unwrapping the defintions of $g$ and $M$ gives
$$g_{a+1} \circ M_a = ( \lambda_{a+1} \ast h_{a+1} ) \circ \mu_a \cdot (\lambda_a^{d_a} \ast L_a)$$
$$g_{a+1} \circ M_a =
( \cdot \lambda_{a+1}^{- d_{a+1}} ) \circ
h_{a+1} \circ
( \cdot \lambda_{a+1} ) \circ
( \cdot \mu_a ) \circ
( \cdot \lambda_a^{- d_a} ) \circ
L_a \circ
(\cdot \lambda_a^{d_a})$$
By definition of residue, the whole middle scaling by $(\lambda_{a+1} \mu_a \lambda_a^{- d_a})$ is identity, so
$$( g_{a+1} \circ M_a) (x) = \lambda_{a+1}^{- d_{a+1}}  \cdot (h_{a+1} \circ L_a) (\lambda_a^{d_a} x).$$
Now $h_{a+1} \circ L_a$  is linearly-related via scalings to the Ritt polynomial $g_{a+1} \circ M_a$, so $h_{a+1} \circ L_a$ is itself a Ritt polynomial.

On the other hand, in order to have $M_a = \mu_a \cdot (\lambda_a^{d_a} \ast L_a)$ be a scaling, we need $L_a$ to be a scaling. So, if $(h, L)$ is a cleanup, and $(g, M)$ satisfies (5a), then $(g, M)$ automatically satisfies (5b).
And if $(h, L)$ satisfies (5a) but not (5b), and $(g, M)$ satisfies (5a), then $(g, M)$ cannot satisfy (5b).
\end{enumerate}
\end{proof}

The following corollary is a precise version
of~\cite[Remark 4.9]{MS}.
When all clusters are $\mathsf{C}$-free,
Corollary~\ref{cor:scaling-non-un} does not apply.
This non-uniqueness is the price of working with
skew-linear-equivalence rather than linear equivalence
(in the sense of~\cite[Definitions 2.39 and 2.58]{MS}).
This issue comes up informally in the proof
of~\cite[Lemma 6.11(6)]{MS}.

\begin{cor} \label{cor:scaling-non-un}
With the setup in Lemma \ref{lem:cleanupRmk},
suppose further that both $(h, L)$ and $(g, M)$ are cleanups of $(A, f)$.
 Then, $\lambda_i = \pm 1$ for all $i$, and $\mu_i = 1$ unless $f_{[j,i)}$ is a
$\mathsf{C}$~cluster of $(A,f)$ for some
$j \in A$.
\end{cor}

\begin{proof}
 
Suppose toward contradiction that $\lambda_{j+1} \neq \pm 1$.
From condition (3) of the definition of
preclustering, it follows that $f_{j+1}$ is in a
$\mathsf{C}$-free cluster of $(A,f)$.
Take the largest $b \in A$ with $b < j$ such that $f_{[b,a)}$ is a $\mathsf{C}$-cluster of $(A,f)$
for some $a \in A$.
Now, for every $i = b, b+1, \ldots j$, either $i \notin A$ or $i \in A$ is the right boundary of a $\mathsf{C}$-free cluster; either way, $\mu_i = 1$ for all these $i$.
To show that $\lambda_i = \pm 1$ for $i = b, b+1, \ldots j+1$, induct on $i-b$.
The base case $i = b$ follows from condition
(3) because $f_b$ is inside a $\mathsf{C}$-cluster.
For the induction step, assume that $\lambda_i = \pm 1$; recall that
$\lambda_{i+1} = \lambda_i^{d_i} \mu_i^{-1}$ by definition of ``residue'';
and that we have already showen that $\mu_i = 1$ for $i = b, b+1, \ldots j$.
So, $\lambda_{i+1} = (\pm 1)^{d_i} (1)^{-1} = \pm 1$.
The last induction step we can carry out is from $j$ to $j+1$, obtaining the desired contradiction.
\end{proof}

\begin{prop} Every non-empty preclustering admits a cleanup. \label{prop:nonemptyclean}
\end{prop}

The following proof translates the
explanation from~\cite[Remark 4.4]{MS} and
the details from the proofs of~\cite[Lemmas 4.6, 4.8, and 6.11(6)]{MS}  to our new setting with
long decompositions.

\begin{proof}

Fix a non-empty preclustering $A$ of a \kld $f$.
Start with a precleanup $(h, L)$ of $(A,f)$ obtained in Lemma~\ref{lem:getpreclean}.
To satisfy requirement (5b) from the definition of
cleanup, just absorb the translation part of
$L_a$ into the Ritt polynomial $h_{a+1}$ whenever
possible.

Because the remaining requirement (5a) is relevant only
when $(A,f)$ has a $\mathsf{C}$-cluster,
for the rest of this proof, we now assume that it
does. We will show how to
use Lemma~\ref{lem:cleanupRmk} to satisfy this
requirement.

Here are all the conditions on $\mu$ and $\lambda$ that we need.
\begin{itemize}
\item For $i \not\in A$, we need $\mu_i = 1$.
\item For $i$ inside $\mathsf{C}$-clusters, we also need $\lambda_i = \pm 1$.
\item When $a$ is the right boundary of a $\mathsf{C}$-free cluster and $L_a(x) = Bx + C$, we need $\mu_a = B^{-1}$.
\end{itemize}

For each $\mathsf{C}$-cluster $f_{[b,a)}$, we have very tight requirements on $\lambda_i$ for $a <i \leq b$, but no requirements on $\mu_a$. So we can let $\mu_a$ absorb the unwanted scalings coming from $\lambda_a$, set all those tightly controlled $\lambda_i := 1$, and restart the process at $b$. This is where the ``non-empty'' assumption is used.

\end{proof}

Empty preclusterings are important as clusterings of
long decompositions of certain polynomials.
However, these polynomials do not require the
clustering analysis for the solution of our
Mahler problem.  The following proposition shows
such polynomials are either subject
to Proposition~\ref{prop:c-free} or essentially
Chebyshev
polynomials and thus exceptional.

\begin{prop} \label{prop:emptyclean}
If the empty set is a preclustering of $f$, then
one of the following alternatives holds.
\begin{itemize}
\item $f$ is linearly equivalent to a \kld all of
whose factors are Ritt polynomials or
\item  $f$ is linearly equivalent to a \kld all of whose factors are plus or minus Chebyshev polynomials.
\end{itemize}
\end{prop}

\begin{proof}
If the empty set is a preclustering of $f$, then so is
$A := k \mathbb{Z}$. Applying Proposition~\ref{prop:nonemptyclean} to it 
produces $h$ and $L$ with no non-trivial linear factors except for $L_{kn}$. Since the empty set is a preclustering of $f$, the concatenation of two periods is also a cluster. For a $\mathsf{C}$-free cluster, \cite[Lemma 4.13]{MS}
shows that $L_0$ is a scaling, and we are done.
For a $\mathsf{C}$-cluster, \cite[Lemma 4.14]{MS}
 shows that $L_0$ is identity or the scaling by $-1$.
\end{proof}

\begin{Rk}
\label{rk:empty-cheb}
According to Proposition~\ref{prop:emptyclean}
if $(\varnothing,f)$ has a $\mathsf{C}$~cluster,
then $f$ is a long decomposition of an
exceptional polynomial.	
\end{Rk}

\begin{prop} \label{prop:cleanunique}
Suppose that two pairs of sequences, $(h, L)$ and
$(g, M)$, are both cleanups of the same non-empty
preclustering $A$ of the same \kld $f$.
As is standard in this subsection, we
assume that $(A,f)$ has a $\mathsf{C}$~cluster.
Then $(h,L)$ may be obtained from $(g,M)$
by first forming $\lambda \ast (g,M)$ for
some $\ksca$ $\lambda$ and then
inserting translations to the
right of type~$\mathsf{A}$ Ritt polynomials in
$\lambda \ast (g,M)$.

More precisely, there is a $\ksca$ $\lambda$ such that
$(h,L)$ comes from $(\widetilde{g}, \widetilde{M}) := \lambda \ast (g,M)$ in
the following way. For each $a \in A$ where $\widetilde{g}_{a+1}$ is a type A Ritt polynomial (see Definition 3.8 in \cite{MS}) and $\widetilde{M}_a$ is not a scaling, there is a translation $T_a$ such that $h_{a+1} = \widetilde{g}_{a+1} \circ T_a$ and $L_a = T_a^{-1} \widetilde{M}_a$. Since both $\widetilde{g}_{a+1}$ and $h_{a+1}$ are Ritt polynomials, there are only finitely many such translations for each fixed $a$.

 Furthermore, $\lambda$ satisfies the conclusion of Corollary~\ref{cor:scaling-non-un}: $\lambda_i = \pm 1$ for all $i$ and
$\mu_i = 1$ inside clusters of $A$
 and to the right of $\mathsf{C}$-free clusters.
 Here, $\mu$ is the residue of $\lambda$ at the degree sequence $h$.
\end{prop}

\begin{proof}

Since $f$ only appears in these definitions up to linear-equivalence, we may assume that $f= L \circ h$.
Fix $b, a \in A$ such that $f_{[b,a)}$ is a $\mathsf{C}$~cluster of $(A, f)$.

Since $(g, M)$ is also a cleanup of $L \circ h$, some \kspl $N$ witnesses that $M \circ g$ is linearly equivalent to $L \circ h$, and then
$$(L_{a+k} \circ h_{a+k},  L_{a+k-1} \circ h_{a+k-1}, \ldots , L_{a+2} \circ h_{a+2}, L_{a+1} \circ h_{a+1})$$
and
$$(N_{a+k+1}^{-1} \circ M_{a+k} \circ g_{a+k}, M_{a+k-1} \circ g_{a+k-1}, \ldots , M_{a+2} \circ g_{a+2}, M_{a+1} \circ g_{a+1} \circ N_{a+1})$$
are linearly equivalent in the sense of~\cite[Definition 2.49]{MS}, and we have two almost-cleanups of them, again in the sense of~\cite[Definition 4.3]{MS}:
$$(h_{a+k}, \ldots h_{a+1};  L_{a+k}, \ldots, L_{a+1}, \id))$$
and
$$(g_{a+k}, \ldots g_{a+1}; (N_{a+k+1}^{-1} \circ M_{a+k}), M_{a+k-1}, \ldots, M_{a+2}, M_{a+1}, N_{a+1}).$$

Since $f_{[b,a)}$ is a $\mathsf{C}$~cluster, there are no further requirements on the linear factors at $k$ and at $0$ in \cite[Definition 4.3]{MS}, so these are actual cleanups in the sense of~\cite[Definition 4.3]{MS}. By~\cite[Lemma 4.10]{MS}, these can
 only differ by scalings of $\pm 1$ and by translations to the right of $\mathsf{C}$-free clusters. In particular, since $N_{a+1}$ is not to the right of a $\mathsf{C}$-free cluster, it must be a scaling by $\pm 1$; and so $N_{a+k+1} = \sigma(N_{a+1})$ is also a scaling by $\pm 1$. So, the original long cleanups $(h, L)$ and $(g, M)$ can only differ by scalings of $\pm 1$ and by translations to the right of $\mathsf{C}$-free clusters, as required.
 \end{proof}

We now move on to comparing cleanups of different preclusterings.
Lemma~\ref{lem:refinery} builds on Remark ~\ref{rk:precluster} tracking what happens to cleanups when a preclustering is refined to have more, smaller clusters per period. The following two lemmas are partial converses of Lemma~\ref{lem:refinery}, used to recognize when a given preclustering is a refinement of another.
Lemma~\ref{lem:clusterfueCfree} is essentially
\cite[Lemma 4.13]{MS} about fusing two adjacent $\mathsf{C}$-free clusters.
Lemma \ref{lem:clusterfueC} is a better version of \cite[Lemma 4.14]{MS} made possible by
our Corollary~\ref{cor:scaling-non-un} which is a
better version of \cite[Remark 4.9]{MS}.

\begin{lem} \label{lem:refinery}
Suppose that $(h, L)$ is a precleanup of a preclustering $A$ of a \kld $f$, and $B \supset A$ is also a preclustering of the same $f$.
Then $(h, L)$ is \emph{not}
a precleanup of $(B,f)$ if and only if
\begin{itemize}
\item[$(\dag)$] there are some $b \geq i > a$ for which
 $f_{[b,a)}$ is a $\mathsf{C}$~cluster of $(A, f)$,
  and $f_i$ is quadratic, and both $i$ and  $(i-1)$ are in $B$.
\end{itemize}

If $(h, L)$ is a cleanup of $(A,f)$ and $(\dag)$ does not hold, then $(h, L)$ is also a cleanup of $(B,f)$.
\end{lem}

\begin{proof}
Part (1) of the definition cleanup only depends on $h$, $L$, and $f$, not on $B$.
Part (2) of that definition is strictly easier to satisfy for a bigger $B$.
For parts (3) and (4), it is necessary and sufficient to have each factor $f_i$ in the same type of cluster in $(A,f)$ and $(B,f)$. By part (2) of Remark \ref{rk:preclustermk}, this fails if and only if $B$ rips a single quadratic factor off a $\mathsf{C}$~cluster of $A$, which is exactly the situation that $(\dag)$ describes more precisely.
Parts (5a) and (5b) of the definition of cleanup are again strictly easier: without $(\dag)$, every ``cluster boundary of $(B,f)$ to the right of a $\mathsf{C}$-free cluster'' is also a ``cluster boundary of $(A,f)$ to the right of a $\mathsf{C}$-free cluster.''
\end{proof}

\begin{lem} \label{lem:clusterfueCfree}
Suppose that $(h, L)$ is a cleanup of a non-empty preclustering $A$ of $f := L \circ h$
with a $\mathsf{C}$~cluster.
Suppose that $c > b > a$ are consecutive elements of $A$; that $f_{[c,b)}$ and $f_{[b,a)}$ are $\mathsf{C}$-free clusters; and that their concatenation $f_{[c,a)}$ is also a cluster. Then $L_b$ is identity, and $B := A \smallsetminus \{b+kn : n \in \mathbb{Z} \}$ is also a preclustering of $L \circ h$, and $(h, L)$ is also a cleanup of $(B, L \circ h)$.
\end{lem}

\begin{proof}
To show that $L_b$ is identity, we
apply~\cite[Lemma 4.13]{MS}
 to a suitable finite chunk of these infinite objects. One period is not a long enough chunk, as it might not contain all three of $a$, $b$, and $c$. Since $A$ is $k$-periodic, $b \leq a+k$ and $c \leq b+k$, so $c \leq a+2k$, so two periods is enough.

Let $A' := \{ x-a : x \in A \mbox{ and } a \leq x \leq a+2k \}$.
Now $( (h_{a+2k}, \ldots , h_{a+1}), (L_{a+2k}, \ldots , L_{a}) )$ is a cleanup in the sense of \cite[Definition 4.3]{MS}
of the preclustering, again in the sense of \cite{MS}, $A'$ of the decomposition $(f_{a+2k}, \ldots , f_{a+1})$.

Applying \cite[Lemma 4.13]{MS} from  to these shows that $L_b$ is identity. Since $A$ is $k$-periodic, it follows that $L_{b+kn}$ are all identity. The rest of the lemma now follows easily.
Note that $B$ is non-empty because $(A,f)$ has
both a $\mathsf{C}$~cluster and a $\mathsf{C}$-free
cluster.   Writing out the definition verifies that it is a preclustering.
\end{proof}

\begin{lem} \label{lem:clusterfueC}
Suppose that $(h, L)$ is a cleanup of a non-empty preclustering $A$ of $f := L \circ h$.
Suppose that $c > b > a$ are consecutive elements of $A$; that $f_{[c,b)}$ and $f_{[b,a)}$ are $\mathsf{C}$~clusters; and that their concatenation $f_{[c,a)}$ is also a cluster.

Then $L_b(x) = \pm x$ and $B := A \smallsetminus \{b+kn : n \in \mathbb{Z} \}$ is also a preclustering of $L \circ h$.
Now $(h,L)$ is also a cleanup of $(B,f)$ if and only if $L_b$ is identity.

\end{lem}

In any case, any cleanup of $(B,f)$ is also a cleanup of $(A,f)$; and one can be obtained from $(h,L)$ via Lemma~\ref{lem:cleanupRmk} as long as $(B,f)$ admits any cleanups.

\begin{proof}
Exactly as in the proof of the last lemma, we apply \cite[Lemma 4.14]{MS}  to two periods of $(A,f)$ to see that $L_b(x) = \pm x$.

\textit{Case 1: Suppose that $B$ is empty.}

If $L_b$ is identity, we have that all $h_i$ are Chebyshevs with at least one of odd degree, and all $L_i$ identity, so for any $e \geq d+k$, we have $f_{[e,d)} = h_{[e,d)}$ is a Chebyshev polynomial whose degree  is not a power of two. So $B$ is a preclustering and $(h,L)$ is a cleanup of it.

If $L_b$ is not identity and some $f_i$ has degree 2, let $d > b$ be the least such $i$.
Let $\lambda_i := -1$ for $i \equiv b, b+1, \ldots d \pmod{k}$ and $\lambda_i = 1$ for other $i$. Use this $\lambda$ to get a cleanup $(g, M)$ of $(A,f)$ via Lemma~\ref{lem:cleanupRmk}, with $M_b$ identity, so all $M_i$ identity. Again, $B$ is a preclustering and $(g, M)$ is a cleanup of it.

If $L_b$ is not identity and all $f_i$ have odd degree, it is easy to see that $f_{[e,d)}$ is linearly related to an odd-degree Chebyshev polynomial for any $e > d$, so $B$ is a preclustering, but it is the exception in Proposition~\ref{prop:emptyclean} admitting no cleanup.

\textit{Case 2: Suppose that $B$ is not empty.}

If $L_b$ is identity, writing out definitions verifies that $B$ is a preclustering and $(h,L)$ is a cleanup of it.

Otherwise, Lemma~\ref{lem:cleanupRmk} with an appropriate $\ksca$ $\lambda$ will produce another cleanup $(g,M)$ of $(A,f)$ with $M_b$ identity, reducing to the previous sentence.

More precisely, let $d$ be the degree sequence of $h$ and define $\lambda$ by starting with $\lambda_b = -1$, proceeding inductively with $\lambda_{i+1} = \lambda_i^{d_i}$ until $\lambda_{a+k}$ is defined, and filling the rest of the period with $\lambda_i = 1$. This makes sense because $b < a+k$ because $B$ is non-empty. The residue of this $\lambda$ at the degree sequence of $h$ is $-1$ at $b+nk$ and $\pm 1$ at $a+k+nk$ and $1$ everywhere else. The first assures that $M_b$ is identity, and the second is ok because $a+nk$ are right boundaries of $\mathsf{C}$~clusters.
\end{proof}

\subsubsection{Gates and clusterings}

Cluster boundaries give rise to the combinatorial invariants we use to solve the Mahler problem for polynomials all of whose factors are swappable but not all of whose factors are $\mathsf{C}$-free.
However, cluster boundaries of a preclustering are very far from an invariant, as we can introduce extra boundaries almost at will. In this subsection, we show that a preclustering with a minimal number of clusters, called a \emph{clustering}, is almost an invariant of the long decompositions. 

In Lemma~\ref{lem:uniqclusttech} we show that two such clusterings only differ by assigning
what we call ``wandering quadratic'' factors (see Definition~\ref{def:loiter-def}) to different clusters. With that result, the boundary set (Definition~\ref{def:def-bdry}) becomes the desired combinatorial invariant (Theorem~\ref{thm:unique-brdyset}).

 The intuitive definition of clustering in terms of the minimal number of clusters turns out to be equivalent to two more useful but less intuitive definitions, one in terms of the linear factors in a cleanup and the other in terms of two particular ways to reduce the number of clusters. We take the first one as our definition (Definition~\ref{def:fakedef}), and obtain the second (Lemma~\ref{lem:clusterfue}) as a consequence.
 In Lemmas~\ref{lem:lem421} and~\ref{lem:lem432}, we describe the only way to obtain one clustering from another by moving a wandering quadratic; we then define the combinatorial invariant (Definition~\ref{def:def-bdry}) that we  use in our proofs. The rest of the section consists of the tedious verification that it is indeed invariant: that any two clusterings come from each other via the operations described in Lemmas~\ref{lem:lem421} and~\ref{lem:lem432}.

 Much of the technical work of this section is closely related to~\cite[Section 4]{MS}, but the overall approach is substantially improved. With long decompositions, we avoid the artificial cluster boundary at $k$ that unnecessarily complicated \cite{MS}. With Definition~\ref{def:def-bdry}, we explicitly identify the combinatorial invariant that remained implicit in \cite{MS}.

\begin{Def} \label{def:gatesdef}
Fix a cleanup $(h, L)$ of a non-empty preclustering $A$ of $f$, and fix consecutive elements $c > b > a$ of $A$. For the duration of this definition,
\begin{itemize}
\item[-] the \emph{left cluster} is $f_{[c,b)}$; and
\item[-] the \emph{right cluster} is $f_{[b,a)}$; and
\item[-] the \emph{linear factor $L$} is $L_b$.\end{itemize}
We define ``$(h, L, A)$ has a \underline{\hspace{.5in}} gate at $b$'' exactly as in~\cite[Definition 4.15]{MS}, where the blank \underline{\hspace{.5in}} is one of ``left-to-right'', ``right-to-left'', ``one-way'', or  ``two-way''.
\vspace{.2cm}

More formally, $( (h_c, \ldots,  h_{a+1}), (\id, L_{c-1}, \ldots,  L_a) )$ is a cleanup in the sense of~\cite[Definition 4.2]{MS} of the preclustering $(c-a, b-a, a-a)$ in the sense of~\cite{MS} of the decomposition $(h_c \circ L_{c-1}, \ldots,  h_{a+1} \circ L_a)$. Whenever this has a left-to-right (respectively, right-to-left, one-way, two-way) gate at $1$ in the sense of \cite[Definition 4.15]{MS}, we now say that $(h, L, A)$ has a
left-to-right (respectively, right-to-left, one-way, two-way) gate at $b$.

\end{Def}

The next lemma justifies the terminology ``$(A,f)$ has a gate at $b$'', explains the purpose of gates, and notes that two-way gates are superfluous.

\begin{lem} \label{lem:purposeofgates}
 Two cleanups of the same preclustering of the same decomposition have the same kinds of gates in the same places. (Compare with~\cite[Lemma 4.17(1)]{MS}.)

 The concatenation of a cluster $f_{[b,a)}$ of $(A, f)$ and an adjacent quadratic factor $f_a$ (respectively, $f_{b+1}$) is a cluster if and only if $(A,f)$ has a right-to-left gate at $a$ (respectively, left-to-right gate at $b$). The clusters $f_{[b,a)}$ and $f_{[b,a-1)}$ (respectively, $f_{[b,a)}$ and $f_{[b+1,a)}$) are of the same kind, $\mathsf{C}$ or $\mathsf{C}$-free.
 (Compare with \cite[Remark 4.20]{MS}.)

 If $c > b > a$ are consecutive elements of $A$ and $(A,f)$ has a two-way gate at $b$,
then the two clusters $f_{[c,b)}$ and $f_{[b,a)}$ are of the same kind (both $\mathsf{C}$ or both $\mathsf{C}$-free), and their concatenation $f_{[c,a)}$ is also a cluster. (Compare with \cite[Lemma 4.16]{MS}.)
\end{lem}

\begin{proof} The first statement follows immediately from comparing~\cite[Remark 3.14]{MS} with Proposition~\ref{prop:cleanunique} here. The second and third statements are immediate consequences of~\cite[Remark 3.14]{MS}. \end{proof}

Beware that when the preclustering has only one cluster per period,
the concatenations in the last lemma are not contained in one period,
and so are not clusters of any nonempty preclustering of $f$.

With the following definition we introduce the
notion a fake gate.  Morally, these identify
cluster boundaries which can be removed.
Parts (2) and (3) below are closely related to the
fake wandering quadratics from~\cite[Definition 4.19]{MS}.

\begin{Def} \label{def:fakedef}
Fix a non-empty preclustering $A$ of $f$, and fix consecutive elements $c > b > a$ of $A$.
The gate of $(A, f)$ at $b$ is \emph{fake} if one of the following holds:\begin{enumerate}
\item The gate is two-way.
\item The gate is left-to-right, $c = b+1$, and $f_{b+1}$ is quadratic.
\item The gate is right-to-left, $b = a+1$, and $f_{a+1}$ is quadratic.
\item The gate is left-to-right, $c = b+2$, both $f_{b+2}$ and $f_{b+1}$ are quadratic,
 and $(A, f)$ has a right-to-left gate at $c$.
\item The gate is right-to-left, $b = a+2$, both $f_{a+2}$ and $f_{a+1}$ are quadratic
 and $(A, f)$ has a left-to-right gate at $a$.
\end{enumerate}

A \emph{clustering} is a preclustering with no fake gates.
\end{Def}

Fakeness of gates is defined entirely in terms of the degrees of the factors $f_i$ and gates.
The degrees of factors are a property of the \kld alone, independent of the preclustering $A$ and cleanup.
The gates are a property of the \kld and the preclustering, still independent of cleanup, by the first part of Lemma~\ref{lem:purposeofgates}. So the fakeness of gates again only depends on $f$ and $A$ and does not depend on the choice of cleanup; and so being a clustering is a property of $(A,f)$, independent of cleanup.

\begin{lem} \label{lem:removefake}
 If a preclustering $A$ of $f$ has a fake gate, then $f$ admits another preclustering $B$ with fewer cluster boundaries per period, that is,\\ with $ | B \cap \{1, 2, \ldots, k \} | < | A \cap \{1, 2, \ldots, k \} | $. \end{lem}

\begin{proof}
Suppose $c > b > a$ are consecutive elements of $A$ and $(A,f)$ has a fake gate at $b$.

For (1), by the last part of Lemma~\ref{lem:purposeofgates} and one of
Lemmas~\ref{lem:clusterfueCfree} or~\ref{lem:clusterfueC},
 $B := A \smallsetminus \{ b+nk : n \in \mathbb{Z} \}$ is another preclustering of $f$.

For (2) and (3), by the second part of Lemma~\ref{lem:purposeofgates}, the same $B$ obviously works as long as it is not empty. If $B$ is empty, then $k=1$ and $A = \mathbb{Z}$ and all factors of $f$ are quadratic. Then all clusters of $(A,f)$ are $\mathsf{C}$-free, so a gate between two of them must be a two-way gate, so $f$ is linearly equivalent to $h$ where every $h_i(x) = u_i x^2$ for some scalar $u_i$. The empty set is obviously a preclustering of $h$, and so also of $f$.

For (4), consider $$B := (A
\smallsetminus [\{ b+nk : n \in \mathbb{Z} \} \cup \{ c+nk : n \in \mathbb{Z} \}]) \cup \{ b+1+nk : n \in \mathbb{Z} \}.$$
 By the second part of Lemma~\ref{lem:purposeofgates}, this $B$ is another preclustering of $f$. As long as $b$ and $c = b+2$ are not the same modulo $k$, this $B$ has fewer cluster boundaries per period than $A$: two were removed and one added back. Otherwise, $k=2$ and each cluster of $(A,f)$ consists of two quadriatic factors, so all clusters of $(A,f)$ are $\mathsf{C}$-free clusters, so the gate at $b$ must be a two-way gate, and just as for (2) and (3) above, the  empty set is a preclustering of $f$.

The proof for (5) is identical to the one for (4).
\end{proof}

The next Corollary is the analogue of~\cite[Lemma 6.11(4)]{MS}.

\begin{cor}
\label{cor:clustering-exists}
 If all factors of a \kld $f$ are swappable, then $f$ admits a clustering. \end{cor}

\begin{proof} Start with the preclustering $\mathbb{Z}$ of $f$ from Remark~\ref{rk:ezrmk45}, and induct on the number of cluster boundaries per period, using Lemma~\ref{lem:removefake} to decrease it as long as there are fake gates. \end{proof}

We can now complete what we started with Lemmas~\ref{lem:clusterfueCfree} and~\ref{lem:clusterfueC}.

\begin{lem} \label{lem:clusterfue}
The concatenation of two adjacent clusters of a clustering is never a cluster.

More precisely, fix consecutive elements $c > b > a$ of a preclustering $A$ of $f$, and suppose that the concatenation $f_{[c,a)}$ of the two adjacent clusters $f_{[c,b)}$ and $f_{[b,a)}$ is itself a cluster. Then $A$ has a gate at $b$, and this gate satisfies one of Conditions (1), (2), or
(3) in Definition~\ref{def:fakedef} of ``fake gates''.
 \end{lem}

\begin{proof}
If the clusters $f_{[c,b)}$ and $f_{[b,a)}$ are of the same kind ($\mathsf{C}$ or $\mathsf{C}$-free), then Lemmas~\ref{lem:clusterfueCfree}
and~\ref{lem:clusterfueC} show that $A$ has a two-way gate at $b$.

Consider now the case that these clusters are of different kinds.
Since one of $f_{[c,b)}$ and $f_{[b,a)}$ is a $\mathsf{C}$~cluster, one of the factors $f_i$ is an odd-degree Chebyshev polynomial, so the whole $f_{[c,a)}$ must be a $\mathsf{C}$~cluster. By Remark~\ref{rk:preclustermk}, the other cluster must consist of a single quadratic factor. The second part of Lemma~\ref{lem:purposeofgates} now verifies condition (2) or (3) in the definition of ``fake gates''.
\end{proof}

We can now connect our clusterings with those in \cite{MS}. The following looks much like Remark~\ref{rk:skew-cluster-rmk}, except that ``pre'' is gone from ``preclusterings''.

\begin{Rk} \label{rk:clustering-def-match}
These observations are again related to parts
of~\cite[Lemma 6.11]{MS}.
\begin{enumerate}
\item If $A$ is a clustering of $f$, and $0 \in A$, then $A \cap \{ 0, 1, \ldots k \}$ is a clustering of $(f_k, \ldots f_1)$ in the sense
of~\cite[Definition 4.26]{MS}. Lemma~\ref{lem:clusterfue} verifies the first part of the definition; and~\cite[Remark 4.27]{MS}  connects the second part of the definition to parts (4) and (5) in our definition of ``fake gates''.
\item More generally, for any $a \in A$, the tuple $A \cap \{ a, a+1, \ldots a+k \}$ is a clustering of the
 decomposition $(f_{a+k}, \ldots f_{a+1})$, in the sense of~\cite[Definition 4.26]{MS}.
\item Converses of these may fail: the cluster boundary at $0$ and $k$ required in~\cite[Definition 4.26]{MS} may not be right for the long decomposition.
\end{enumerate} \end{Rk}

So far in this section we have identified two ways of getting a new preclustering from an old one: fusing two clusters (in each period) and removing the boundary between them from $A$; and moving a quadratic factor from one cluster to the next, through a gate in the correct direction. The next two results examine this second operation. Each time, we have to separately consider the special case of one cluster boundary per period. Empty preclusterings are irrelevant to this story. The extreme case $k=1$ is also irrelevant to this story.

The next two results characterize the essential non-uniqueness of clusterings: a quadratic factor can cross through a gate from one cluster to another. The rest of this section is devoted to showing that this is all that can happen, and that this can only happen once at any particular gate.

\begin{lem} \label{lem:lem421}
Fix a non-empty preclustering $A$ of $f$ and consecutive elements $c > b > a$ of $A$.

\textbf{Part 1.} Suppose that $b+1 \notin A$, and $(A,f)$ has a one-way left-to-right gate at $b$, and $f_{b+1}$ is quadratic. Let $\alpha: A \rightarrow \ZZ$ be defined by
 $$\alpha(i) := \begin{cases}
    i+1 &\mbox{if } i \equiv b \pmod{k}\\
    i &\mbox{otherwise. }\end{cases}$$
Then $B := \operatorname{range} (\alpha)$ is another preclustering of $f$.

Furthermore, for all adjacent $d>e$ in $A$, the cluster $f_{[d,e)}$ of $(A,f)$ and the cluster $f_{[\alpha(d),\alpha(e))}$ of $(B,f)$ are of the same kind ($\mathsf{C}$ or $\mathsf{C}$-free).

Further yet, for all $e \in A$ with $e \not\equiv b \pmod{k}$, the gate of $(A,f)$ at $e$ and the gate of $(B,f)$ at $\alpha(e)$ are of the same kind (none, one-way left-to-right, one-way right-to-left, or two-way).

Finally, the direction of the moved gate is changed: $B$ has a one-way right-to-left gate at $b+1 = \alpha(b)$.

\vspace{.5cm}
\textbf{Part 2.}
The reverse is also true: If $b-1 \notin A$, and $(A,f)$ has a one-way right-to-left gate at $b$, and $f_{b}$ is quadratic; then $\alpha: A \rightarrow \ZZ$ defined by
 $$\alpha(i) := \begin{cases}
    i-1 &\mbox{if } i \equiv b \pmod{k} \\
    i &\mbox{otherwise }\end{cases}$$
produces another preclustering $B := \operatorname{range} (\alpha)$ of $f$.

Again, the direction of the moved gate is changed: $B$ has a one-way left-to-right gate at $b-1 = \alpha(b)$.

Again, all other gates, and all cluster types are unchanged, exactly as in the two ``furthermore''s in Part 1.
\end{lem}

Our proof of this lemma is rather long, but routine.

\begin{proof}
We prove Part 1; the proof of Part 2 is identical and left to the reader.

Since $b+1 \notin A$, this $\alpha$ respects order and $k$-periodicity.
Since $b \in A$ and $b+1 \notin A$ and $A$ is $k$-periodic, it follows that $k \neq 1$.
Setting
$$C := A \smallsetminus ( b + k\mathbb{Z})$$
we get
$$ A = C \cup ( b + k\mathbb{Z})\mbox{ and }B =
C \cup (b+1+k\mathbb{Z}).$$

To show that $B$ is a preclustering and that $\alpha$ respects cluster types, we only need to consider alleged clusters of $B$ with the new boundary $b+1$.

\textit{Case 1.}
If $A$ has more than one cluster boundary per period, these new clusters are $f_{[c,b+1)}$ and $f_{[b+1, a)}$, corresponding to the clusters $f_{[c,b)}$ and $f_{[b, a)}$ of $A$.
Since $c \in A$ and $b+1 \notin A$, we know that $c \neq b+1$, so $f_{[c,b+1)}$ is not empty.
By the second statement in Lemma~\ref{lem:purposeofgates}, $f_{[b+1, a)}$ is a cluster.
Now $f_{[c,b+1)}$ is the result of removing the \emph{quadratic} factor $f_{b+1}$ from the cluster $f_{[c,b)}$, so by part (3) of Remark~\ref{rk:preclustermk} it is a cluster, of the same kind as $f_{[c,b)}$.
By the same part (3) of Remark~\ref{rk:preclustermk}, the cluster $f_{[b, a)}$ is of the same kind as $f_{[b+1, a)}$.

\textit{Case 2.}
If $A$ has one cluster boundary per period, $A  = b + k\mathbb{Z}$ and $a = b-k$ and $B = b + 1 + k\mathbb{Z}$.
By the second statement in Lemma~\ref{lem:purposeofgates}, $f_{[b+1, a)}$ is a cluster; by part (3) of Remark~\ref{rk:preclustermk} applied at its left end $b$, it is of the same kind as $f_{[b, a)}$. By part (3) of
Remark~\ref{rk:preclustermk} applied at the right end $a$ of $f_{[b+1, a)}$, we get that  $f_{[b+1, a+1)}$ is a cluster of the same kind as $f_{[b+1, a)}$. So $f_{[b+1, a+1)}$, the only alleged cluster per period of $B$, is indeed a cluster, of the same kind as  $f_{[b, a)}$, the only cluster per period of $A$.

It remains to show that $\alpha$ respects the gates as described. We fix a cleanup $(h, L)$ of $(A,f)$ and build a cleanup $(g, M)$ of $(B,f)$ with
$$g_i = h_i \text{ for } i \not\equiv b \pmod{k} \text{ and } M_i = L_i \text{ for } i, i+1 \not\equiv b \pmod{k} \text{ .}$$
We verify that the new gate at $(b+1)$ is one-way right-to-left as we build the cleanup.
The types of all other gates remain the same: the linear factors $M_i = L_i$ in the cleanups are the same, and we already showed that the types of clusters are the same.

To lighten notation, we talk about adjusting $L_{b+1}$ and $L_b$ and $h_{b+1}$; corresponding entries in other periods are correspondingly adjusted. Since $k \neq 1$, the linear $L_{b+1}$ and $L_b$ are not corresponding factors from different periods, and can be adjusted independently.
Since $b+1 \notin A$, the old linear factor $L_{b+1}$ is identity.
Since $b \notin B$, the new linear factor $M_b$ must be identity.

\textit{Case 1.} If both $f_{[c,b)}$ and $f_{[b, a)}$ are $\mathsf{C}$-free clusters, the gate at $b$ cannot be one-way, so the hypotheses of this Lemma cannot be satisfied.

\textit{Case 2.}
If both $f_{[c,b)}$ and $f_{[b, a)}$ are $\mathsf{C}$~clusters, then $h_{b+1}(x) = x^2 -2$ and $L_b(x) = \lambda^{-1} x$.
We can keep $h_{b+1}$, but need a new linear $M_{b+1}$ such that $M_{b+1} \circ h_{b+1} = h_{b+1} \circ L_b$.
This is exactly the defining property of $M_{b+1} = A_\lambda$ from~\cite[Remark 3.14]{MS}
So get $M$ from $L$ by replacing $L_{b+1} = \id$ by $M_{b+1} = A_\lambda$, and replacing $L_b(x) = \lambda^{-1} x$ by $M_b = \id$; and propagating through $k$-periodicity.
Now $(h, M)$ is a cleanup of $(B,f)$ with a right-to-left gate at $b+1$, as wanted.
If this gate is two-way, $M_{b+1} = A_\lambda$ must also be a scaling, forcing $\lambda = \pm 1$, forcing the original gate of $(A,f)$ at $b$ to be two-way, which we assumed it is not.

\textit{Case 3.}
Suppose that $f_{[c,b)}$ is a $\mathsf{C}$-free cluster, so $h_{b+1}(x) = x^2$, and $f_{[b, a)}$ is a $\mathsf{C}$~cluster.
For $(A,f)$ to have a left-to-right gate at $b$, we must have $L_b = \id$.
For the new cleanup, we need $g_{b+1}(x) = x^2 -2$. Let the new linear factor $M_{b+1}(x) := x+2$.
Clearly, $M_{b+1} \circ g_{b+1} = h_{b+1} \circ L_b$.
Now $(g, M)$, obtained from $(h, L)$ by replacing $h_{b+1}$ by $g_{b+1}$ and $L_{b+1}$ by $M_{b+1}$,
is a cleanup of $(B,f)$. It has a right-to-left gate at $b+1$ as long as $h_{b+2} \circ M_{b+1} \circ (-2)$ is a Ritt polynomial; which it is, since $M_{b+1} \circ (-2) = \id$ and $h_{b+2}$ is the rightmost of the non-empty $\mathsf{C}$-free cluster $f_{[c,b+1)}$ of $(B,f)$. Since this gate is between clusters of different kinds, it must be one-way by the third part of Lemma~\ref{lem:purposeofgates}.

\textit{Case 4.}
Suppose that $f_{[c,b)}$ is a $\mathsf{C}$~cluster, so $h_{b+1}(x) = x^2-2$, and $f_{[b, a)}$ is a $\mathsf{C}$-free cluster.
For $(A,f)$ to have a left-to-right gate at $b$, we must have $L_b(x) = \lambda^{-1} x$.
For the new cleanup, we need $g_{b+1}(x) = x^2$ and a new linear $M_{b+1}$ such that
 $M_{b+1} \circ g_{b+1} = h_{b+1} \circ L_b$,
which is exactly the defining property of $M_{b+1} = B_\lambda$.
Now $(g, M)$, obtained from $(h, L)$ by replacing $h_{b+1}$ by $g_{b+1}$ and $L_{b+1}=\id$ by $M_{b+1}$ and $L_b$ by $M_b = \id$, is a cleanup of $(B,f)$, with a right-to-left gate at $b+1$, as desired. Again, since this gate is between clusters of different kinds, it must be one-way by the third part of Lemma~\ref{lem:purposeofgates}.
\end{proof}

We now verify that when a clustering is fed to Lemma~\ref{lem:lem421}, a clustering is produced.

\begin{lem} \label{lem:lem432} (Compare with~\cite[Lemma 4.32]{MS})\\
Fix a non-empty clustering $A$ of $f$ and consecutive elements $c > b > a$ of $A$.
Suppose that $(A,f)$ has a one-way left-to-right (respectively, right-to-left) gate at $b \in A$ and that $f_{b+1}$ (resp., $f_b$) is quadratic. 
Let
 $$B := (A \smallsetminus b + k\mathbb{Z}) \cup  b+1+k\mathbb{Z}$$
 $$\text{ (respectively, } B:= (A \smallsetminus \ b+k \mathbb{Z} ) \cup \{ b-1+nk : n \in \mathbb{Z} \} \text{ ).}$$
Then $B$ is also a clustering of $f$.
\end{lem}

\begin{proof}
Again, we leave the ``respectively'' part to the reader.

Since $A$ is a clustering, the gate at $b$ must fail parts (2) and (3) of Definition~\ref{def:fakedef} of ``fake gate'', so $b+1$ is not in $A$. Now all hypotheses of Lemma~\ref{lem:lem421} are satisfied, so we get its conclusions: $B$ is another preclustering of $f$, with the same number of clusters per period as $A$, and with a one-way right-to-left gate at $b+1$. Corresponding clusters of $B$ and $A$ are of the same kind, $\mathsf{C}$ or $\mathsf{C}$-free; and all the ``other'' cluster boundaries have gates in the same direction(s).
As in that lemma, let
$$C := A \smallsetminus b + k\mathbb{Z} \text{ to get }
A = C \cup b + k\mathbb{Z} \text{ and } B = C \cup b+1+k\mathbb{Z} \text{ .}$$

Let us verify that $B$ has no fake gates.\begin{enumerate}
\item Since $A$ is a clustering, it has no two-way gates. The new gate of $B$ at $b+1$ is one way, and all other gates of $B$ are in the same direction(s) as the corresponding gates of $A$. So $B$ also has no two-way gates.
\item Where could $B$ have a left-to-right gate at $d$ with a quadratic $f_{d+1}$ and a cluster boundary at $d+1$?
 If $d, d+1 \in C$, then this contradicts $A$ being a clustering.
 Since $B$ has a one-way gate in the other direction at $b+1$, we cannot have $d \equiv b+1 \pmod{k}$.
 So $d+1 \equiv b+1 \pmod{k}$; without loss of generality, $d+1 = b+1$.
 But then $b=d$ is a cluster boundary of $B$, a contradiction.
\item Where could $B$ have a right-to-left gate at $d$ with a quadratic $f_{d}$ and a cluster boundary at $d-1$?
 Again, $d, d-1 \in C$ contradicts $A$ being a clustering. Again, $(d-1, d) \equiv (b, b+1)
 \pmod{k}$ contradicts $b \notin B$.
 So we must have $d-1 \equiv b+1 \pmod{k}$.\\

 Then $d \equiv b+2 \pmod{k}$ is in $B$; since $k \neq 1$, this means that $b+2 \in C$, and $A$ also has a right-to-left gate at $b+2 \equiv d \pmod{k}$. Together with $A$'s original left-to-right gate at $b$ and quadratics $f_{b+1}$ and $f_{d} = f_{b+2}$, this makes $A$'s gates at $b$ and $b+2$ fake by parts (4) and (5) of Definition~\ref{def:fakedef}.
\item[(4,5)] Finally, where could $(B,f)$ have quadratic factors $f_d$ and $f_{d+1}$ with a right-to-left gate at $d+2$ and a left-to-right gate at $d$?
 Again, $d, d+2 \in C$ immediately  contradicts $A$ being a clustering.
 Again, since $B$ has a one-way right-to-left gate at $b+1$, we cannot have $d \equiv b+1 \pmod{k}$.
 So we must have $d+2 \equiv b+1 \pmod{k}$.
 Now $d \equiv b-1  \pmod{k}$ is in $C$, so $A$ also has a left-to-right gate at $d \equiv b-1 \pmod{k}$.
 However, $b \equiv d+1 \pmod{k}$ is also a cluster boundary of $A$, and $f_{d+1}$ is a quadratic, making the $A$'s gate at $d$ fake by part (2) of Definition~\ref{def:fakedef}; contradicting $A$ being a clustering.
\end{enumerate}
\end{proof}

We introduce the boundary set, a slight variant
of a clustering, replacing  each ambiguous boundary, at $b$ or $b+1$, with the half-integer $b+\frac{1}{2}$ lying between the two options for the boundary.

\begin{Def} \label{def:def-bdry}
For a clustering $A$ of a \kld $f$, the \emph{boundary set} is the subset $B$ of $\frac{1}{2}\ZZ$ such that, for any integer $i$,
\begin{enumerate}
\item $i - \frac{1}{2} \in B$ if and only if the factor $f_i$ is quadratic, and Lemma~\ref{lem:lem421} applies:\\
 either $i \in A$ and $(A,f)$ has a right-to-left gate at $i$;\\
 or  $i-1 \in A$ and $(A,f)$ has a left-to-right gate at $i$.
\item $i \in B$ if and only if $i \in A$ and
Lemma~\ref{lem:lem421} does not apply.
\end{enumerate}

A \emph{boundary list} $\alpha$ of a \kld $f$ is an increasing bijection from $\ZZ$ to a boundary set $B \subset \frac{1}{2}\ZZ$ of some clustering
$(A,f)$. \end{Def}

\begin{Rk}
\label{rk:boudary-list-gate-swap}
 When a new clustering is obtained from
another one by moving a quadratic factor from one
cluster to another as in Lemma~\ref{lem:lem421}, the
original clustering and the new one have the same
boundary set.  We will show with
Theorem~\ref{thm:unique-brdyset} that this only
way to produce new clusterings from old so that
ultimately the boundary set is an invariant of the
long decomposition independent of a choice of clustering.
\end{Rk}

\begin{Rk}
\label{rk:boundary-list-ST}
The boundary list will help us to keep track of
how the gates move as $\STk$ acts on long decompositions.  Annoyingly, this introduces a new non-uniqueness: the starting point $\alpha(0)$.	
\end{Rk}

To prove our Theorem~\ref{thm:unique-brdyset} that
the boundary set is an invariant of a long decomposition
we require several technical lemmas describing
how clusters may overlap.  These results are
analogous to the work in~\cite[Lemmas 4.23, 4.24, and 4.25]{MS}.

The next lemma says that is a cluster can bite off a factor from an adjacent cluster, either the two clusters can fuse (so this cannot happen in a clustering), or Lemma~\ref{lem:lem421} applies.

\begin{lem} \label{lem:bite-one}
Fix consecutive elements $c>b>a$ of a preclustering $A$ of a \kld $f$.
Suppose that the concatenation $f_{[b+1,a)}$ of this cluster $f_{[b,a)}$ with an adjacent factor $f_{b+1}$ is a cluster.
Then exactly one of the following holds\begin{enumerate}
\item $(A,f)$ has a two-way gate at $b$;
\item $c=b+1$ and $f_{b+1}$ is quadratic;
\item $b+1 \notin A$ and $f_{b+1}$ is quadratic and $(A,f)$ has a one-way left-to-right gate at $b$.
\end{enumerate}
The same is true on the other side: if $f_{[b,a-1)}$ is a cluster, $(A,f)$ has a two-way gate at $a$; or $f_a$ is a quadratic and a single cluster; or $f_a$ is quadratic and $a-1 \not\in A$ and $(A,f)$ has a right-to-left gate at $a$.
\end{lem}

\begin{proof}
Again, we only prove one side.
\textit{Case 1.} If $f_{b+1}$ is not quadratic, then the two clusters $f_{[c,b)}$ and $f_{[b,a)}$ are of the same kind, $\mathsf{C}$~or $\mathsf{C}$-free.

\textit{Case 1.1.} If $f_{[c,b)}$ is a $\mathsf{C}$-free cluster, then (by Lemma~\ref{lem:refinery}) $A' := A \cup (b+1 + k\ZZ)$ is also a preclustering of $f$, and any cleanup of $(A,f)$ is also a cleanup of $(A',f)$. Now Lemma~\ref{lem:clusterfueCfree} applies to the concatenation $f_{[b+1,a)}$ of the clusters $f_{[b+1,b)}$ and $f_{[b,a)}$ of $(A', f)$. So $(A', f)$ has a two-way gate at $b$, and so $(A,f)$ has a two-way gate at $b$.

\textit{Case 1.2.} If $f_{[c,b)}$ is a $\mathsf{C}$~cluster, and $(L,h)$ is a cleanup of $(A,f)$, then all $h_i$ with $c \geq i > a$ are Chebyshev polynomials. In order to have linear $M$ and $N$ for which
 $$(M \circ h_c, h_{c-1}, \ldots, h_{b+1}, L_b \circ h_b \circ N)$$
is linearly equivalent to
 $$(h_c, h_{c-1}, \ldots, h_{b+1}, \circ h_b) \text{, }$$
we must have $L_b(x) = \pm x$, since Chebyshev polynomials are not otherwise linearly related to themselves (see \cite[Section 3]{MS} for details
of this well-known fact). So again, $(A,f)$ has a two-way gate at $b$.

\textit{Case 2.} If $f_{b+1}$ is quadratic, by the second part of Lemma~\ref{lem:purposeofgates}, $(A,f)$ has a left-to-right gate at $b$.
 If this gate is two-way, we get conclusion (1) of the lemma.
 If $b+1 \in A$, we get conclusion (2) in the lemma.
 If neither of these holds, we get conclusion (3) of the lemma.
\end{proof}

It follows that a cluster cannot bite off more than one factor from an adjacent cluster, unless the two clusters fuse. In particular, this cannot happen in a clustering.

\begin{lem} \label{lem:bite-two}
Fix consecutive elements $c>b>a$ of a preclustering $A$ of a \kld $f$, and suppose that $c \geq b+2$.
Suppose that the concatenation $f_{[b+2,a)}$ of this cluster $f_{[b,a)}$ with two adjacent factors $f_{b+2}$ and $f_{b+1}$ is a cluster. Then the whole $f_{[c,a)}$ is a cluster. Furthermore, $(A,f)$ has a two-way gate at $b$, or $f_{[c,b)}$ is a $\mathsf{C}$~cluster and $b = a+1$ and $f_b$ is quadratic.

The same is true on the other side: if $b-2 \geq a$ and $f_{[b,a-2)}$ is a cluster, and $d < a$ is the next element of $A$, then the concatenation $f_{[b,d)}$ of two clusters is a cluster; and $(A,f)$ has a two-way gate at $b$, unless
$f_{[a,d)}$ is a $\mathsf{C}$~cluster, and $b = a+1$ and $f_b$ is quadratic.
\end{lem}

\begin{proof}
\textit{Case 1.} If $f_{[b+1,a)}$ is also a cluster, then Lemma~\ref{lem:bite-one} applies, and conclusion (2) cannot hold, and conclusion (1) is what we want. Suppose toward contradiction that conclusion (3) holds. Now Lemma~\ref{lem:lem421} applies, providing us with a new preclustering $B$ of $f$, with a cluster boundary at $(b+1)$ instead of $b$, with a one-way gate the wrong way at that gate. Since $f_{[b+2,a)}$ is a cluster, Lemma~\ref{lem:bite-one} also applies to this new preclustering $B$. But conclusion (1) needs a two-way gate where we have a one-way gate; and, by the second part of Lemma~\ref{lem:purposeofgates}, conclusions (2) and (3) need a one-way gate in the other direction at $(b+1)$.

\textit{Case 2.} If $f_{[b+1,a)}$ is not a cluster, then Remark~\ref{rk:preclustermk} applies, since $f_{[b+1,a)}$ is contained in the cluster $f_{[b+2,a)}$, and contains the cluster $f_{[b,a)}$. This means that $f_{[b+2,a)}$ must be a $\mathsf{C}$~cluster, and all factors of $f_{[b+1,a)}$ must be quadratic, and $f_{[b,a)}$ must be a $\mathsf{C}$-free cluster consisting of a single quadratic factor.

So we have: $b = a+1$, and $f_b$ and $f_{b+1}$ are quadratic, and $f_{b+2}$ is linearly related to an odd-degree Chebyshev polynomial. In particular, $f_{[c,b)}$ is a $\mathsf{C}$~cluster. Fix a cleanup $(h,L)$ of $(A,f)$. Then $h_{b+2}$ is an odd-degree Chebyshev polynomial, $h_{b+1}(x) = x^2-2$ is the Chebyshev polynomial of degree 2, and $h_b(x) = x^2$. We are interested in the linear factor $L_b$. For $f_{[b+2,a)}$ to be a cluster, we must have linear $M$ and $N$ with
 $$(M \circ h_{b+2}, x^2-2, L_b \circ x^2 \circ N)$$
is linearly equivalent to
 $$(h_{b+2}, x^2-2, x^2-2).$$
By \cite[Remark 3.14]{MS}  about the linear relations between $x^2$ and $x^2 -2$, this exactly corresponds to the definition of ``$(A,f)$ has a right-to-left gate at $b$''. Now the second part of Lemma~\ref{lem:purposeofgates} shows that the concatenation of the cluster $f{[c,b)}$ and the adjacent quadratic factor $f_b$ is a cluster. Since the whole cluster $f_{[b,a)}$ consists of this one factor, we have shown that $f_{[c,a)}$ is a cluster.
\end{proof}
 
With the next lemma we show that two overlapping
$\mathsf{C}$-clusters for which the intersection
contains all the odd degree factors fuse into a
single $\mathsf{C}$-cluster.  This result provides
some necessary details for the proof
of~\cite[Lemma 4.25]{MS}.

\begin{lem} \label{lem:eat-many-quad}
Fix a \kld $f$ and integers $d > c > b > a$.
Suppose that $f_{[d,b)}$ and $f_{[c,a)}$ are both $\mathsf{C}$~clusters,
and all factors of both $f_{[d,c)}$ and $f_{[b,a)}$ are quadratic.
Then the whole $f_{[d,a)}$ is a $\mathsf{C}$~cluster.
\end{lem}

\begin{proof}
We first replace $k$ by a sufficiently big multiple to ensure that $k \gg d-a$.

The proof proceeds by induction on $(b-a)$. In the base case $b-a =0$,
 the desired cluster $f_{[d,a)}$ in the conclusion is
 one of the clusters $f_{[d,b)}$ in the hypothesis.

We consider three preclusterings $A$, $B$, and $D$ of $f$.
All clusters of $A$, $B$, and $D$ consist of one factor each, except for $f_{[c,a)}$ for $A$,
$f_{[c,b)}$ for $B$, and $f_{[d,b)}$ for $D$.

We build two cleanups $(h,L)$ and $(g, M)$ of $(B, f)$. 

To build $(h, L)$, we start with a cleanup $(\hat{h}, \hat{L})$ of $(A,f)$.
So, $\hat{h}_i$ are Chebyshev polynomials when $c \geq i > a$; and $\hat{L}_i = \id$ when $c > i > a$.
Let $h_i(x) := x^2$ and $L_i(x) := x-2$ for $i$ with $b \geq i > a$; and let $h_i = \hat{h}_i$ and $L_i = \hat{L}_i$ for other (modulo $k$) $i$.

Similarly, to build $(g, M)$, start with a cleanup $(\hat{g}, \hat{M})$ of $(D,f)$ and
take $M_i(x) = x -2$ and $g_i(x) = x^2$ for $i$ where $d \geq i > c$, as in the previous
case. Because we do not know anything about the factor $f_{d+1}$,
it may be necessary to further adjust $g_{d+1}$ and $M_d$ to satisfy part (5b) of the definition of ``cleanup''.

We know $(h, L)$ has a one-way right-to-left gate at $b$.
By part (1) of Lemma~\ref{lem:purposeofgates}, it follows that so does $(g, M)$.
Since $M_b = \hat{M}_b$, it also follows that the cleanup $(\hat{g}, \hat{M})$ of the preclustering $D$ of $f$ has a one-way right-to-left gate at $b$.
By part (2) of Lemma~\ref{lem:purposeofgates}, we know that $f_{[d, b-1)}$ is a cluster.
The induction step is now complete.
\end{proof}

We come now to the main technical lemma explaining
exactly under what conditions two clusters may overlap.
We will
use this result repeatedly to compare clusterings of the same
\kld and then employ it later to understand the action
of Ritt swaps.

\begin{lem} \label{lem:clusteroverlap}
Fix a \kld $f$ and integers $d > c > b > a$. Suppose that $f_{[d,b)}$ and $f_{[c,a)}$ are both clusters. Then either the whole $f_{[d,a)}$ is a cluster; or $c = b+1$ and $f_{b+1}$ is quadratic and at least one of $f_{[d,b)}$ and $f_{[c,a)}$ is a $\mathsf{C}$~cluster. \end{lem}

\begin{proof}
The three small pieces $f_{[d,c)}$ and $f_{[c,b)}$ and $f_{[b,a)}$, are almost always clusters themselves.
We first deal with this general case, and then use Remark~\ref{rk:preclustermk} to analyze situations where some of the pieces are not clusters.

\textit{Case 1.} All three $f_{[d,c)}$ and $f_{[c,b)}$ and $f_{[b,a)}$ are clusters.

 To apply our previous Lemmas~\ref{lem:bite-one} and~\ref{lem:bite-two}, we need a preclustering $A$ of $f$ for which these $d > c > b > a$ are consecutive elements of $A$. We can always do this, possibly at the expense of replacing $k$ by a suitably big multiple: the conclusion of this lemma is completely local, so this is not a problem.

 Now Lemma~\ref{lem:clusterfue} applies at each of the two boundaries $c$ and $b$. If both gates are two-way, the whole $f_{[d,a)}$ is a cluster. If the outer clusters are single quadratics that join the middle cluster via one-way gates, they can both join, again making the whole $f_{[d,a)}$ a cluster. The only other option is that the middle cluster is the single quadratic, and at least one of the gates is one-way. Since there are no one-way gates between $\mathsf{C}$-free clusters, at least one of the clusters must be a $\mathsf{C}$~cluster.

\textit{Case 2.}
If the middle piece $f_{[c,b)}$ is not a cluster, then both $f_{[d,b)}$ and $f_{[c,a)}$ are $\mathsf{C}$~clusters and the overlap $f_{[c,b)}$ consists of quadratic factors, so the outer pieces $f_{[d,c)}$ and $f_{[b,a)}$ still have odd-degree factors, and so they are clusters. In fact, $f_{[d,i)}$ and $f_{[i,a)}$ are clusters for any $i$ with $c \geq i \geq b$, since they are contained in a cluster and contain an odd-degree factor. If $c = b+1$, we are done. Otherwise, we can use Lemma~\ref{lem:bite-two}, as soon as we have a preclustering $A$ of $f$ for which these $d > b > a$ are consecutive elements. As in Case 1, we can always find such $A$.
We now have two adjacent clusters $f_{[d,b)}$ and $f_{[b,a)}$, and the $f_{[b+2,a)}$ is also a cluster. Since we know that $f_{[b,a)}$ contains an odd-degree factor, the whole $f_{[d,a)}$ must be a cluster by Lemma~\ref{lem:bite-two}.

\textit{Case 3.}
If one of the outer pieces, $f_{[d,c)}$ or $f_{[b,a)}$,
is not a cluster, it means that the corresponding big cluster, $f_{[d,b)}$ or $f_{[c,a)}$, respectively,
is a $\mathsf{C}$~cluster whose odd-degree factors are inside the middle piece $f_{[c,b)}$. Then both $f_{[d,b)}$ and $f_{[c,a)}$ are $\mathsf{C}$~clusters.

If $f_{[b,a)}$ is a cluster, we can apply the reasoning in Case 2 again: $f_{[d,b)}$ and $f_{[b,a)}$ are adjacent clusters, and $f_{[i,a)}$ is a cluster for any $i$ with $b \geq i \geq a$. Applying
Lemmas~\ref{lem:bite-one} and~\ref{lem:bite-two} again, and noting that the middle piece $f_{[c,b)}$ is not a single quadratic, we get that the whole $f_{[d,a)}$ must be a cluster.
The same reasoning applies if $f_{[d,c)}$ is a cluster.

The last remaining possibility that both $f_{[d,c)}$ and $f_{[b,a)}$ consist of nothing but quadratics
is handled by Lemma~\ref{lem:eat-many-quad}.

\end{proof}

The following (surprisingly useful) technical lemma shows that a one-quadratic cluster of a clustering cannot be contained in a $\mathsf{C}$~cluster.

\begin{lem} \label{lem:drown-quad}
If \begin{itemize}
\item $A$ is a preclustering of a \kld $f$; and
\item $a, a-1 \in A$ and $f_a$ is quadratic; and
\item $f_{[e,d)}$ is a $\mathsf{C}$~cluster for some $e, d$ with $e \geq a$ and $a-1 \geq d$;
\end{itemize}
then $A$ is not a clustering. \end{lem}

\begin{proof}
Let $z$ be the index of an odd-degree factor $f_z$ in the $\mathsf{C}$~cluster $f_{[e,d)}$, closest to $a$ on the left or to $a-1$ on the right. We write out the details of the first case; the second is identical.
Let $y>x$ be the boundaries of the cluster $f_{[y,x)}$ of $(A,f)$ that contains the factor $f_z$. So now
 $$y \geq z > x \geq a \text{ and } e \geq z \text{ and } a-1 \geq d; \text{ implying that } x-1 \geq d \text{ .}$$

Since we assumed that $f_z$ is the closest odd-degree factor, all $f_i$ with $x \geq i \geq a$ are quadratic. Remark~\ref{rk:preclustermk} then forces every such $i$ to be a cluster boundary of $A$.

We will show that the concatenation $f_{[y,x-1)}$ of adjacent clusters $f_{[y,x)}$ and $f_{[x,x-1)}$ of $(A,f)$ is itself a cluster, contradicting
Lemma~\ref{lem:clusterfue}. Since $f_{[y,x-1)}$ contains the odd-degree factor $f_z$, by Remark~\ref{rk:preclustermk} it suffices to show that $f_{[y,x-1)}$ is contained in a cluster.

If $e \geq y$, then $f_{[y,x-1)}$ is contained in $f_{[e,d)}$. If $y > e$, then the clusters $f_{[y,x)}$ and $f_{[e,d)}$ overlap, and the overlap includes the odd-degree factor $f_z$. So by
Lemma~\ref{lem:clusteroverlap}, the whole $f_{[y,d)}$ is a cluster, again containing $f_{[y,x-1)}$.
\end{proof}

We now return to proving Theorem~\ref{thm:unique-brdyset}  that the boundary set of a \kld (see Definition~\ref{def:def-bdry}) is unique.
It is easy to see that whenever the empty set is a preclustering of $f$, it is the unique clustering of $f$.
So, we compare two a priori unrelated non-empty clusterings $A$ and $D$ of the same \kld $f$.

The first Lemma~\ref{lem:uniqclusttech} in this series deals with one mismatched cluster boundary. Options (2) and (3) in its conclusion do not produce a mismatch of boundary sets; option (1) is sorted out in the
following Lemma~\ref{lem:no-screwy-quads}.

\begin{lem} \label{lem:uniqclusttech}
Suppose that $A$ and $D$ are non-empty clusterings of the same \kld $f$ and $b \in A \smallsetminus D$.
Then one of the following is true.\begin{enumerate}
\item Both $f_b$ and $f_{b+1}$ are quadratic; and $b+1$ and $b-1$ are both in $D$.
\item $A$ has a one-way left-to-right gate at $b$, and $f_{b+1}$ is quadratic,\\ and $b+1 \in D$, and
Lemma~\ref{lem:lem432} gives another clustering $$B = (A \smallsetminus b+k\mathbb{Z}) \cup b+1+k\mathbb{Z}$$ of $f$ which shares the cluster boundary $b+1$ with $D$.
\item $A$ has a one-way right-to-left gate at $b$, and $f_{b}$ is quadratic, $b-1 \in D$, so
Lemma~\ref{lem:lem432} gives another clustering $B = (A \smallsetminus b+k\mathbb{Z}) \cup b-1+k\mathbb{Z}$ of $f$ which shares the cluster boundary $b-1$ with $D$.
\end{enumerate} \end{lem}

\begin{proof}
Fix consecutive elements $c > b > a$ of $A$.\\
Since $b \notin D$, there are  consecutive elements
$e > d$ of $D$ with $e > b > d$.\\
Comparing $e$ to $c$ and $d$ to $a$ breaks this proof into four cases.
In each case, we get a contradiction (via
 Remark~\ref{rk:preclustermk},
 Lemma~\ref{lem:drown-quad},
 Lemma~\ref{lem:clusteroverlap},
 and Lemma~\ref{lem:clusterfue})
 or one of the desired conclusions
 (via Lemma~\ref{lem:purposeofgates} and
 Lemma~\ref{lem:lem432}).

\textit{Case 1.} If $e \geq c$ and $a \geq d$, then the concatenation $f_{[c,a)}$ of two clusters of $(A,f)$ is contained in the cluster $f_{[e,d)}$ of $(D, f)$.
By Lemma~\ref{lem:clusterfue}, this concatenation
$f_{[c,a)}$ cannot be a cluster itself.
Remark~\ref{rk:preclustermk} then
shows that $c = b+1 = a+2$, and $f_b$ and $f_{b+1}$ are
quadratic, and $f_{[e,d)}$ is a $\mathsf{C}$~cluster,
and Lemma~\ref{lem:drown-quad}
provides the contradiction.

\textit{Case 2.} If $e \geq c$ and $b > d > a$, then
Lemma~\ref{lem:clusteroverlap} applies to the clusters
$f_{[e,d)}$ and $f_{[b, a)}$.

If the whole $f_{[e,a)}$ is a cluster, then $f_{[c,a)}$
is contained in a cluster, and contains the cluster
$f_{[b,a)}$ with at least two factors: $b > d > a$
implies that $b \geq a+2$. Remark~\ref{rk:preclustermk}
then shows that $f_{[c,a)}$ is a cluster, contradicting
Lemma~\ref{lem:clusterfue}.

Thus $b = d+1$ and $f_{b}$ is quadratic. We want to show that the concatenation $f_{[c,b-1)}$ of the cluster $f_{[c,b)}$ of $(A,f)$ and the adjacent quadratic $f_b$ is a cluster. This $f_{[c,b-1)}$ is contained in the cluster $f_{[e,d)} = f_{[e,b-1)}$, and contains the cluster $f_{[c,b)}$.
Remark~\ref{rk:preclustermk} gives what we want, unless $c = b+1$ and $f_c$ is quadratic and $f_{[e,d)}$ is a $\mathsf{C}$~cluster. In that case,
Lemma~\ref{lem:drown-quad} provides the contradiction.

\textit{Case 3.} The situation where $c > e$ but $a \geq d$ works exactly like Case 2.

\textit{Case 4.} Finally, suppose that $c > e > b > d > a$.
Lemma~\ref{lem:clusteroverlap} now applies to $f_{[e,d)}$ with $f_{[b, a)}$, and to $f_{[c,b)}$ with $f_{[e, d)}$.
Each of these two instances of
Lemma~\ref{lem:clusteroverlap} produces one of two conclusions, breaking this case into four subcases.

\textit{Subcase 4.1.} Suppose that both concatenations $f_{[e,a)}$ and $f_{[c,d)}$ are clusters. Their overlap $f_{[e,d)}$ contains at least two factors, so another application of
Lemma~\ref{lem:clusteroverlap} shows that the whole $f_{[c,a)}$ is a cluster, contradicting
Lemma~\ref{lem:clusterfue}.

\textit{Subcase 4.2.} Suppose that $b = d+1$ and $f_b$ is quadratic, but
 the overlap $f_{[e,b)}$ of $f_{[c,b)}$ with $f_{[e, d)}$ is not a single quadratic and $f_{[c,d)}$ is a cluster.
 Since the concatenation $f_{[c,d)}$ of the cluster $f_{[c,b)}$ and the adjacent quadratic $f_b$ is a cluster,
Lemma~\ref{lem:purposeofgates} shows that $(A,f)$ has a right-to-left gate at $b$.
As in Case 2 above, Lemma~\ref{lem:lem432} now gives us another clustering $B = (A \smallsetminus b+k\mathbb{Z}) \cup b-1+k\mathbb{Z}$ of $f$ which shares the cluster boundary $d = b-1$ with $D$.

\textit{Subcase 4.3.} The reasoning is exactly the same for the symmetric situation
 where $e = b+1$ and $f_{b+1}$ is a quadratic,
 but the overlap $f_{[b,d)}$ of $f_{[e,d)}$ and $f_{[b,a)}$ is not a single quadratic.

\textit{Subcase 4.4.} Finally, we arrive at the first option in the conclusion of this lemma: $f_b$ and $f_{b+1}$ are both quadratic, and $e = b+1$ and $d = b-1$.
\end{proof}

The last lemma was the induction step for the proof of Theorem~\ref{thm:unique-brdyset}, and the next lemma is the base case, sorting out option (1) in the conclusion of the last lemma.

\begin{lem} \label{lem:no-screwy-quads}
Fix two clusterings $A$ and $D$ of the same \kld $f$. Suppose that
 for every $b \in A \smallsetminus D$, both $f_b$ and $f_{b+1}$ are quadratic; and $b+1$ and $b-1$ are both in $D$.
Suppose symmetrically that
 for every $b \in D \smallsetminus A$, both $f_b$ and $f_{b+1}$ are quadratic; and $b+1$ and $b-1$ are both in $A$.
Then $A = D$.
\end{lem}

\begin{proof}
Fix cleanups $(g, M)$ of $(A,f)$, and $(h, L)$ of $(D,f)$.

Observe first that for every $b \in A \smallsetminus D$, the cluster $f_{[b+1, b-1)}$ of $D$ broken by $b$ is a $\mathsf{C}$-free cluster, since it has degree $4$. By Lemma~\ref{lem:refinery}, this means that $A \cup D$ is another preclustering of $f$, and $(g, M)$ is a cleanup of it. Symmetrically, $(h, L)$ is also a cleanup of $(A \cup D, f)$.

For any $d \in D \smallsetminus A$, the cleanup $(g, M)$ of $(A,f)$ must have $M_d = \id$.
So the cleanup $(g, M)$ of $(A \cup D, f)$ has a two-way gate at such $d$.
So every cleanup of $(A \cup D, f)$, including $(h, L)$, has a two-way gate at such $d$.
So $L_d = \id$. So the cleanup $(h, L)$ of $(D,f)$ has a two-way gate at $d$.
But this fake gate of $(D,f)$ contradicts the definition of ``clustering''.

The same contradiction arises symmetrically from any $b \in A \smallsetminus D$.
\end{proof}

We come now to our Theorem~\ref{thm:unique-brdyset}
showing that the boundary set is an invariant of
a long decomposition, as are the types of
clusters and the gates appearing in a clustering.
We prove the theorem by showing that any two
clusterings may be converted one to the
other through a finite sequence of applications of
Lemma~\ref{lem:lem432}.

\begin{thm} \label{thm:unique-brdyset}
Any two non-empty clusterings $A$ and $D$ of the same \kld $f$ give rise to the same boundary set $B$ for $f$.

Corresponding clusters of $(A,f)$ and $(D,f)$ are of the same kind ($\mathsf{C}$ or $\mathsf{C}$-free).

For each $b \in A \cap D \supseteq B \cap \mathbb{Z}$, the gates of $(A,f)$ at $b$ and $(D,f)$ at $b$ are of the same kind (none, one-way left-to-right, one-way right-to-left).

For each half-integer $b -\frac{1}{2} \in B$ with $b \in A$ and $b-1 \in D$ or vice versa, $A$ and $D$ have one-way gates in opposite direction at this cluster boundary.
\end{thm}

\begin{proof}
We build a finite sequence $(A_i, D_i)$ for $i = 0, 1, 2, \ldots, n$ of pairs of clusterings of $f$ such that
 $A_0 = A$ and $D_0 = D$ and $A_n = D_n$; and for each $i$, either $A_{i+1}$ is obtained from $A_i$ via Lemma~\ref{lem:lem432} and $D_{i+1} = D_i$; or vice versa.
All the conclusions of this theorem follow from the corresponding conclusions of Lemma~\ref{lem:lem421}.
For example, at each step, $A_i$ and $A_{i+1}$ will give rise to the same the boundary set; and $D_i$ and $D_{i+1}$ will also give rise to the same the boundary set. Thus, $A = A_0$ will
have the same boundary set as $A_n = D_n$ which
has the same boundary set as $D = D_0$. 

To build the sequence, we induct on
 $$| ( A_i \bigtriangleup D_i ) \cap \{1, 2, \ldots , k\}| \text{, }$$
the number of cluster boundary disputes between $D_i$ and $A_i$, in one period.

When this number is zero,  $A_i = D_i$ and we are done.

Otherwise, if some $b \in A_i \smallsetminus D_i$ satisfies condition (2) or (3) in
Lemma~\ref{lem:uniqclusttech} with $A = A_i$ and $D=D_i$, that lemma provides $B =: A_{i+1}$ that has one fewer disputes with $D_{i+1} := D_i$.

Otherwise, if some $b \in D_i \smallsetminus A_i$ satisfies condition (2) or (3) in
Lemma~\ref{lem:uniqclusttech} with $A = D_i$ and $D=A_i$, that lemma provides $B =: D_{i+1}$ that has one fewer disputes with $A_{i+1} := A_i$.

Finally, the case that every $b \in A_i \bigtriangleup D_i$ satisfies condition (1) in
Lemma~\ref{lem:uniqclusttech}, is ruled out by
Lemma~\ref{lem:no-screwy-quads}.
\end{proof}

The point of this theorem is that now the number of clusters per period and the kinds ($\mathsf{C}$ or $\mathsf{C}$-free) of corresponding clusters are properties of the \kld alone, independent of the clustering. The existence of one-way gates between corresponding clusters is also a property of the \kld; and if we distinguish three kinds of such gates instead of two, separating out the gates occupied by quadratic that correspond to half-integers in the boundary set, the kinds of one-way gates are also a property of the \kld, independent of the clustering.

\subsubsection{Clusterings and the action of $\STk$}
\label{sec:cluster-action}
In this subsection we focus on the interaction between 
the action of $\STk$ and clusterings.  Our key result
will be that while clusterings themselves are not invariant
under this partial group action, the number of 
clusters, the kind (i.e. type $\mathsf{C}$ or 
$\mathsf{C}$-free), and the presence or absence (but not
the direction) of one-way gates between the are all
invariant under the action of $\Affk$.

 We now collect some easy results from~\cite{MS} about the interactions between the partial functions $t_i \star$ and clusters,
explaining how they fit with our new context of
long decompositions.

\begin{prop} \label{prop:one-swap-and-clusters}
(Compare with \cite[Remark 4.5(5), Proposition 4.11, Lemma 4.17(2), and Lemma 4.31]{MS})
Suppose that $t_i \star [f] = [g]$. That is,
$t_i \star [f]$ is defined and $[f]$ and $[g]$ are Ritt-swap-related at $i$. \begin{enumerate}
\item Then $(f_{i+1}, f_i)$ is a cluster.
\item Moreover, $f$ admits a clustering $A$ with $ i \notin A$.
\item Any (pre)clustering $A$ of $f$ with $ i\notin A$ is also a (pre)clustering of $g$, with the same gates in the same places.
 \end{enumerate}
\end{prop}

\begin{proof}
The first statement is an immediate consequence of the definitions.

For the second statement, fix a clustering $A$ of $f$.
If $i \notin A$, we are done; so suppose that $i \in A$, and let $b > i > a$ be consecutive elements of $A$.
Since tautological Ritt swaps are not allowed, at least one of $f_{i+1}$ and $f_i$ is not quadratic; without loss of generality, suppose that $f_{i}$ is not quadratic. Then by Lemma~\ref{lem:clusteroverlap}, $f_{[i+1, a)}$ is a cluster.
By Lemma~\ref{lem:purposeofgates}, it follows that $f_{i+1}$ is quadratic and $A$ has a left-to-right gate at $i$. By Lemma~\ref{lem:lem432}, $B := (A \smallsetminus  i + k\mathbb{Z}) \cup i+1 + k\mathbb{Z}$ is another clustering of $f$, and $i \notin B$.

To prove (3), we start by showing that a preclustering $A$ of $f$ with $ i \notin A$ is also a
preclustering of $t_i \star [f]$. Fix a cleanup $(h,L)$ of $(A,f)$.
Since $i \notin A$, $f_i$ and $f_{i+1}$ are in the same cluster.

If this is a $\mathsf{C}$~cluster, then $h_i$ and $h_{i+1}$ are Chebyshev polynomials and $h_i \circ h_{i+1} = h_{i+1} \circ h_i$ is a basic Ritt swap. Let
 $$\tilde{h}_i := h_{i+1} \text{ and } \tilde{h}_{i+1} := h_i \text{ and } \tilde{h}_j := h_j \text{ for } j \not\equiv i, i+1 \pmod{k} \text{ .}$$
It is straightforward to verify that $(\tilde{h},L)$ is
a cleanup of $\tilde{g} := L \circ \tilde{h}$. This
$\tilde{g}$
is clearly Ritt-swap-related to $f$ at $i$, so it is linearly equivalent to $g$. Thus, $(\tilde{h}, L)$ is also a cleanup of $g$, and $A$ is also a preclustering of $g$. This argument works exactly the same way when both $h_i$ and $h_{i+1}$ are monomials inside a $\mathsf{C}$-free cluster.

In general, inside a $\mathsf{C}$-free cluster we
need~\cite[Lemma 3.37]{MS} to conclude that there are polynomials $\tilde{h}_i$ and $\tilde{h}_{i+1}$ such that $h_i \circ h_{i+1} = \tilde{h}_{i+1} \circ \tilde{h}_i$ is a basic Ritt swap. It is now slightly less trivial to verify that $(\tilde{h}, L)$ is also a cleanup of $g$: if $(i-1) \in A$ and $\tilde{h}_i$ is not a monomial, one might worry about condition [5b] in the definition of cleanup. This is actually not an issue because the Ritt swap forces $\tilde{h}_i$ to have non-trivial in-degree (see Definition~\ref{def:inoutdeg}. Thus, $\tilde{h}_i$ cannot be a
type $\mathsf{A}$ Ritt polynomial; so $\tilde{h}_i$ and $\tilde{h}_i \circ L_{i-1}$ cannot both be Ritt polynomials for a non-trivial translation $L_{i-1}$.

Since the types of clusters and the linear factors did not change, the two cleanups have the same gates in the same places. Thus, one is a clustering if and only if the other is.

\end{proof}

The next two lemmas are used in the proof of
Theorem~\ref{thm:braidhtm}, an analogue of~\cite[Theorem 2.52]{MS}, that the action of
the generators $t_i$ of $\Affk$ on linear equivalence
classes of \klds satisfy the braid relations.

\begin{lem} \label{lem:lem438} (Compare
with~\cite[Lemma 4.38]{MS})
If $t_i \star [f]$ and $t_{i+1} \star [f]$ are both defined, then $(f_{i+2}, f_{i+1}, f_i)$ is a cluster or $f_{i+1}$ is quadratic.\end{lem}
\begin{proof}
By Proposition~\ref{prop:one-swap-and-clusters},
$(f_{i+1}, f_i)$ and $(f_{i+2}, f_{i+1})$ are both clusters.
Lemma~\ref{lem:clusteroverlap} with $d = i+2$ and $c = i+1$ and $b = i$ and $a = i-1$ finishes the proof. \end{proof}

\begin{lem} \label{lem:lem439} (Compare with~\cite[Lemma 4.39]{MS})
If one of $t_i \star (t_{i+1} \star (t_i \star [f] ))$
or $t_{i+1} \star (t_i \star  (t_{i+1} \star  [f] ))$ is defined, then $(f_{i+2}, f_{i+1}, f_i)$ is a cluster.\end{lem}
\begin{proof}
Suppose that $t_i \star (t_{i+1} \star  (t_i  \star [f] ))$ is defined; the other case is completely analogous.
Let $g$ be a representative of $t_i \star [f]$, and let $h$ be a representative of $t_{i+1} \star (t_i \star [f] )$.
Since $t_{i+1} \star [g] = [h]$ and tautological Ritt swaps are not permitted, at most one of $g_{i+1}$ and $h_{i+1}$ can be quadratic. Lemma~\ref{lem:lem438} applies to both $g$ and $h$, so at least one of
$(g_{i+2}, g_{i+1}, g_i)$ and $(h_{i+2}, h_{i+1}, h_i)$ is a cluster. So
 $$A := (\mathbb{Z} \smallsetminus  i + k\mathbb{Z} ) \cup i+1 + k\mathbb{Z}  $$
is a preclustering of $g$ or $h$.
Therefore, by Proposition~\ref{prop:one-swap-and-clusters} applied once or twice, $A$ is also a preclustering of $f$, and, in particular, $(f_{i+2}, f_{i+1}, f_i)$ is a cluster.
\end{proof}

This next theorem is the key step in proving
Corollary~\ref{cor:group-action-corr} that our
partial action of the generators of $\Affk$ on $\kleq$
is actionable, thus giving a partial group action
of $\Affk$ itself on $\kleq$.

\begin{thm} \label{thm:braidhtm}
Suppose that $k \geq 3$ and $f$ is a \kld and $i$ is an integer.
Then
 $$t_i \star (t_{i+1} \star (t_i \star [f] )) =  t_{i+1} \star (t_i \star (t_{i+1} \star  [f] )) \text{ .}$$
That is, they are equal when defined, and one is defined if and only if the other is. \end{thm}

\begin{proof}
By Lemma~\ref{lem:lem439}, $(f_{i+2}, f_{i+1}, f_i)$ is a cluster, so
 $$A := (\mathbb{Z} \smallsetminus  i +  k\mathbb{Z} ) \cup i+1 + k\mathbb{Z}$$
is a preclustering of $f$. Let $(h,L)$ be a cleanup of $(f, A)$.

The rest of the proof is almost a repeat of the proof of
Proposition~\ref{prop:one-swap-and-clusters}.

If this is happening inside a $\mathsf{C}$~cluster, all three of $h_{i+2}$, $h_{i+1}$, and $h_i$ are Chebyshev polynomials, which commute with
each other.  Setting
$$\tilde{h}_i := h_{i+2} \text{ and } \tilde{h}_{i+2} := h_i \text{ and } \tilde{h}_j := h_j \text{ for } j \not\equiv i, i+2 \pmod{k}$$
gives a cleanup $(\tilde{h}, L)$ of both $t_i \star (t_{i+1} \star (t_i \star [f] ))$
and $t_{i+1} \star (t_i \star (t_{i+1} \star  [f] ))$.
This argument works exactly the same way when all three of $h_i$, $h_{i+1}$, and $h_{i+2}$ are monomials inside a $\mathsf{C}$-free cluster.

Since each of the three factors $h_i$, $h_{i+1}$, and $h_{i+2}$ swaps with each of the others in evaluating either
$t_i \star (t_{i+1} \star (t_i \star [f] ))$ or $ t_{i+1} \star (t_i  \star (t_{i+1} \star [f] ))$, at most one of them can be non-monomial.
Applying~\cite[Lemma 3.37]{MS} repeatedly, as in the proof of Proposition~\ref{prop:one-swap-and-clusters}, finishes this proof.
\end{proof}

Another consequence of
Proposition~\ref{prop:one-swap-and-clusters}
is an analogue of Theorem~\ref{thm:unique-brdyset}:
the number of clusters per period, the kind of clusters,
and the presence of gates at cluster boundaries are
invariant under the partial action of $\Affk$.
The exact location of the gate changes if and only if
the direction of the gate changes.  Action
by powers of $\phi$ shifts the whole picture.

\begin{prop} \label{prop:asf-and-gates}
Fix  $f \in \kleq$ and $w \in \Affk$ with
$w \star f =: g$ defined. 
Then there are boundary lists $\alpha$ and $\beta$ for $f$ and $g$, respectively, such that for all $j$, \begin{enumerate}
\item $\{ \alpha(j), \beta(j) \} \subseteq \{ i+1, i+\frac{1}{2},  i \}$ for some $i$;
\item the $j^\text{th}$ clusters of $f$ and $g$ are of the same kind, $\mathsf{C}$ or $\mathsf{C}$-free;
\item $f$ has a gate at $\alpha(j)$ if and only if $g$ has a gate at $\beta(j)$;
\item if $\alpha(j) = i+1$ and $\beta(j) = i$, then $f$ has a right-to-left gate at $\alpha(j)$ and $g$ has a right-to-left gate at $\beta(j)$; and vice versa.
\end{enumerate}

If $h = \phi^n \star f$ for some $n \in \mathbb{Z}$, there are boundary lists $\alpha$ and $\beta$ for $f$ and $h$, respectively, such that $\beta(j) = \alpha(j+n)$ for all $j$; and corresponding clusters and cluster boundaries are exactly the same.
\end{prop}

\begin{proof}
To prove the first part, we apply Theorem~\ref{thm:unique-brdyset} and Proposition~\ref{prop:one-swap-and-clusters}, and induct on the length of a word in the generators representing $w$.  The way the direction of
gates changes in Theorem~\ref{thm:unique-brdyset} and
Proposition~\ref{prop:one-swap-and-clusters}
prevents the cluster boundaries from drifting by more
than $1$.

The second part about
$\phi^n \star f$ is immediate. \end{proof}

\begin{Def} \label{def:gate-becomes}
With the setup in Proposition~\ref{prop:asf-and-gates}, we say that
 \emph{the cluster boundary $\beta(j)$ comes from the cluster boundary $\alpha(j)$ via $w \star f$.}
When $f$ and $w$ are clear from context, we simply say that $\alpha(j)$ \emph{becomes} $\beta(j)$, or that
$\alpha(j)$ and $\beta(j)$ are \emph{corresponding boundaries}. \end{Def}

\begin{Rk}
\label{rk:weird-gate-becomes}
It is possible for the same $i$ to be a cluster boundary of both $f$ and $g = w \star f$, but not correspond to itself: if $f$ has a one-factor cluster $f_i$ with gates in the same direction on both sides, after a quadratic passes through these gates, both boundaries $\alpha(j) = i$ and $\alpha(j+1) = i+1$ move one step to, say,
$\beta(j) = i+1$ and $\beta(j+1) = i+2$; but $\alpha(j+1) = i+1$ and $\beta(j) = i+1$ are not corresponding boundaries!
\end{Rk}

\begin{prop}
\label{prop:no-gate-wall}
In a clustering, a cluster boundary without a gate is a wall.	
\end{prop}
\begin{proof}
Proposition~\ref{prop:asf-and-gates} shows that
if $f$ has a cluster boundary at $i$ with no gate,
then for every $w \in \Affk$, if defined, $w \star f$
has a cluster boundary at $i$ with no gate.
By Lemma~\ref{lem:purposeofgates}, $t_i \star
(w \star f)$ is never defined.

To finish verifying that $f$ has a wall at $i$, we must show that $\epsilon_p \star f$ and $\epsilon^{-1}_p \star f$ are not defined (see
Definition~\ref{def:epsilon}).

 Suppose toward contradiction that $\epsilon_p \star f$
(respectively, $\epsilon^{-1}_p \star f$) is defined, so, up to linear equivalence, we may assume that each $f_i$ is a Ritt polynomial with in-degree (respectively, out-degree) at least $p$.

 If none of the $f_i$ are type $\mathsf{C}$,
 this gives a cleanup of the empty (pre)clustering of
 $f$. In particular, $f$ has no cluster boundaries.

 Otherwise, $p$ must be $2$, because type~$\mathsf{C}$ Ritt polynomials do not have other in-degrees or out-degrees. To build a cleanup for $f$, we must convert those $f_i$ which are type~$\mathsf{C}$ into actual Chebyshev polynomials. The linear factors used for this are exactly the ones giving one-way gates at cluster boundaries. For more details see \cite[Remark 3.14 and Theorem 3.15]{MS}.
\end{proof}

It follows from Propositions~\ref{prop:warm-up}
and~\ref{prop:no-gate-wall} that Theorem~\ref{thm:skew-inv} holds when there is gateless cluster
boundary.

\begin{cor}
\label{cor:nogate-Mahler}
Theorem~\ref{thm:skew-inv} hold for any polynomial $P$
whose long decomposition has a cluster boundary
with no gate. 	
\end{cor}

\subsubsection{Wandering quadratics}
We have already shown that if a long decomposition
has no walls, then every cluster boundary has
a one-way gate.  In this subsection we will show
that each of these one-way gates is  ``owned by''
a unique ``wandering quadratic''.  This precise
characterization will give us enough control to
prove Theorem~\ref{thm:skew-inv} in the absence of
walls.

We now return to a more detailed analysis of the interaction of the action of $\Affk$ with boundaries and gates.
We begin by keeping track of corresponding factors of $f$ and $w \star f$ the same way that Definition \ref{def:gate-becomes} keeps track of corresponding gates.

\begin{Def}
\label{def:def-becomes} Suppose that
$w \star f  = g$ for some \klds $f$ and $g$ and some
$w \in \Affk$. We say that \emph{a factor
$g_b$ comes from the factor $f_a$ via $w \star f = g$}
if the permutation $\pi(w)$ of
Remark~\ref{rk:meet-the-groups-rk} takes the integer
$a$ to  $b$.	
 Alternatively, we say that the factor $f_a$
\emph{becomes} $g_b$ via $w \star f$, or that
$f_a$ and $g_b$ are \emph{corresponding factors}.
\end{Def}

A computationally useful equivalent to
Definition~\ref{def:def-becomes} may be given
inductively.

\begin{Rk}
\label{rk:def-becomes-generators}
For \klds $f$ and $g = t_i \star f$,
a factor $g_b$ comes from the
factor $f_a$ via $t_i \star f$ if
$b = a \notin \{ i, i+1 \}$ or $\{ a, b \} = \{ i, i+1 \}$.

For an element $w \in \Affk$ for which $w \star f = g$,
a factor $g_b$ come from the factor $f_a$ via
$w \star f = g$ if there are
\begin{itemize}
\item a sequences $\left( w_i
\right)_{i=1}^n$ of generators,  $\{ t_j ~:~ 1 \leq j \leq k \}$,
\item a sequence of
\klds $\left( H(i) \right)_{i = 0}^n$, and
\item  a
sequence of factors $\left( h(i) \right)_{i=0}^n$
\end{itemize}
so
that
\begin{itemize}
\item $w = w_n \cdots w_2 w_1$,
\item $H(0) = f$,
\item $H(m+1) = w_{m+1} \star H(m)$,
\item $h(0) = f_a$,
\item $h(m+1)$ is the factor of $H(m+1)$ which
comes from $h(m)$ via $w_{m+1} \star
H(m) = H(m+1)$, and
\item $h(n) = g_b$.
\end{itemize}
\end{Rk}

With the next definition we introduce
wandering and loitering quadratics, the key notions
of this section. The word
``wandering quadratic'' is used in~\cite{MS} for
a related, though quite different, notion.

\begin{Def} \label{def:loiter-def}
Fix a \kld $f$ with a non-empty clustering, and let $B$ be the boundary set of $f$.
A quadratic factor $f_a$ is a \emph{loitering quadratic of $f$} if $a - \frac{1}{2} \in B$.

A quadratic factor $f_a$ is a \emph{wandering quadratic of $f$} if $f_a$ becomes a loitering quadratic
$g_b$ via $w \star f = g$ for some $w \in \Affk$.

Let $b'$ be the boundary of $f$ corresponding to the boundary $b-\frac{1}{2}$ of $g$ via $w \star f$ in the sense of Definition~\ref{def:gate-becomes}. We say that the wandering quadratic $f_a$ \emph{owns} the gate that $f$ has at $b'$.
\end{Def}

\begin{lem}
\label{lem:unowned-gate-wall}
Let $f$ be a long decomposition with at least one quadratic factor and at least one type-$\mathsf{C}$ factor.  If some
some clustering of $f$ has an unowned gate at $a$, then $f$ has a wall at $a$.	
\end{lem}
\begin{proof}
Since $f$ has a quadratic factor, $\epsilon^{\pm 1}_2 \star f$ is not defined.  Likewise, because $f$ has a type-$\mathsf{C}$ factor,
$\epsilon_p^{\pm 1} \star f$ is not defined for odd $p$.  Because $f$ has a cluster boundary with an unowned gate at $a$,
$t_a \star f$ is not defined. (See the first paragraph of the proof of Proposition~\ref{prop:no-gate-wall} for more details.)
That is, $f$ has a wall at $a$.	
\end{proof}

\begin{Rk}
When $f_a$ is a loitering quadratic of $f$, that
cluster boundary of $f$ is ambiguous:
some clusterings of $f$ have a one-way gate at $a$ and
others at $(a-1)$, so this quadratic could join either
of the two clusters.	
\end{Rk}

We shall soon vindicate the last notion by showing that each gate is owned by at most one wandering quadratic.

\begin{lem} \label{lem:waqualem}
Fix a \kld $f$ and a word $w$ in the affine permutation group such that $w \star f = g$ is defined.
Suppose that $g_a$ is a loitering quadratic of $g$, and that it comes from $f_b$ via $w \star f$.
Then $f_i$ is not quadratic for any $i$ between $a$ and $b$, including $a$.
\end{lem}

The intuitive idea of the proof is quite simple: one quadratic is near a gate, and the other is further. As the far quadratic moves to the gate,
it will bump into the near quadratic and stop (since
tautological Ritt swaps are not allowed).
The only possible exception would be when the near
quadratic crosses the gate and stops being between
the gate and the far quadratic.
However, in that case, when the far quadratic arrives at
the gate, the gate will be one-way the wrong way and
the quadratic will not be able to enter the gate.

\begin{proof}
What we need to show that if $b > a$, $f_i$ is not quadratic for $i = b-1, \ldots a+1, a$;
 and if $b < a$, $f_i$ is not quadratic for $i = b+1, \ldots a-1, a$.
The proofs of these two statements are identical; we prove the first.

Toward contradiction, let $(f, w, b > c \geq a)$ be a counterexample with $w$ expressed as
a product of generators $w= w_r \ldots w_1$
with the least possible $r$. That is, $f_b$ and $f_c$ are quadratic, and the loitering quadratic $g_a$ of $g$ comes from $f_b$ via $w \star f$.

For $s = 0, 1, \ldots r$, let $B(s)$ and $C(s)$ be the places where $f_b$ and $f_c$ land via $w_s \ldots w_1 \star f$. So $B(0) = b$ and $C(0) = c$ and $B(r) = a$.
Two quadratic factors cannot Ritt-swap with each other and $B(0) > C(0)$, so $B(s) > C(s)$ for all $s$.
In particular, $a = B(r) > C(r)$ while $C(0) = c \geq a$. Thus, $C(s) = a$ for some $s < r$.

Let $f' := w_s \ldots w_1 \star f$, let $w' := w_r \ldots w_{s+1}$, and let $b' := B(s)$.
Now $w' \star f' = w \star f = g$ and the loitering quadratic $g_a$ of $g$ comes from $f'_{b'}$ via $w' \star f'$; and $b' = B(s) > C(s) = a$. So $(f', w', a, b')$ is another counterexample. So the presentation of $w'$ in this counterexample cannot be shorter than that of $w$, so we must have $s=0$ and $w' = w$ and $c=a$.  Thus, we have

 $$b = B(0) > C(0) = a = B(r) > C(r) \text{ .}$$

\begin{claim}\label{Claim1} The factors $f_d$ are not quadratic for any $d = b-1, \ldots a+1$.
\end{claim}
\begin{pfc}
Suppose not, and let $D(s)$ be the place where $f_d$ lands via $w_s \ldots w_1 \star f$.
So $D(0) = d$, so $B(0) > D(0) > C(0) = a = B(r) > D(r)$.
Thus, for some $s \neq 0, r$, we have $B(s) > D(s) = a$.
Again, let $f' := w_s \ldots w_1 \star f$, let $w' := w_r \ldots w_{s+1}$, and let $b' := B(s)$.
As before, $(f', w', b', a)$ is another counterexample.
But since $s \neq 0$, the presentation of $w'$ is strictly shorter than that of $w$, which is a contradiction.
\end{pfc}

\begin{claim}\label{Claim2} For all $s > 0$, we have $a > C(s)$; for all $s < r$, we have $B(s) > a$.
\end{claim}
\begin{pfc}
Since $a > C(r)$, if $C(s) \geq a$ for some $s > 0$, then $C(s') = a$ for some $s'$ with $s < s' < r$.
Again, $f' := w_{s'} \ldots w_1 \star f$ and $w' := w_r \ldots w_{s'+1}$ and $b' := B(s')$ give a counterexample, with $w'$ having a shorter
presentation than that of $w$, a contradiction.

Similarly, if $a \geq B(s)$ for some $s < r$, then $a = B(s')$ for some $s' < s$ because $B(0) > a$.
Now $(f, u, b, a)$ is another counterexample, with $u = w_{s'} \ldots w_1$ having a strictly shorter presentation than that of $w$, which is a contradiction.
\end{pfc}

Claim~\ref{Claim2} implies that $w_1 = t_{a-1}$ and $w_r = t_a$.

\begin{claim}\label{Claim3} For all $s \neq 0, r$, the factor of $h_a$ of $h := w_s \ldots w_1 \star f$ is not quadratic.
\end{claim}
\begin{pfc}
From Claim~\ref{Claim2}, we have $B(s) > a > C(s)$.
Let $f_d$ be the factor of $f$ that becomes $h_a$ via $w_s \ldots w_1 \star f$.
Then $B(0) > d > C(0)$, contradicting Claim~\ref{Claim1}.
\end{pfc}

\begin{claim} \label{Claim4} For all $s$ with $0 < s < r$,
 the \kld $h: = w_s \ldots w_1 \star f = w_{s+1} \ldots w_r \star g$ admits a clustering with a left-to-right gate at $(a-1)$.
\end{claim}

\begin{pfc}
We induct backwards on $s$, from $r-1$ to $1$. 
For the base case  
$s = r-1$, we are looking at $t_a \star g$ and $g_a$
is a loitering quadratic of $g$. There is a clustering
of $g$ with a cluster boundary at $(a-1)$ and a
left-to-right gate at that boundary. For that
clustering, $t_a \star g$ is an intra-cluster swap that
does not change the locations and directions of gates.

For the inductive step, we are looking at $w_s \star h$
for $h = w_{s+1} \ldots w_r \star g$. By the inductive hypothesis, some clustering of $h$ has a cluster boundary at $(a-1)$ and a left-to-right gate at that boundary.
By Claim~\ref{Claim3}, $h_a$ is not quadratic; so every clustering of $h$ has a cluster boundary at $(a-1)$ and a left-to-right gate at that boundary.
By Proposition~\ref{prop:one-swap-and-clusters}, this means that $t_{a-1} \star h$ is not defined, so $w_s = t_i$ for some $i \neq a$. By the same Proposition~\ref{prop:one-swap-and-clusters}, it follows that $w_s \star h$ also admits a clustering with a left-to-right gate at $(a-1)$.

Claim~\ref{Claim4} is now proved. \end{pfc}

So for $s=1$ we have that $h = t_{a-1} \star f$ has a left-to-right gate at $(a-1)$ and no quadratic at $a$. But then $t_{a-1} \star h$ is not defined, contradicting $t_{a-1} \star h = f$.

\end{proof}

\begin{cor} Each gate is owned by at most one wandering quadratic. \end{cor}

\begin{proof}
Suppose that quadratic factors $f_{a}$ and $f_{b}$ of $f$ both own the gate at $c$.
So there is some $u \in \Affk$
 and a gate at $a'$ of $g := u \star f$ corresponding to
  the gate of $f$ at $c$ such that $g_{a'}$ comes from $f_a$ via $u \star f$.
Similarly, there is some $v \in \Affk$ and a gate at $b''$ of $h := v \star f$ corresponding to the gate of $f$ at $c$ such that $h_{b''}$ comes from $f_a$ via $v \star f$.

Now for $w := v u^{-1}$ we have $w \star g = h$ and the gate of $h$ at $b''$ corresponds to the gate of $g$ at $a'$ via $w \star g$. Both gates are occupied, so Proposition~\ref{prop:asf-and-gates} tells us that $a' = b''$.

Let $b'$ be the index for which $g_{b'}$ comes from $f_b$ via $u \star f$.
Now $h_{b''}$ comes from $g_{b'}$ via $w \star g$, but $g_{b''} = g_{a'}$ is quadratic.
This contradicts Lemma~\ref{lem:waqualem} unless $a' = b'$; in which case $a =b$, and the two owners of the same gate turn out to be the same person as wanted. \end{proof}

\begin{cor} If a wandering quadratic $f_a$ owns the gate at $b$, then $|b-a| \leq k$. \end{cor}
\begin{proof} If $b > a+k$, the factor $f_{a+k}$, which is quadratic because of the $k$-periodicity of $f$, contradicts Lemma~\ref{lem:waqualem}. \end{proof}

We next show that a wandering quadratic can always
move to a gate it owns as directly as possible, that
is, via a single transit in the sense of Definition~\ref{def:transit}.

\begin{lem}
\label{lem:wqua-one-transit}	
Suppose that a wandering quadratic $f_a$ of a long
decomposition $f$  owns the gate at $b$.
Then there is a transit $u$ with $g_b$ a loitering
quadratic
coming from $f_a$ via $u \star f = g$. \end{lem}

\begin{proof}
We work out the case where $a < b$; the case where
$a  > b$ is identical; and the empty word $u$ works for
the case $a=b$.

Let $w$ be a reduced word
witnessing that $f_a$ owns the gate at $b$, shortest among such words.
That is, $w \star f = h$ is defined and $h_b$ is a loitering quadratic of $h$ that comes from $f_a$ via $w \star f$;
 and whenever the same happens for some other $w'$ and $h'$ (with the same $b$!), $w'$ is no shorter than $w$.

We first show that $t_b$ does not occur in $w$.
Suppose toward contradiction that $w = v t_b w'$ for some words $v$ and $w'$ with $w'$ not containing $t_b$.
Let $h' := w' \star f$. Now $t_b \star h'$ is defined, so $h'_b$ must be a quadratic.  Since only one
quadratic owns this gate at $b$, $h'_b$ corresponds
to $f_a$ via $w' \star f = h'$ contradicting minimality
of the length of $w$.

The subgroup of $\Affk$ generated by
$\{t_i ~:~ i \neq b \}$ is an isomorphic
copy of $\Symk$.
Thus, we may present $w$ in second canonical form with respect to $b > a+1 > a > b-k$ (see Remark~\ref{rk:canonical-form}) as $w = u s w_3 w_2 w_1$ where \begin{itemize}
 \item $w_1$ is a reduced word in $t_i$ with $b-k < i < a-1$, permuting factors in the ``right chunk'' $f_{(a, b-k)}$;
 \item $w_2$ is empty: it permutes factors in the one-factor middle chunk;
 \item $w_3$   is a reduced word in $t_i$ with $a < i < b$, permuting factors in the ``left chunk'' $f_{[b, a)}$;
 \item $s = t_{(c,a]}$  is the transit moving the wandering quadratic as we want; and
 \item $u$ is the rest of the word $v$ in the second canonical form, moving factors from the right chunk leftward.
\end{itemize}

 Since $u$ does not move the wandering quadratic, $s$
 must get it all the way to the gate at $b$ where it
 ends up in $w \star f$; so we have $c=b$.
 Now $s w_3 w_2 w_1$ already moves the wandering quadratic $f_a$ to the gate at $b$. Thus, by the minimality of the length of $w$, the word
 $u$ is empty.
 So now we have $w = s w_3 w_1$. Since
 $s w_3$ and $w_1$ act on disjoint chunks of factors,
they  commute. So, $w = w_1 s w_3$. By the minimality of the length of $w$, the word $w_1$ is empty and
$w = s w_3$.
 Using Lemma~\ref{lem:transit-twist-commute},
 we get $w = \widehat{w_3} s$,
 which, by the minimality of the
 length of $w$, gives $w =s$ as wanted.
\end{proof}

We get the following immediate corollary.

\begin{cor}
\label{cor:wqua-gate-direction}
  If a wandering quadratic $f_a$ owns the gate at $b$,
  then all gates of $f$ between $a$ and $b$ are are in
  the direction from $a$ to $b$. \end{cor}

\begin{lem}
\label{lem:one-stuck-loitering}
Fix a reduced word $u$ and long decompositions $f$ and $g$ with $u \star f = g$.
If both $f_b$ and $g_b$ are loitering quadratics,
 then $h_b$ is a loitering quadratic
 whenever $u = u_2 u_1$ as words, and $h = u_1 \star f$.
\end{lem}

\begin{proof}
We induct on the length of $u$. The base cases where the length of $u$ is $0$ or $1$ are trivial. Assuming that the length of $u$ is at least $2$, we may write $u = t_y v t_x$ for some $x, y \in \mathbb{Z}$ and some, possibly empty, word $v$.

Suppose that $u$ is the shortest counterexample; then it is a counterexample in the strongest sense possible:

\begin{itemize}
\item[$(\spadesuit)$] for every way of writing $u$ as $u_2 u_1$ with non-empty $u_1$ and $u_2$, the $b$th factor of $u_1 \star f$ is not a loitering quadratic.
\end{itemize}

Otherwise, one of $u_2$ and $u_1$ is a shorter counterexample.
In particular, $t_x$ must move the loitering quadratic $f_b$. We work out the case where $x=b$; the other case where $x = b-1$ is symmetric.

\begin{claim*}
For any words $v_1$, $v_2$ with $v = v_2 v_1$, the long decomposition $v_1 t_x \star f$ has a one-way, left-to-right, unoccupied gate at $(b-1)$.
\end{claim*}

\begin{pfc}
If the gate of $v_1 t_x \star f$ corresponding to the gate of $f$ at $b$ is occupied, we contradict $(\spadesuit)$. Before the gate can move or change direction, it must have been occupied.
\end{pfc}

So $v t_x \star f$ has a one-way, left-to-right, unoccupied gate at $(b-1)$; but $u \star f = t_y \star (v t_x \star f)$ has a loitering quadratic at $b$. It follows that $y=b$, and $u = t_b v t_b$, and
\begin{itemize}
\item[$(\clubsuit)$] the $(b+1)^\text{st}$ factor of $v t_b \star f$ is a wandering quadratic.
\end{itemize}

The Claim also implies that $t_{b-1} \star (v_1 t_x \star f)$ is not defined for any such $v_1$, and so $t_{b-1}$ does not occur in $v$.

Thus, $v$ lies in a copy of $\Symk$, so we may replace
$v$ by its first canonical form $v = s_0 s_Q w$ where
 $w$ is a word in $t_i$ with $i \neq b-1, b, b+1$,
 permuting the factors in the big chunk on the left that
 consists of all the factors except the $b^\text{th}$
 and the $(b+1)^\text{st}$;
 and $s_Q$ is a right-to-left transit move the wandering
 quadratic from its starting $(b+1)^\text{st}$
  spot; and $s_0$ is a right-to-left transit moving
  the $b^\text{th}$ factor of $t_b \star f$ left.
  But we know from $(\clubsuit)$ that the quadratic ends up back in the $(b+1)^\text{st}$ spot! So, there are only two possibilities for $s_0$ and $s_Q$.

\vspace{.1in}
\noindent
{\bf Case 1:} If $s_Q$ is empty, then $s_0$ must also be empty.
 Since $w$ commutes with $t_b$, we now have
 $$u = t_b v t_b \simeq t_b s_0 s_Q w t_b = t_b w t_b \simeq w t_b t_b$$
where $\simeq$ means ``reducible of each other'' so same length.
This contradicts $u$ being reduced.

\vspace{.1in}
\noindent
{\bf Case 2:} If $s_Q$ is non-empty, it still can only move the wandering quadratic one step, so $s_Q = t_{b+1}$. Now $s_0$ must move its factor past the quadratic, so at least two steps: $s_0 = s_1 t_{b+1} t_b$. So now
$$u = t_b v t_b \simeq t_b s_0 s_Q w t_b = t_b (s_1 t_{b+1} t_b) t_{b+1} w t_b \simeq
t_b s_1 \mathbf{ t_b t_{b+1} t_b } w t_b \simeq t_b s_1  t_b t_{b+1} \mathbf{  w t_b t_b }$$
again, contradicting  $u$ being reduced. (In the
displayed equations, the two boldface expressions
come from the braid relation $t_{b+1} t_b  t_{b+1} \simeq t_b t_{b+1} t_b$ and the commutation of $w$ with $t_b$.)

In both cases, we have obtained the desired contradiction.
\end{proof}

From Lemma~\ref{lem:one-stuck-loitering} we immediately
obtain the following useful corollary.

\begin{cor}
\label{cor:one-stuck-loitering}
Fix a reduced word $u$ and long decompositions $f$ and $g$ with $u \star f = g$.
If both $f_b$ and $g_b$ are loitering quadratics, then $t_b$ and $t_{b-1}$ do not occur in $u$.
\end{cor}

\begin{Def}
\label{def:lazy}
A long decomposition f is \emph{lazy} is all wandering quadratics of $f$ are loitering.
\end{Def}

\begin{lem}
\label{lem:stuck-loitering}
If $g$ and $u \star	g$ are both lazy and the same
gates are occupied by wandering quadratics in both,
then $u = u_r \ldots u_1$ where $u_i$ permutes factors
in the $i^\text{th}$ cluster (with respect to some
listing of the clusters) and does not move any
wandering quadratic in that cluster.
\end{lem}

\begin{proof}
Fix a reduced presentation $w$ of $u$.
By Corollary~\ref{cor:one-stuck-loitering}, whenever $g$ has an occupied gate at $b$, neither $t_b$ nor $t_{b-1}$ occur in $w$.
Furthermore, by Lemma~\ref{lem:one-stuck-loitering} and Corollary~\ref{cor:one-stuck-loitering}, whenever $g$ has an unoccupied gate at $a$, the gate remains unoccupied, so $t_a$ does not occur in $w$. Intracluster swaps in different clusters commute with each other, and so they can be collected as described. \end{proof}

\begin{prop}
\label{prop:lazy-transit}
Let $f$ be a lazy long decomposition with $r$ clusters and $p$ wandering quadratics per period,
and suppose that $w \star f$ is also lazy, for some $w \in \Affk$.
Then
$w \star f$ may be expressed as
$$w \star f = s_p \ldots s_2 s_1 u_r \ldots u_2 u_1  \star f$$
where
 \begin{itemize}
 \item $u_j$ permutes factors in the $j$th cluster
 (with respect to some listing of the clusters) and
 \item $s_i$ is a transit that moves the $i$th wandering quadratic (again with respect to some listing of the wandering quadratics) from one gate to another.
 \end{itemize}
\end{prop}

\begin{proof}
We are actually going to obtain a slightly different form $u_r \ldots u_2 u_1 s_p \ldots s_2 s_1$ where transits happen first and intracluster swaps happen second. The stated result follows from Lemma~\ref{lem:transit-twist-commute}, or by applying the proved result to $w^{-1}$.

We induct on the number of loitering quadratics of $f$ moved to a different gate by $w\star f$.
The base case where none are moved is Lemma~\ref{lem:stuck-loitering}, with all $s_j$ empty.

Now suppose that a wandering quadratic $f_a$ of $f$ is moved to $g_b$ via $w \star f =g$.
Applying Lemma~\ref{lem:wqua-one-transit} produces a transit $s$,
a word $w' := w s^{-1}$, and a decomposition $f' := s \star f$ with
the factor $f'_b$ being a loitering quadratic of $f'$, and all other loitering quadratics of $f$ unmoved.

Now $w' \star f' = g$, both $f'$ and $g$ are lazy, and one more loitering quadratic stays put in
$w' \star f' = g$ as compared to $w \star f = g$. Thus, by inductive hypothesis, we may express $w'$ as desired; and then
$w = s w'$ is also of the requisite form.

\end{proof}

\begin{prop}
\label{prop:via-lazy-transit}	
For a long decomposition $f$ with $r$ clusters and $p$ wandering quadratics per period,
and for any $w \in \Affk$,
$w \star f$ may be expressed as
$$w \star f = v_{in} s_p \ldots s_2 s_1 u_r \ldots u_2 u_1 v_{out} \star f$$
where
 \begin{itemize}
 \item $v_{out}$ moves all wandering quadratics to a nearest owned gate, via transits,
 so that $v_{out} \star f$ is lazy,
 \item $u_j$ permutes factors in the $j$th cluster
 (with respect to some listing of the clusters),
 \item $s_i$ is a transit that moves the $i$th wandering
 quadratic (again with respect to some listing of the
 wandering quadratics) from one gate to another, and
 \item $v_{in}$ moves some wandering quadratics from a gate into clusters, via transits.
 \end{itemize}
  \end{prop}

\begin{proof}
Lemma~\ref{lem:wqua-one-transit} gives the transits
making up $v_{in}$ and $v_{out}$.  Apply
Proposition~	\ref{prop:lazy-transit} to the
long decomposition
$v_{out} \star f$ and  $v_{in}^{-1} w v_{out}^{-1}
\in \Affk$ to obtain the $u_j$'s and $s_i$'s.
\end{proof}

\begin{Rk}
\begin{enumerate}
\item $s_j$ and $u_j$ can be done in any order.  See
Lemma~\ref{lem:transit-twist-commute}. 
\item All intermediate decompositions between $s_p \ldots s_2 s_1 u_r \ldots u_2 u_1 v_{out} \star f$
and $v_{out} \star f$ are lazy.
\item  In constructing the transits in $v_{out}$ and $v_{in}$, if a particular wandering quadratic $f_i$ owns the gates at both sides of the cluster it is in, we are free to choose to move this $f_i$ to either one of those gates.
\end{enumerate}
\end{Rk}

We proceed to complete the proof of Theorem~\ref{thm:skew-inv} under the
hypothesis that a long decomposition of the
polynomial $P$ admits a non-empty clustering.
With the results we have already proven, this
reduces to proving Theorem~\ref{thm:skew-inv}
under much stronger hypotheses on such a clustering.

\begin{prop}
\label{prop:one-way-Mahler}
Let $f$ be a long decomposition admitting
a non-empty clustering
with at least one $\mathsf{C}$~cluster and
only one-way gates all of which are owned by some
wandering quadratic.  If for some $\gamma \in \STpk$, we have
$\gamma \star f = f^\tau$,
then there are $\alpha$ and $\beta$ in $\Symk$ such
that $\beta \phi^N \alpha \star f = f^\tau$ encodes
the same skew-invariant curve.  Thus, that curve is
a skew-twist.
\end{prop}

\begin{proof}
Since we have at least on type $\mathsf{C}$ factor,
$\epsilon_p^{\pm 1} \star f$ is not defined for any
odd $p$.  Since gates are owns, $f$ must have at least
one quadratic factor; so, $\epsilon_2^{\pm 1} \star f$ is
also undefined.  Thus, $\gamma \in \STk$ and may be expressed
as $\gamma = \phi^N w$ for some $N \in \ZZ$ and $w \in \Affk$.
We now have $\phi^N w \star f = f^\tau$

 The following claim identifies the two cases
 we will consider.

 \begin{claim*}
  One of the following holds \begin{itemize}
 \item[$\clubsuit$] Applying Proposition~\ref{prop:via-lazy-transit} to $w \star f$, we may take $v_{out} \in \Symk$.
 \item[$\spadesuit$]  The long decomposition $f$ does not have a cluster boundary at $k$ and $f_{k+1}$ and $f_k$ are not loitering quadratics. The cluster $C_{bad}$ of $f$ containing $f_{k+1}$ and $f_k$ also contains a wandering quadratic $f_a$ which owns exactly one of the gates on the sides of this cluster.
 If $f_a$ owns the gate on the left of this cluster, $k \geq a$; and
 if  $f_a$ owns the gate on the right of this cluster, $a \geq k+1$. 
 \end{itemize}
\end{claim*}
\begin{pfc}
  The only
 reason that $\clubsuit$ would fail is that $v_{out}$ contains $t_k$.
This cannot happen if either $f_{k+1}$ or $f_k$ is a loitering quadratic of $f$, because such
a loitering quadratic is not moved by $v_{out}$.

  Consider the situation were the
  transit $s$ in $v_{out}$ containing
  $t_k$ is left-to-right; the other one is analogous.
 At some stage, this transit $s$ takes the wandering quadratic from the $(k+1)^\text{st}$
  position to the $k^\text{th}$ one. If $f$ had a cluster boundary at $k$, the wandering quadratic would already be loitering when it is in the
  $(k+1)^\text{st}$ spot, contradicting the fact that $v_{out}$ moves each wandering quadratic to a nearest gate.

 Thus, ``the cluster $C_{bad}$ of $f$ containing $f_{k+1}$ and $f_k$'' is well-defined. The wandering quadratic moved by the transit $s$ must be inside this cluster when $t_k$ moves it
  from the $(k+1)^\text{st}$ position to the $k^\text{th}$ position. So, it starts out in $C_{bad}$. Since $s$ contains $t_k$, it starts out at the $(k+1)^\text{st}$ spot or further left, and moves to the gate on the right side of $C_{bad}$.  If it also owns the gate on the left side of $C_{bad}$, it could have moved there instead, giving us a different $v_{out}$, with no $t_k$.
 \end{pfc}

We consider now the two cases identified in the claim.

\vspace{.1in}
\noindent
{\bf Case $\clubsuit$:}  Replacing $f$ and $f^\tau$ by $v_{out} \star f$ and $v_{out} \star f^\tau$, it suffices to solve the Mahler problem for a lazy decomposition. Fix a
representation of $w$ as given by  Proposition~\ref{prop:lazy-transit}.  We argue
just like the case with walls: in
 $$\phi^N s_p \ldots s_2 s_1 u_r \ldots u_2 u_1$$
 at most one $s_h$ and at most one $u_j$ contains $t_k$. We move those left via commutation to get $\phi^N u_j s_h v$ where $v \in \Symk$.

 We consider the special case where  $C_{bad}^\tau$ corresponds to $C_{bad}$ via $\phi^N w \star f = f^\tau$.
 Then for each $b$, if there is a gate of $f^\tau$ at $b$, it corresponds to a gate of $f$ at $b$ via
 $\phi^n w \star f = f^\tau$.  This implies that
 $w$ does not move any loitering quadratics, so that
 all $s_i$'s are empty.
 The following shows that the $u_i$s act trivially.

 \begin{claim*}
 Fix a cluster $f_{(b,a]}$ of some clustering of a long decomposition $f$, and a word $u$ in $\{ t_i : b > i >a \}$. If the degree sequence of $u \star f$ is the same as the degree sequence of $f$
 and positions of the monomials are also preserved, then $u$ represents the identity of $\Affk$.	
 \end{claim*}
 \begin{pfc}
 If $f_{(b,a]}$ is a $\mathsf{C}$~cluster, then the
 degrees completely determine the factors $f_j$.
 Because tautological swaps are not allowed, the
 order of factors of the same degree cannot change.

 If $f_{(b,a]}$ is a $\mathsf{C}$-free cluster, then
 the degrees completely determine the monomial factors
 and again because tautological swaps are not allowed,
 the order of monomial factors of the same degree
 cannot change.  Since every swap in a $\mathsf{C}$-free
 cluster must involve a monomial, the action of
 $u$ is trivial.	
 \end{pfc}

From now on we may assume that  $C_{bad}^\tau$
does not correspond to $C_{bad}$ via $\phi^N w \star f = f^\tau$.

\begin{claim*}
The transit $s_j$ which contains $t_k$
does not contain $t_{k-N}$. Thus,
$\widehat{s}$ defined by
$\phi^N s_j = \widehat{s} \phi^N$ does not contain $t_k$.
 \end{claim*}

 \begin{pfc}
 We work out the case that
 $s_j$ is right-to-left. So, each gate it crosses starts right-to-left and ends left-to-right. Consider the gates on the sides of the cluster $X$ of $f$ containing $f_k$, and the cluster $Y$ of $f$
 for which $X^\tau$ corresponds to $Y$ via $\phi^N w \star f = f^\tau$.
  Since $s_j$ contains $t_k$, it crosses $X$.
  Thus, $X$ has right-to-left gates on both sides in $f$.
  So $X^\tau$ also has right-to-left gates on both sides in $f^\tau$.
  If $s_j$ also contains $t_{k - N}$, then it crosses $Y$.  So, the corresponding clusters $Z'$ and $Z$ in $s_j \star f$ and in $w \star f$,
  respectively,
 have left-to-right gates at both ends.
  But $X^\tau$ corresponds to $Z$ via $\phi^N \star (w\star f) = f^\tau$, so it also must have left-to-right gates at both ends, contradicting
  that $X^\tau$ must have right-to-left gates
  at both ends.
  \end{pfc}

We now have the ingredients to complete the proof in this
case.
Let $s_h$ be the
transit containing $t_k$
and let $u_j$ be the one acting on $C_{bad}$, and let
$\widehat{u_j}$ and $\widehat{s_h}$ be defined by
$\phi^N u_j s_h = \widehat{u_j} \widehat{s_h} \phi^N$.
Because $s_h$ does not contain $t_{k-N}$,
we know that $\widehat{s_h}$ does not contain $t_k$.
Because $C_{bad}$ gets moved to a different cluster by
$\phi^N$, we know that $\widehat{u_j}$
does not contain $t_k$.
Thus, $\phi^N w = \widehat{u}_j \widehat{s}_h \phi^N v$
with $v \in \Symk$ and
$\widehat{u}_j \widehat{s}_h \in \Symk$,
and we are done with Case $\clubsuit$.

 \vspace{.1in}
 \noindent
{\bf Case $\spadesuit$}: Throughout the proof in
this case we shall refer to the wandering quadratic
$f_a$ and any factor it becomes via $v \star f$ for any
$v \in \STk$ as $Q_{bad}$.

As in Case $\clubsuit$, we may assume that all
other wandering quadratics of $f$ are loitering.
Suppose that $Q_{bad}$ owns the gate on the left of $C_{bad}$, which then must be right-to-left.
The other direction is symmetric. The gate on the right
of $C_{bad}$ is either occupied or right-to-left,
because its owner is at a gate somewhere.

We first consider the special case where this gate is owned by $f_{a-k}$.
Under this hypothesis, Corollary~\ref{cor:wqua-gate-direction} gives that all
gates in $f$ are right-to-left and all gates in
$w \star f$, and therefore in $f^\tau = \phi^N w \star f$ are left-to-right.  This is a contradiction.

So, we may assume that this gate is owned by some
wandering quadratic $f_b$ with $a > b > a-k$. We shall
refer to any wandering quadratic $f_b$ may become
as $Q_{own}$.

 Gates on both sides of $C_{bad}$ have right-to-left gates in $f$ (one might be occupied).
 Consider the cluster $Y$ of $f$ corresponding to $C_{bad}^\tau$ via $\phi^N w \star f = f^\tau$, and the cluster $Z$ of $w \star f$ corresponding to $Y$.

 The wandering quadratic $Q_{own}$ in $Z$
 corresponding to $Q_{bad}^\tau$ via $\phi^N (w\star f) = f^\tau$ does not correspond to $Q_{bad}$ via $w \star f$: because
  $Q_{bad}^\tau$ does not own the gate to the right of $C_{bad}$, but
  the wandering quadratic in $w \star f$ corresponding to $Q_{bad}$ via $w \star f$ owns the gate on its right, since $s_j$ that brought $Q_{bad}$ there is a right-to-left transit.

 Furthermore, $s_j$ cannot go all the way across $Y$, as that would leave no wandering quadratic there. So, moving $s_j$ across $\phi^N$ gets rid of $t_k$.
 Moving $u_h$ across $\phi^N$ also removes $t_k$
 because we have already ruled out the special case
 that $C_{bad}$ corresponds to $C_{bad}^\tau$ via
 $\phi^N w \star f = f^\tau$.
 Thus, $\phi^N w = \widehat{u_j} \widehat{s_h} \phi^N v$
  with $v \in \Symk$ and $\widehat{u_j} \widehat{s_h} \in \Symk$, concluding the proof of Case
  $\spadesuit$.

\end{proof}

\begin{cor}
\label{cor:thm1.3-one-way-gate}
Suppose that a polynomial $P$ admits
a long decomposition $f$ having
a non-empty clustering
with at least one $\mathsf{C}$~cluster and
only one-way gates all of which are owned by some
wandering quadratic.  Then every $(P,P^\tau)$-skew
invariant curve is a skew-twist.
\end{cor}
\begin{proof}
Consider a $(P,P^\tau)$-skew invariant curve
$\mathcal{C}$ encoded by $\gamma \star f = f^\tau$ for some
 $\gamma \in \STpk$.  By Proposition~\ref{prop:one-way-Mahler},
 $\mathcal{C}$ is a skew-twist.	
\end{proof}

\subsection{Completing the proof of Theorem~\ref{thm:skew-inv}.}
\label{sec:complete-skew-inv}
 
We now complete the proof of
Theorem~\ref{thm:skew-inv}.

\begin{proof}
We are given a non-exceptional polynomial $P$.
We need to show that every skew-invariant curve for
$(P,P^\tau)$ is a skew-twist.

We have shown with Proposition~\ref{prop:warmup-indecomposable} that Theorem~\ref{thm:skew-inv} 
holds when $P$ is indecomposable.  For the remainder
of this proof we assume that $P$ is decomposable.

Fix a long decomposition $f$ of the polynomial $P$. We break
into three cases.

\textit{Case 1: Unswappable factor}  If some factor $f_i$ of $f$ is not swappable in the sense of Definition~\ref{def:Ritt-poly}, then by Lemma~\ref{lem:unswappable-to-wall}, the long decomposition $f$ has a wall at $i$. Proposition~\ref{prop:warm-up} is Theorem~\ref{thm:skew-inv} for this $f$.

\textit{Case 2: All factors swappable, $\mathsf{C}$-free}
 Suppose that all factors of $f$ are swappable but no factor is of
 type $\mathsf{C}$.   
 Proposition~\ref{prop:c-free} breaks this into two subcases. In one subcase, the long decomposition $f$ has a wall and then Proposition~\ref{prop:warm-up} is Theorem~\ref{thm:skew-inv} for this $f$. The other subcase is exactly
Proposition~\ref{prop:thm1.3-in-out-case}.

\textit{Case 3: All factors swappable with at least one
type $\mathsf{C}$ factor}
Suppose now that all factors of $f$ are swappable and
at least one of them is type $\mathsf{C}$.

 By Remark~\ref{rk:ezrmk45}, such decompositions always admit preclusterings. If the empty set is a preclustering of such a decomposition, Proposition~\ref{prop:emptyclean} shows that $P$ is an exceptional polynomial (Chebyshev or negative Chebyshev, up to skew-conjugation), which is explicitly excluded by the hypotheses of Theorem~\ref{thm:skew-inv}. Otherwise, by
 Corollary~\ref{cor:clustering-exists},  $f$ admits a nonempty clustering $A$.

 If some cluster boundary $a \in A$ has no gate, then $f$ has a wall at $a$ by Proposition~\ref{prop:no-gate-wall}; and again Proposition~\ref{prop:warm-up} is Theorem~\ref{thm:skew-inv} for this $f$. By Definition~\ref{def:fakedef}, no gates of $(A, f)$ are two-way. Thus, we may assume that all cluster boundaries of $(A, f)$ have one-way gates.

 If some gate of $(A,f)$ is not owned by any wandering quadratic and $f$ has a quadratic factor, Lemma~\ref{lem:unowned-gate-wall} shows that $f$ has a wall, and again Proposition~\ref{prop:warm-up} is Theorem~\ref{thm:skew-inv} for this $f$.

 It remains to prove Theorem~\ref{thm:skew-inv} for long decompositions $f$ with all of the following properties: \begin{itemize}
 \item all factors of $f$ are swappable;
 \item at least one factor of $f$ is type $\mathsf{C}$;
 \item $f$ admits a non-empty clustering $A$;
 \item all cluster boundaries of $(A, f)$ have one-way gates; and
 \item every gates of $(A,f)$ is owned by some wandering quadratic.
\end{itemize}
 This is done in Corollary~\ref{cor:thm1.3-one-way-gate}.

\end{proof}

\section{Independence for non-exceptional polynomial type}
\label{non-exceptional-sect}

In this section we complete the proof of Theorem~\ref{thm:ind-mahler}.
As we noted in the introduction, this is achieved by showing a much stronger algebraic
independence result for Mahler functions of non-exceptional polynomial type to the
effect that if $f$ is a transcendental $q$-Mahler function of non-exceptional polynomial type,
$p \in \QQ_+$ is multiplicatively independent from $q$, and $g_1, \ldots, g_m$
is a sequence of functions each satisfying some algebraic difference equation with respect to
the substitution $t \mapsto t^p$, then $f$ is algebraically independent from
$g_1, \ldots, g_m$ over $\CC(t)$.

We prove this theorem as an instance of a general 
result about solutions to difference equations.  
With the following Convention we outline the conditions
required for our general theorem. 

\begin{conv}
\label{axconv}
Throughout the remainder of this section
we work with a quintuple $(C,K,L,\sigma,\tau)$ satisfying the following.

\begin{enumerate}
\item $C \subseteq K \subseteq L$ is a tower of algebraically closed fields of characteristic zero.
\item $\sigma$ and $\tau$ are commuting automorphisms of $L$.
\item $\sigma(K) = K$ and $\tau(K) = K$.
\item $C = \operatorname{Fix}(\sigma) = \operatorname{Fix}(\tau)$ is the common fixed field of $\sigma$ and $\tau$.
\item $\sigma$ and $\tau$ are independent in the sense that for $a \in L \smallsetminus C$ and
$(m,n) \in \ZZ^2 \smallsetminus \{ (0,0) \}$ one has $\sigma^m \tau^n (a) \neq a$.
\item For $\rho$ any element of the group generated by $\sigma$ and
$\tau$ and $c \in C \smallsetminus \{ 0, 1 \}$ a nonzero constant distinct from one,
there are no nonzero solutions to $\rho(y) = c y$.
\item For $\rho$ and $\mu$ independent elements of $\langle \sigma, \tau \rangle$, the group of
automorphisms generated by $\sigma$ and $\tau$, $M \in \ZZ_+$ a positive integer, and $a \in L^\times$, if
$\rho( \frac{\mu(a)}{a} ) = (\frac{\mu(a)}{a})^M$, then $a \in K^\times$.
\end{enumerate}
\end{conv}

We note now that our intended structures fit Convention~\ref{axconv}.

\begin{prop}
Let $p$ and $q$ be multiplicatively independent positive rational numbers.   Let $C := {\mathbb C}$,
$K := \CC (t)^\text{alg}$, and $L := \CC((t))^\text{alg}$. Let $\sigma:L \to L$ be
defined by $\sum a_r t^r \mapsto \sum a_r t^{rp}$ and $\tau:L \to L$ be defined by
$\sum a_r t^r \mapsto \sum a_r t^{rq}$.    Then $(C,K,L,\sigma,\tau)$ meets the conditions of
Convention~\ref{axconv}.
\end{prop}
\begin{proof}
Conditions (1) and (2) are obvious.  For condition (3), note that $\sigma$ and $\tau$ preserve the
subfield $\CC( \{ t^r ~:~ r \in \QQ \})$ of $K$, and hence, its algebraic closure, $K$, itself.   For
condition (4), we express $a \in L\setminus C$ as $a = \sum a_r t^r$
and let $s\neq 0$ be minimal such that $a_s\neq 0$. Then
among the terms in $\sigma(a)$ with a non-zero coefficient,
$t^{ps}$ is the non-constant term of minimal degree.
Since $ps\neq s$, we have $\sigma(a)\neq a$.
Applying the same reasoning to $\tau$ and $\sigma^m \tau^n$ gives conditions (4) and (5).

For condition (6), the automorphism $\rho$ is given by the substitution
$t \mapsto t^{ap + bq}$ for some $(a,b) \in \ZZ^2$.  Suppose that
$\rho(y) = c y$ and $y$ is not zero.   If $(a,b) = (0,0)$,
then $\rho(y) = y$ which is not $cy$ unless $y = 0$.  So, we
may assume that $(a,b) \neq (0,0)$.  Express $y = y_- + y_0 + y_+$
where $y_0 \in C$, $y_-$ is a Puisseux series supported with only
negative exponents, and $y_+$ is a Puisseux series supported with only
positive exponents.  From the additivity of the equation and the
preservation of the three cases (exponent zero, all exponents negative,
or all exponents positive), we see that we must have $\rho(y_*) = c y_*$
for $* \in \{ 0, -, + \}$.  Since $\rho$ is the identity function on
$C$ and $c \neq 1$, we must have $y_0 = 0$.  Since the order of
$\rho(y_*)$ is  $(ap + bq)$ times the order of $y_*$ and $ap + bq \neq 1$,
unless $y_* = 0$ (for $* = + \text{ or } -$), the equality $\rho(y) = cy$
cannot hold.

For
condition (7), notationally it suffices to consider $\mu = \sigma$ and
$\rho = \tau$.  Write our solution $a$ as $a = c t^r (1 + \epsilon)$ where
$c \in \CC^\times$, $r \in \QQ$, and $\epsilon \in \CC((t))^\text{alg}$ has
strictly positive order.  We have $\tau (\frac{\sigma(a)}{a}) = t^{q(p-1)r} \tau(\frac{\sigma(1+\epsilon)}{1+\epsilon}))$,
while $(\frac{\sigma(a)}{a})^M = t^{(p-1)Mr} (\frac{\sigma(1 + \epsilon)}{1 + \epsilon})^M$.  Thus, for $a$ to be
a solution we must
have $qr = Mr$ (so that $r =0$ or $q = M$) and $\tau(\frac{\sigma(1 + \epsilon)}{1 + \epsilon}) =
(\frac{\sigma(1+\epsilon)}{1+\epsilon})^M$.   If $\epsilon  =0 $, then $a = c t^r \in K$, as required.
If $\epsilon \neq 0$, then we may express $\epsilon = d t^s + \eta$ where $d \in \CC^\times$ and the order of
$\eta$ is strictly greater than $s$.  Depending on whether $p < 1$ or $p > 1$, we have $\frac{\sigma(1 + \epsilon)}{1 + \epsilon} = 1 + d t^{sp} + \text{ higher order}$ (or $1 - d t^s + \text{ higher order}$, respectively).
Thus, $\tau ( \frac{ \sigma(1 + \epsilon)}{1 + \epsilon} ) = 1 + d t^{spq} + \text{ higher order}$
(or, $1 - d t^{sq} + \text{ higher order}$, respectively) and $(\frac{\sigma(1 + \epsilon)}{1 + \epsilon})^M =
1 + M d t^{sp} + \text{ higher order}$ (or $1 - Md t^s + \text{ higher order}$, respectively).  In either case,
because $p$ and $q$ are multiplicatively independent, we cannot have such identities.  Thus, $\epsilon = 0$ and
$a \in K$.
\end{proof} 

Let us reformulate Theorem~\ref{thm:abs-Mah} in the language of Convention~\ref{axconv}.

\begin{thm}
\label{axiomaticthm}
Let $(C,K,L,\sigma,\tau)$ be a quintuple satisfying Convention~\ref{axconv}.
Suppose that $f \in L$, $g_1, \ldots, g_n \in L$ is a sequence of elements of
$L$, $P \in K[X]$ is a non-exceptional polynomial over $K$, and that $f$ satisfies
$\sigma(f) = P(f)$ while each $g_i$ satisfies a nontrivial $\tau$-algebraic difference
equation over $K$.  Then $f$ is algebraically independent from $g_1, \ldots, g_n$ over $K$. 
\end{thm}

We will deduce Theorem~\ref{axiomaticthm} from the following theorem about a single solution to a
difference equation of non-exceptional polynomial type.

\begin{thm}
\label{neti}
Let $(C,K,L,\sigma,\tau)$ be a quintuple meeting Convention~\ref{axconv}, $P \in K[X]$ a non-exceptional
polynomial, and $f \in L$ a solution to $\sigma(f) = P(f)$.  Then $f$ and $\tau(f)$ are algebraically
independent over $K$.	
\end{thm}

\begin{remark}
It follows from the main theorem of~\cite{MS} that Theorem~\ref{neti} may be strengthened to the
conclusion that $\{ \tau^j(f) ~:~ j \in \ZZ \}$ are algebraically independent over $K$.   Indeed, in
Theorem~\ref{axiomaticthm} we could strengthen the conclusion to the conclusion that
$\{ \tau^j (f) ~:~ j \in \ZZ \}$ is algebraically independent over $K ( \{ \sigma^i \tau^j (g_k) ~:~
i, j \in \ZZ, 1 \leq k \leq n \} )$. 	
\end{remark}

\begin{lem}
\label{reductiontoneti}
Theorem~\ref{neti} implies Theorem~\ref{axiomaticthm}.	
\end{lem}
\begin{proof}
Consider $(C,K,L,\sigma,\tau)$, $f$, $P$, $g_1, \ldots, g_n$ as in the hypotheses of Theorem~\ref{axiomaticthm}.
We may take $f \notin K$ as the the theorem holds trivially when $f \in K$.
If $f$ is algebraically dependent on $g_1, \ldots, g_n$ over $K$, then $f \in K(g_1, \ldots,g_m)^\text{alg}$.
Applying $\tau$ repeatedly, we see that for each $j \in \NN$ that
$\tau^j(f) \in K(\tau^j(g_1), \ldots, \tau^j(g_m))^\text{alg}$.  Thus,
$K \langle f \rangle_\tau := K( \{ \tau^j (f) ~:~ j \in \NN \}) \subseteq K \langle g_1, \ldots, g_m \rangle_\tau^\text{alg}$.
By our hypothesis on the $g_j$'s, $\operatorname{tr.deg}_K (K \langle g_1, \ldots, g_m \rangle_\tau) < \infty$.
Thus, there must be a nontrivial algebraic dependence over $K$ amongst $\{ \tau^j (f) ~:~ j \in \NN \}$.
Take $N$ so that $\{ \tau^j(f) ~:~ 0 \leq j \leq N \}$ are algebraically dependent over $K$.
Because $\sigma$ and $\tau$ commute, we see that $\sigma(\tau^j(f)) = \tau^j (\sigma(f)) = \tau^j (P(f))
= P^{\tau^j} (\tau^j(f))$.  That is, $\tau^j(f)$ satisfies the difference equation $\sigma(y) = P^{\tau^j}(y)$.
Since $P$ is non-exceptional, so is each ${}^{\tau^j} P$.   The variety $V$ given as the
locus of $(f, \tau(f), \ldots, \tau^N(f))$ over $K$
is then $\sigma$-skew-invariant variety for $(P,P^\tau, \ldots, P^{\tau^N})$ which projects
dominantly to each coordinate.  By Proposition 2.21 of~\cite{MS}, for some pair $i < j$ the projection of $V$ to
$i^\text{th}$ and $j^\text{th}$ coordinates is a $\sigma$-skew-invariant curve for $(P^{\tau^i}, P^{\tau^j})$.
Applying $\tau^{-i}$, we obtain a nontrivial algebraic dependence between $f$ and $\tau^{j-i} (f)$.   This
contradicts Theorem~\ref{neti} in the case of the quintuple $(C,K,L,\sigma,\tau^{j-i})$, $f$, and $P$.
\end{proof}

The rest of this section is devoted to the proof of Theorem~\ref{neti}.

By moving to another quintuple satisfying Convention~\ref{axconv}, we may arrange that from a counterexample to
Theorem~\ref{neti} that we may produce one in which $f$ and $\tau(f)$ are related by a linear relation.

\begin{lem}
\label{failure-to-linear}
If Theorem~\ref{neti} fails, then there is some quintuple $(C,K,L,\sigma,\tau)$ satisfying
Convention~\ref{axconv}, a non-exceptional $P \in K[X]$, a linear polynomial $\lambda \in K[X]$, and
$f \in L \smallsetminus K$ satisfying $\sigma(f) = P(f)$ and $\tau(f) = \lambda(f)$.	
\end{lem}

\begin{proof}
Consider a possible counterexample to Theorem~\ref{neti}.  So we have a quintuple
$(C,K,L,\sigma,\tau)$ satisfying Convention~\ref{axconv}, a non-exceptional polynomial
$P \in K[X]$ and some $f \in L \smallsetminus K$ with $\sigma(f) = P(f)$ and
$f$ and $\tau(f)$ algebraically dependent over $K$.  Throughout this proof, we will successively replace each
part of these data until we arrive at a situation as described in the conclusion of the Lemma.

 Let
$Y = \operatorname{loc}(f,\tau(f)/K)$ be the locus of $\left( f,\tau(f) \right)$ over $K$.
Then $Y$ is $\sigma$-skew-invariant for $(P,P^\tau)$.  By Theorem~\ref{thm:skew-inv}, $Y$ is a skew-twist.
 Thus, there are are $\alpha$, $\beta$ and $n$ so that either  $P = \beta \circ \alpha$, 
 $P^\tau = \alpha^{\sigma^n} \circ \beta^{\sigma^n}$, and
 $Y$ is defined by $y = \alpha^{\sigma^n}	 \circ P^{\sigma^{n-1}} \circ \cdots  P^\sigma  \circ P$ or 
 $P^\tau = \beta \circ \alpha$, 
 $P = \alpha^{\sigma^n} \circ \beta^{\sigma^n}$, 
 and
 $Y$ is defined by $x = \alpha^{\sigma^n}	 \circ P^{\tau \sigma^{n-1}} \circ \cdots  P^{\tau \sigma}  \circ P^\tau(y)$.
 In the former case, we make no changes yet.  In the latter case, we
replace $\tau$ by $\tau^{-1}$, $P$ by $P^\tau$, and $f$ by $\tau(f)$.   In so doing, we may assume that
$\sigma(f) = P(f)$ and that we are in the first case.

Let $\pi := \alpha^{\sigma^n} \circ P^{\sigma^{n-1}} \circ \cdots \circ P^\sigma \circ P$.  That is,
$Y$ is defined by $y = \pi(x)$ and we have the compositional identity $\pi^\sigma \circ P = P^\tau \circ \pi$.
Note, in particular, this means that $\pi(f) = \tau(f)$.

Let $k$ be the number of indecomposable factors in a complete docomposition of $P$.   The curve $Z$ defined by
$y = \pi^{\tau^{k-1}} \circ \cdots \pi^\tau \circ \pi (x)$ is $\sigma$-skew-invariant for $(P,P^{\tau^{k-1}})$.
Hence, by Theorem~\ref{thm:skew-inv}  there is some polynomial $\lambda$ which is a proper initial compositional
factor of $P$ and a natural number $m$ so that
$\pi^{\tau^{k-1}} \circ \cdots \pi^\tau \circ \pi = \lambda^{\sigma^m} \circ P^{\sigma^{m-1}}
\circ \cdots \circ P^\sigma \circ P$.   The number of indecomposable compositional factors of the
lefthand side of this equation is $k$ times the number of compositional factors of $\pi$ while on the righthand
side it is $km$ plus the number of factors of $\lambda$.  As $\lambda$ is a proper initial factor of $P$, it has fewer than $k$
indecomposable factors.   As $k$ divides the lefthand side, and
thus also the righthand side, namely, $km + \text{ number of
indecomposable factors of } \lambda$,  it must be that $\lambda$ has
no indecomposable factors.  That is, $\lambda$ is linear.

We compute that $\lambda(\sigma^m(f)) = \lambda ( P^{\sigma^{m-1}} \circ \cdots \circ P^\sigma \circ P (f)) =
\pi^{\tau^{k-1}} \circ \cdots \circ \pi^\tau \circ \pi (f) = \pi^{\tau^{k-1}} \circ
\cdots \circ \pi^\tau (\tau(f)) = \cdots = \tau^k(f) = \sigma^{-m} \tau^k (\sigma^m (f))$.  Thus,
if we  replace $\tau$ with $\sigma^{-m} \tau^k$, $f$ with $\sigma^m(f)$ and $P$ with $\sigma^m(P)$, then
we have our desired solution.

\end{proof}

With the next lemma we show that the situation described in the conclusion of Lemma~\ref{failure-to-linear}
cannot occur.

\begin{lem}
\label{nlin}
If $(C,K,L,\sigma,\tau)$ satisfies Convention~\ref{axconv}, $f \in L \smallsetminus K$, $P \in K[X]$ is
non-exceptional, $\lambda \in K[X]$ is linear, then it is not possible to have $\sigma(f) = P(f)$ and
$\tau(f) = \lambda(f)$.	
\end{lem}
\begin{proof}

Using the fact that $\tau$ and $\sigma$ commute, we have
$P^\tau \circ \lambda(f) = \tau (P(f))  = \tau(\sigma(f)) = \sigma(\tau(f)) = \sigma( \lambda(f)) = \lambda^\sigma \circ P(f)$.  Since
$f \in L \smallsetminus K$, we have the functional identity

\begin{equation}
\label{funceq}	
 \lambda^\sigma \circ P  = P^\tau \circ \lambda \text{ .}
\end{equation}

We make two reductions before completing this proof.

Write $P(X) = \sum_{i=0}^d p_i X^i$ where $d = \deg(P)$.

\begin{claim}
\label{normalform}
We may assume that $p_{d-1} = 0$.  	
\end{claim}
\begin{pfc}
Let $\mu(X) := X - \frac{p_{d-1}}{d p_d}$.  Set $\widetilde{f} := \mu(f)$,
$\widetilde{P} := (\mu^\sigma) \circ P \circ \mu^{-1}$, and $\widetilde{\lambda} := \mu^\tau \circ \lambda \circ \mu^{-1}$.
Then $\sigma(\widetilde{f}) = \widetilde{P}(\widetilde{f})$,
$\tau(\widetilde{f}) = \widetilde{\lambda}(\widetilde{f})$,
$\widetilde{f} \in L \smallsetminus K$, and the coefficient of $X^{d-1}$
in $\widetilde{P}$ is zero.  Replacing $f$ by $\widetilde{f}$,
$P$ by $\widetilde{P}$, and $\lambda$ by $\widetilde{\lambda	}$, the
claim is established.
\end{pfc}

We continue with the reductions.

\begin{claim}
\label{constantco}
We may assume that $p_d \in C^\times$.	
\end{claim}
\begin{pfc}
From Equation~\ref{funceq} and Claim~\ref{normalform}, we see that
there is some $A \in K^\times$ with $\lambda(X) = A X$.  Thus, reading
Equation~\ref{funceq} degree-by-degree for any
index $i$, we have

\begin{equation}
\label{coeffeq}
A^\sigma p_i = A^i p_i^\tau	
\end{equation}

Since $d = \deg(P)$, $p_d \neq 0$.   Since $P$ is not conjugate to
a monomial, there is some positive $m$ with $p_{d-m} \neq 0$.
Dividing the instances of Equation~\ref{coeffeq} for
 $i = d$ and $i = d - m$, we obtain

\begin{equation}
\label{meq}
p_d/p_{d-m} = A^m \tau(p_d/p_{d-m})
\end{equation}

Take $q \in L$ with $q^m$, then from Equation~\ref{meq},
we see that
\begin{equation}
\label{Aq}	
A = \zeta q/\tau(q)
\end{equation}

for some $m^\text{th}$ root
of unity $\zeta$.  Combining Equation~\ref{coeffeq} in the case of
$i = d$ and Equation~\ref{Aq}, we obtain

\begin{equation}
p_d^\tau \zeta^d q^d / \tau(q^d) = \zeta \sigma(q) p_d/ \sigma \tau (q)	
\end{equation}

which implies

\begin{equation}
\tau ( \frac{ p_d \sigma(q)}{q^d} ) = \zeta^{1-d} \frac{p_d \sigma(q)}{q^d}
\end{equation}

Since $(C,K,L,\sigma,\tau)$ satisfies Convwntion~\ref{axconv},
the only solutions to the difference equation $\tau(y) = \zeta^{1-d} y$
are constant.   Thus, there is some $c \in C$ with $p_d =
c \frac{q^d}{\sigma(q)}$.     Let $\mu(X) = \frac{1}{q} X$,
then replacing $P$ with $\mu^\sigma \circ P \circ \mu^{-1}$, $f$ with
$\mu(f)$, and $\lambda$ with $\mu^\tau \circ \lambda \circ \mu^{-1}$,
we arrange that the leading coefficient of $P$ may be taken to be $c$.
\end{pfc}

With our reductions, because $p_d^\tau = p_d$, from Equation~\ref{coeffeq}
in the case of $i = d$, we conclude that $A^\sigma = A^d$.  Because
$\tau(f) = A f$, we see that $\sigma(\frac{\tau(f)}{f}) = (\frac{\tau(f)}{f})^d$, which implies that $f \in K$ by Convention~\ref{axconv} contrary to our hypotheses.
	
\end{proof}


\begin{thebibliography}{10}

\bibitem{AB}
{\sc B.~Adamczewski and J.~P. Bell}, {\em A problem about {M}ahler functions},
  Ann. Sc. Norm. Super. Pisa Cl. Sci. (5), 17 (2017), pp.~1301--1355.

\bibitem{ADHW}
{\sc B.~Adamczewski, T.~Dreyfus, C.~Hardouin, and M.~Wibmer}, {\em Algebraic
  independence and linear difference equations}.
\newblock \url{arXiv:2010.09266}, 2021.

\bibitem{BoOt}
{\sc R.~V. Book and F.~Otto}, {\em String-rewriting systems}, Texts and
  Monographs in Computer Science, Springer-Verlag, New York, 1993.

\bibitem{Brown}
{\sc K.~S. Brown}, {\em Buildings}, Springer Monographs in Mathematics,
  Springer-Verlag, New York, 1998.
\newblock Reprint of the 1989 original.

\bibitem{Exel}
{\sc R.~Exel}, {\em Partial actions of groups and actions of inverse
  semigroups}, Proc. Amer. Math. Soc., 126 (1998), pp.~3481--3494.

\bibitem{Mahler-first}
{\sc K.~Mahler}, {\em Arithmetische {E}igenschaften der {L}\"{o}sungen einer
  {K}lasse von {F}unktionalgleichungen}, Math. Ann., 101 (1929), pp.~342--366.

\bibitem{Matsumoto}
{\sc H.~Matsumoto}, {\em G\'{e}n\'{e}rateurs et relations des groupes de {W}eyl
  g\'{e}n\'{e}ralis\'{e}s}, C. R. Acad. Sci. Paris, 258 (1964), pp.~3419--3422.

\bibitem{MS}
{\sc A.~Medvedev and T.~Scanlon}, {\em Invariant varieties for polynomial
  dynamical systems}, Ann. of Math. (2), 179 (2014), pp.~81--177.

\bibitem{Ng15}
{\sc K.~D. Nguyen}, {\em Algebraic independence of local conjugacies and
  related questions in polynomial dynamics}, Proc. Amer. Math. Soc., 143
  (2015), pp.~1491--1499.

\bibitem{Nishioka-book}
{\sc K.~Nishioka}, {\em Mahler functions and transcendence}, vol.~1631 of
  Lecture Notes in Mathematics, Springer-Verlag, Berlin, 1996.

\bibitem{Pa17}
{\sc F.~Pakovich}, {\em Polynomial semiconjugacies, decompositions of
  iterations, and invariant curves}, Ann. Sc. Norm. Super. Pisa Cl. Sci. (5),
  17 (2017), pp.~1417--1446.

\bibitem{Pa19}
\leavevmode\vrule height 2pt depth -1.6pt width 23pt, {\em Invariant curves for
  endomorphisms of $\mathbb{P}^1 \times \mathbb{P}^1$}.
\newblock \url{arXiv:1904.10952}, 2019.

\bibitem{Ritt}
{\sc J.~F. Ritt}, {\em Prime and composite polynomials}, Trans. Amer. Math.
  Soc., 23 (1922), pp.~51--66.

\bibitem{ScSi}
{\sc R.~Sch\"{a}fke and M.~Singer}, {\em Consistent systems of linear
  differential and difference equations}, J. Eur. Math. Soc. (JEMS), 21 (2019),
  pp.~2751--2792.

\bibitem{Za99}
{\sc U.~Zannier}, {\em On a functional equation relating a {L}aurent series
  {$f(x)$} to {$f(x^m)$}}, Aequationes Math., 55 (1998), pp.~15--43.

\end{thebibliography}
\end{document}